\documentclass[letterpaper,11pt]{article}
\usepackage[margin=1in]{geometry}  

\usepackage{ifthen}
\newboolean{arxiv}
\setboolean{arxiv}{TRUE}
\usepackage{bbm}
\usepackage{graphicx}
\usepackage{amsmath,amssymb,amsthm,amsfonts}

\usepackage{paralist}
\usepackage{bm}
\usepackage{xspace}
\usepackage{url}
\usepackage{prettyref}
\usepackage{boxedminipage}
\usepackage{wrapfig}
\usepackage{color}
\usepackage{xspace}

\usepackage{amsmath,amsthm,amsfonts,amssymb}
\usepackage{graphicx}

\usepackage{nicefrac}

\newtheorem*{definition*}{Definition}

\usepackage{subcaption}

\usepackage[utf8]{inputenc}

\usepackage{xcolor}
\definecolor{expert}{HTML}{008000}
\definecolor{error}{HTML}{f96565}

\usepackage{color-edits}
\addauthor{sw}{blue}

\usepackage{thmtools}
\usepackage{thm-restate}

\usepackage{tikz}
\usetikzlibrary{arrows,calc} 
\newcommand{\tikzAngleOfLine}{\tikz@AngleOfLine}
\def\tikz@AngleOfLine(#1)(#2)#3{%
\pgfmathanglebetweenpoints{%
\pgfpointanchor{#1}{center}}{%
\pgfpointanchor{#2}{center}}
\pgfmathsetmacro{#3}{\pgfmathresult}%
}

\declaretheoremstyle[
    headfont=\normalfont\bfseries, 
    bodyfont = \normalfont\itshape]{mystyle} 

\usepackage{listings}
\usepackage{amsmath}
\usepackage{amsthm}
\usepackage{tikz}
\usepackage{caption}
\usepackage{mdwmath}
\usepackage{multirow}
\usepackage{mdwtab}
\usepackage{eqparbox}
\usepackage{multicol}
\usepackage{amsfonts}
\usepackage{tikz}
\usepackage{multirow,bigstrut,threeparttable}
\usepackage{amsthm}
\usepackage{bbm}
\usepackage{epstopdf}
\usepackage{mdwmath}
\usepackage{mdwtab}
\usepackage{eqparbox}
\usetikzlibrary{topaths,calc}
\usepackage{latexsym}
\ifthenelse{\boolean{arxiv}}{\usepackage{cite}}{}
\usepackage{amssymb}
\usepackage{bm}
\usepackage{amssymb}
\usepackage{graphicx}
\usepackage{mathrsfs}
\usepackage{epsfig}
\usepackage{psfrag}
\ifthenelse{\boolean{arxiv}}{\usepackage{setspace}}{}
\ifthenelse{\boolean{arxiv}}{\usepackage[
            CJKbookmarks=true,
            bookmarksnumbered=true,
            bookmarksopen=true,
            colorlinks=true,
            citecolor=red,
            linkcolor=blue,
            anchorcolor=red,
            urlcolor=blue
            ]{hyperref}}{}

\usepackage{comment}
\usepackage{mathtools}
\usepackage{blkarray}
\usepackage{multirow,bigdelim,dcolumn,booktabs}

\usepackage{xparse}
\usepackage{tikz}
\usetikzlibrary{calc}
\usetikzlibrary{decorations.pathreplacing,matrix,positioning}

\usepackage[T1]{fontenc}
\usepackage[utf8]{inputenc}
\usepackage{mathtools}
\usepackage{blkarray, bigstrut}
\usepackage{gauss}

\newcommand*{\BraceAmplitude}{0.4em}%
\newcommand*{\VerticalOffset}{0.5ex}%
\newcommand*{\HorizontalOffset}{0.0em}%
\newcommand*{\blocktextwid}{3.0cm}%
\NewDocumentCommand{\InsertLeftBrace}{%
	O{} 
	O{\HorizontalOffset,\VerticalOffset} 
	O{\blocktextwid} 
	m   
	m   
	m   
}{%
	\begin{tikzpicture}[overlay,remember picture]
	\coordinate (Brace Top)    at ($(#4.north) + (#2)$);
	\coordinate (Brace Bottom) at ($(#5.south) + (#2)$);
	\draw [decoration={brace, amplitude=\BraceAmplitude}, decorate, thick, draw=black, #1]
	(Brace Bottom) -- (Brace Top) 
	node [pos=0.5, anchor=east, align=left, text width=#3, color=black, xshift=\BraceAmplitude] {#6};
	\end{tikzpicture}%
}%
\NewDocumentCommand{\InsertRightBrace}{%
	O{} 
	O{\HorizontalOffset,\VerticalOffset} 
	O{\blocktextwid} 
	m   
	m   
	m   
}{%
	\begin{tikzpicture}[overlay,remember picture]
	\coordinate (Brace Top)    at ($(#4.north) + (#2)$);
	\coordinate (Brace Bottom) at ($(#5.south) + (#2)$);
	\draw [decoration={brace, amplitude=\BraceAmplitude}, decorate, thick, draw=black, #1]
	(Brace Top) -- (Brace Bottom) 
	node [pos=0.5, anchor=west, align=left, text width=#3, color=black, xshift=\BraceAmplitude] {#6};
	\end{tikzpicture}%
}%
\NewDocumentCommand{\InsertTopBrace}{%
	O{} 
	O{\HorizontalOffset,\VerticalOffset} 
	O{\blocktextwid} 
	m   
	m   
	m   
}{%
	\begin{tikzpicture}[overlay,remember picture]
	\coordinate (Brace Top)    at ($(#4.west) + (#2)$);
	\coordinate (Brace Bottom) at ($(#5.east) + (#2)$);
	\draw [decoration={brace, amplitude=\BraceAmplitude}, decorate, thick, draw=black, #1]
	(Brace Top) -- (Brace Bottom) 
	node [pos=0.5, anchor=south, align=left, text width=#3, color=black, xshift=\BraceAmplitude] {#6};
	\end{tikzpicture}%
}%

\usetikzlibrary{patterns}

\definecolor{cof}{RGB}{219,144,71}
\definecolor{pur}{RGB}{186,146,162}
\definecolor{greeo}{RGB}{91,173,69}
\definecolor{greet}{RGB}{52,111,72}

\theoremstyle{plain}
\newtheorem{thm}{Theorem}[section]

\newtheorem{lemma}[thm]{Lemma}
\newtheorem{remark}[thm]{Remark}
\newtheorem{corollary}[thm]{Corollary}
\newtheorem{definition}[thm]{Definition}

\def \bP {\mathbb{P}}
\def \bQ {\mathbb{Q}}
\def \bC {\mathbb{C}}
\def \bE {\mathbb{E}}
\def \bR {\mathbb{R}}

\def \var {\mathrm{Var}}
\def\1{\mathbbm{1}}

\usepackage{xspace}

\newcommand{\stepa}[1]{\overset{\rm (a)}{#1}}
\newcommand{\stepb}[1]{\overset{\rm (b)}{#1}}
\newcommand{\stepc}[1]{\overset{\rm (c)}{#1}}
\newcommand{\stepd}[1]{\overset{\rm (d)}{#1}}
\newcommand{\stepe}[1]{\overset{\rm (e)}{#1}}
\newcommand{\stepf}[1]{\overset{\rm (f)}{#1}}

\newcommand{\naturals}{\mathbb{N}}

\newcommand{\TV}{{\sf TV}}
\newcommand{\LC}{{\sf LC}}

\newcommand{\KL}{{\sf KL}}

\newcommand{\pth}[1]{\left( #1 \right)}
\newcommand{\qth}[1]{\left[ #1 \right]}
\newcommand{\sth}[1]{\left\{ #1 \right\}}
\newcommand{\bpth}[1]{\Bigg( #1 \Bigg)}
\newcommand{\bqth}[1]{\Bigg[ #1 \Bigg]}

\newcommand{\Bern}{\text{\rm Bern}}
\newcommand{\Poi}{\text{\rm Poi}}
\newcommand{\Unif}{\text{\rm Unif}}

\newcommand{\indc}[1]{{\mathbf{1}_{\left\{{#1}\right\}}}}

\definecolor{myblue}{rgb}{.8, .8, 1}
\definecolor{mathblue}{rgb}{0.2472, 0.24, 0.6} 
\definecolor{mathred}{rgb}{0.6, 0.24, 0.442893}
\definecolor{mathyellow}{rgb}{0.6, 0.547014, 0.24}

\newcommand{\sfC}{{\mathsf{C}}}
\newcommand{\sfD}{{\mathsf{D}}}

\newcommand{\sfH}{{\mathsf{H}}}

\newcommand{\sfK}{{\mathsf{K}}}

\newcommand{\calC}{{\mathcal{C}}}
\newcommand{\calD}{{\mathcal{D}}}

\newcommand{\calH}{{\mathcal{H}}}

\newcommand{\calN}{{\mathcal{N}}}
\newcommand{\calO}{{\mathcal{O}}}
\newcommand{\calP}{{\mathcal{P}}}

\newcommand{\calX}{{\mathcal{X}}}

\newcommand{\rmd}{\mathrm{d}}
\newcommand{\Perm}{\mathrm{Perm}}
\newcommand{\icom}{\mathrm{i}}
\newcommand{\trace}{\mathrm{Tr}}
\newcommand{\calCN}{\calC\calN}
\newcommand{\Law}{\mathrm{Law}}

\usepackage{cleveref}
\crefname{lemma}{Lemma}{Lemmas}
\Crefname{lemma}{Lemma}{Lemmas}
\crefname{thm}{Theorem}{Theorems}
\Crefname{thm}{Theorem}{Theorems}
\Crefname{assumption}{Assumption}{Assumptions}
\crefformat{equation}{(#2#1#3)}

\newcommand{\Capa}{\sfC_{\chi^2}(\calP)}
\newcommand{\Diam}{\sfD_{\chi^2}(\calP)}
\newcommand{\DiamH}{\sfD_{H^2}(\calP)}
\newcommand{\AffH}{\mathsf{\Delta}_{H^2}(\calP)}

\begin{document}
\title{Approximate independence of permutation mixtures}
\author{Yanjun Han, Jonathan Niles-Weed}
\maketitle

\begin{abstract}
We prove bounds on statistical distances between high-dimensional exchangeable mixture distributions (which we call \emph{permutation mixtures}) and their i.i.d.\ counterparts.
Our results are based on a novel method for controlling $\chi^2$ divergences between exchangeable mixtures, which is tighter than the existing methods of moments or cumulants.
At a technical level, a key innovation in our proofs is a new Maclaurin-type inequality for elementary symmetric polynomials of variables that sum to zero and an upper bound on permanents of doubly-stochastic positive semidefinite matrices.
We obtain as a corollary a new de Finetti-style theorem (in the language of Diaconis and Freedman, 1987), as well as several new statistical results, including a differential privacy guarantee for the ``shuffled privacy model'' with Gaussian noise and improved generic consistency guarantees for empirical Bayes procedures in compound decision problems.
\end{abstract}


\section{Introduction}
Let $n$ be a positive even integer and $\mu$ a positive real number. Consider two probability distributions on $\bR^n$:
\begin{itemize}
	\item Under $\bP_n$, the observation $X$ is of the form $\vartheta + Z$, where $\vartheta$ is uniformly distributed on the subset of \emph{balanced} vectors in $\{\pm \mu\}^n$ (i.e., vectors with exactly $n/2$ entries equal to $-\mu$ and $n/2$ equal to $+\mu$) and $Z \sim \calN(0,  I_n)$ is independent Gaussian noise;
	\item Under $\bQ_n$, the observation $X$ is of the form $\vartheta + Z$, where $\vartheta$ is uniformly distributed on $\{\pm \mu\}^n$ and $Z \sim \calN(0, I_n)$ is independent Gaussian noise.
	Note that the entries of $X$ under $\bQ_n$ are themselves i.i.d., and that the marginal distribution of each coordinate of $X$ under $\bQ_n$ and $\bP_n$ is the same.
\end{itemize}
In other words, let $\nu_\bP$ denote the distribution of $n$ uniformly random draws from the multiset $\{-\mu, \dots, -\mu, +\mu, \dots, +\mu\}$ \emph{without replacement} and $\nu_\bQ$ the corresponding distribution of draws \emph{with replacement}.
Then $\bP_n = \nu_\bP \star \calN(0, I_n)$ and $\bQ_n = \nu_\bQ \star \calN(0,  I_n)$.
For which values of $n$ and $\mu$ are $\bP_n$ and $\bQ_n$ statistically close?

The law of large numbers implies that a random vector $\vartheta$ with i.i.d.\ $\{\pm \mu\}$ entries will be nearly balanced, and, if $\mu \lesssim 1$, the Gaussian noise will make it difficult to detect whether $\vartheta$ is perfectly balanced or only nearly so.
It is therefore natural to conjecture that $\bP_n$ and $\bQ_n$ will not be distinguishable from each other as long as $\mu$ is not too large.
Surprisingly, however, as we show in \Cref{sec:failure}, existing approaches to bounding statistical distances between mixtures fail to establish this fact.
The goal of this paper is to develop new strategies for analyzing such high-dimensional mixtures, which achieve tight bounds.

As an example, a corollary of one of our main results shows that for the model described above,
\begin{equation}\label{eq:balanced_example}
	\chi^2(\bP_n \| \bQ_n) = \begin{cases}
		O(\mu^{4}) &\text{if } \mu \leq 1, \\
		O(\exp(\mu^2)) &\text{if } \mu > 1.
	\end{cases}
\end{equation}
Strikingly, \eqref{eq:balanced_example} implies that the $\chi^2$ divergence between these two high-dimensional models is bounded \emph{independent} of the dimension $n$---in particular, the sequence $\bP_n$ is contiguous to $\bQ_n$ in the sense of Le Cam~\cite{Le-60}, so that asymptotic properties of the mixture $\bP_n$ can be analyzed under the simpler product measure $\bQ_n$.
The bound in~\eqref{eq:balanced_example} also reveals other features: since $\chi^2(\calN(+\mu, 1) \| \calN(-\mu, 1)) = e^{4\mu^2} - 1$, the $\chi^2$ divergence between the two models is 
polynomially smaller than the ``single-letter'' $\chi^2$ divergence that would arise from comparing the two univariate marginal distributions that appear in the definition of $\bP_n$ and $\bQ_n$.
As we shall show, up to constants, the bounds in~\eqref{eq:balanced_example} are tight.

Our results hold for much more general models, which we call \emph{permutation mixtures}.
Let $\calP$ be a collection of probability measures on a fixed probability space.
Given $P_1,\cdots,P_n\in \calP$, we define two mixture models: 
\begin{itemize}
    \item The distribution $\bP_n$ of the observation $X=(X_1,\cdots,X_n)$ is the ``permutation mixture'' defined by
    \begin{align}\label{eq:bP_n}
        (X_1,\cdots,X_n) \sim \bE_{\pi\sim \Unif(S_n)}\qth{ \otimes_{i=1}^n P_{\pi(i)} }; 
    \end{align}
    \item The distribution $\bQ_n$ of the observation $X=(X_1,\cdots,X_n)$ is an i.i.d.\ product of one-dimensional mixtures defined by
    \begin{align}\label{eq:bQ_n}
        (X_1,\cdots,X_n) \sim \pth{\frac{1}{n}\sum_{i=1}^n P_i}^{\otimes n}. 
    \end{align}
\end{itemize}
As before, $\bP_n$ is a measure with exchangeable but not independent coordinates, whereas $\bQ_n$ is a product measure whose one-dimensional marginals agree with those of $\bP_n$.
Note that the simple example given above corresponds to the setting where $\calP = \{\calN(-\mu, 1), \calN(+\mu, 1)\}$ and $P_1, \dots, P_n$ are divided evenly between the two possibilities.
Our main contribution is a set of new techniques for bounding $\chi^2(\bP_n \| \bQ_n)$ for general permutation mixtures.

Permutation mixtures arise, explicitly or implicitly, in a number of statistical settings in which either the data or parameter space enjoy permutation symmetry.
However, the lack of general techniques for analyzing such mixtures has meant that researchers have had to turn to \textit{ad hoc} arguments, many of which lead to overly complicated proofs or require unnecessary assumptions.
The tools we develop here allow many of these results to be strengthened and simplified.
In \Cref{sec:statistics}, we illustrate this phenomenon with vignettes from high-dimensional statistics, differential privacy, and empirical Bayes estimation, showing how our bounds improve and clarify existing statistical results.

More generally, the study of permutation mixtures is motivated by important theoretical questions in a number of areas: 


\vspace{0.2cm} \noindent \textbf{de Finetti-style theorems.} The classical de Finetti theorem~\cite{de1929funzione} asserts that an infinite exchangeable sequence of binary random variables is a mixture of i.i.d.\ Bernoulli sequences. Since the work of Diaconis~\cite{Dia77}, there has been interest in obtaining quantitative forms of this theorem for finite exchangeable sequences~\cite{stam1978distance,diaconis1980finite,gavalakis2021information}, and, more broadly, to obtain ``de Finetti-style theorems'' in general settings~\cite{diaconis1987dozen}.
Our main results, which establish that the exchangeable measure $\bP_n$ is close to the product measure $\bQ_n$, are theorems of this type.
As \Cref{thm:deFinetti} shows, our bounds yield new approximation results for general exchangeable sequences.

\vspace{0.2cm} \noindent \textbf{Mean-field approximation.} Whether a high-dimensional distribution can be accurately approximated by a product measure is a central question in probability theory.
These ``mean-field'' approximations are useful to study the behavior of large random systems with weak dependence and open the door to precise asymptotic computations of, e.g., the free energy for complicated models arising from statistical physics~\cite{Par88,JaiRisKoe19,LacMukYeu24}.
The present case can be seen as a simple example of this philosophy:  is the interaction between coordinates in a permutation mixture weak enough that the mixture distribution $\bP_n$ is close to a measure with i.i.d.\ coordinates?
Our main theorems answer this question in the affirmative.

\vspace{0.2cm} \noindent \textbf{Information geometry of high-dimensional mixture models.}
High-dimensional mixtures---such as the measures $\bP_n$ and $\bQ_n$ defined above---are ubiquitous models in theoretical statistics (see, e.g., \cite{lepski1999estimation,cai2011testing,ingster2012nonparametric,wu2016minimax,jiao2018minimax}). The analysis of such measures is routine in the univariate case (or, more generally, for product measures), but obtaining sharp bounds is difficult for high-dimensional mixtures whose coordinates are not independent. The setting we study is one in which we can rigorously compare $\bP_n$ to a simpler, independent counterpart $\bQ_n$. This comparison may be directly relevant in analyzing mixtures with permutation invariance.
Moreover, as we highlight in \Cref{sec:techniques}, our proofs are based on new techniques for bounding divergences between high-dimensional mixtures, which may be of broader applicability.

\vspace{0.2cm} Our main results are stated in terms of two information-theoretic quantities: given two probability measures $P$ and $Q$ on the same probability space, we define the $\chi^2$ divergence and squared Hellinger distance:
\begin{equation*}
	\chi^2(P\|Q) = \int \frac{(\rmd P - \rmd Q)^2}{\rmd Q}, \qquad H^2(P,Q) = \int \pth{\sqrt{\rmd P} - \sqrt{\rmd Q}}^2.
\end{equation*}
Both are fundamental measures of the statistical similarity between $P$ and $Q$; see, e.g.,~\cite{Tsy09}.

We now present our main bounds of the $\chi^2$ divergence, $\chi^2(\bP_n \| \bQ_n)$, between the permutation mixture $\bP_n$ in \eqref{eq:bP_n} and its i.i.d.\ counterpart $\bQ_n$ in \eqref{eq:bQ_n}. Our upper bounds will depend on several quantities of the distribution family $\calP$, defined as follows. 

\begin{definition}\label{defn:capacity_diam}
For a given family $\calP$ of probability distributions over the same space, define: 
\begin{enumerate}
    \item The \emph{$\chi^2$ channel capacity}, denoted by $\Capa$: 
    \begin{align*}
\Capa := \sup_{\rho\in \Delta(\calP)} I_{\chi^2}(P; X) = \sup_{\rho\in \Delta(\calP)} \bE_{P \sim \rho}\qth{\chi^2( P \|\bE_{P' \sim \rho}[P']  )}, 
\end{align*}
where $P\sim \rho$ and $X|P\sim P$ in the $\chi^2$ mutual information, and $\Delta(\calP)$ denotes the class of all prior distributions over $\calP$; 
\item The \emph{$\chi^2$ diameter}, denoted by $\Diam:= \sup_{P_1, P_2\in \calP} \chi^2(P_1 \| P_2)$; 
\item The \emph{$H^2$ diameter}, denoted by $\DiamH := \sup_{P_1, P_2\in \calP} H^2(P_1, P_2)$; 
\item The \emph{maximum $H^2$ singularity}, denoted by $\AffH$: 
\begin{align*}
\AffH := \pth{1-\frac{\DiamH}{2}}^{-2}. 
\end{align*}
\end{enumerate}

\end{definition}

When the $\chi^2$ divergence is replaced by the Kullback--Leibler (KL) divergence in \Cref{defn:capacity_diam}, the $\chi^2$ channel capacity coincides with the usual notion of channel capacity in information theory. We have $\Capa\le \Diam$ by convexity, and $\AffH \le \Diam + 1$ by \cite[Eqn. (7.33)]{polyanskiy2024information}. In addition, $H^2(P_1,P_2)=2$ if and only if $P_1$ and $P_2$ are mutually singular, in which case it is easy to see that $\chi^2(\bP_n \|\bQ_n)$ must grow with $n$ in the worst case (cf. \Cref{lemma:tightness}). Therefore, the quantity $\AffH$ measures the extent of singularity of $\calP$. Finally, we note that all quantities (whenever finite) depend only on $\calP$ and not on the dimension $n$ of $\bP_n$ or $\bQ_n$. We present our main theorem below, which upper bounds the $\chi^2$ divergence using only the above quantities. 

\begin{thm}\label{thm:main}
The following upper bounds hold: 
\begin{enumerate}
    \item If $\Capa<\infty$, then
    \begin{align}\label{eq:upper_bound_1}
        \chi^2(\bP_n \| \bQ_n) \le 10\sum_{\ell=2}^n \Capa^\ell; 
    \end{align}
    \item If $\AffH<\infty$ and $\Capa<\infty$, then
    \begin{align}\label{eq:upper_bound_2}
        \chi^2(\bP_n \| \bQ_n) \le (e\AffH)^{\Capa} - 1; 
    \end{align}
    \item If $\Diam < \infty$, then 
    \begin{align}\label{eq:upper_bound_3}
        \chi^2(\bP_n \| \bQ_n)\le (1+\Diam)^{1+\Capa} - 1. 
    \end{align}
\end{enumerate}
\end{thm}

\Cref{thm:main} shows that if $\Capa\le 1-\delta$ for some constant $\delta>0$, then the $\chi^2$ divergence is upper bounded by a dimension-independent constant $O_\delta(\Capa^2)$. In addition, even if $\Capa\ge 1$ but $\AffH<\infty$, the $\chi^2$ divergence is still upper bounded by a large albeit dimension-independent constant. This shows that the contiguity relation $\{\bP_n\} \triangleleft \{\bQ_n\}$ is not unique to the Gaussian family but holds for a large class of probability families $\calP$. Specializing $\calP$ to specific classes of distributions leads to the following corollary. 

\begin{corollary}\label{cor:specific_families}
The following results hold for specific families:
\begin{enumerate}
    \item Gaussian family $\calP = \{\calN(\theta,1): |\theta|\le \mu\}$: there exists an absolute constant $\mu_c>0$ such that
    \begin{align*}
    \chi^2(\bP_n \| \bQ_n) = \begin{cases}
        O(\mu^4) &\text{if } \mu \le \mu_c, \\
        \exp(O(\mu^3)) - 1 &\text{if } \mu > \mu_c. 
    \end{cases}
    \end{align*}
    \item Gaussian family with small support $\calP = \{\calN(\theta,1): |\theta| \le \mu, \theta \in \Theta\}$ with $|\Theta|<\infty$:
    there exists an absolute constant $\mu_c>0$ such that
    \begin{align*}
    	\chi^2(\bP_n \| \bQ_n) = \begin{cases}
    		O(\mu^4) &\text{if } \mu \le \mu_c, \\
    		\exp(O(|\Theta|\mu^2)) - 1 &\text{if } \mu > \mu_c. 
    	\end{cases}
    \end{align*}
    \item Bernoulli family $\calP = \{\Bern(p): p\in [\varepsilon,1-\varepsilon]\}$: there exists an absolute constant $\varepsilon_c\in (0,1/2)$ such that
    \begin{align*}
        \chi^2(\bP_n \| \bQ_n) = \begin{cases}
            O((1-2\varepsilon)^4) &\text{if } \varepsilon \ge \varepsilon_c, \\
            O\pth{\frac{1}{\varepsilon}} &\text{if } 0<\varepsilon < \varepsilon_c. 
        \end{cases}
    \end{align*}
    \item Poisson family $\calP = \{\Poi(\lambda): \lambda\in [0,M]\}$: it holds that
    \begin{align*}
        \chi^2(\bP_n \| \bQ_n) = \begin{cases}
            O(M^2) &\text{if } M \le 1, \\
            \exp\pth{ O(M^{3/2}) } - 1 &\text{if } M>1. 
        \end{cases}
    \end{align*}
\end{enumerate}
\end{corollary}
The proof of \Cref{cor:specific_families} follows from \Cref{thm:main} and the computations of $(\Capa,\AffH)$, which we defer to \Cref{append:cor_family}. For the example presented at the beginning of the introduction, a weaker upper bound $\exp(O(\mu^2))$ of \eqref{eq:balanced_example} follows from point 2 of \Cref{cor:specific_families}; we refer to \Cref{subsec:warmup} for the proof of the stronger form \eqref{eq:balanced_example} (via the inequality $\Capa\le 1-\exp(-\mu^2)$ in \eqref{eq:ub_f_mu}). We will also discuss the tightness of these upper bounds in \Cref{subsec:tightness}; typically, the quadratic dependence in \eqref{eq:upper_bound_1} is tight when $\Capa$ is small, and both the base and exponent of \eqref{eq:upper_bound_2} are not improvable in general. However, for specific $\calP$ there could still be gaps, and we refer the discussions to \Cref{subsec:tightness}. 

\Cref{thm:main} admits several generalizations to related settings. 

\ifthenelse{\boolean{arxiv}}{\paragraph{A ``noisy'' finite de Finetti theorem.}}{\vspace{0.2cm} \noindent \textbf{A ``noisy'' finite de Finetti theorem.}} 
The same proof technique of \Cref{thm:main} also yields a version of the de Finetti theorem for finite exchangeable sequences with added noise.
Given a permutation mixture $\bP_n$ and its i.i.d.\ counterpart $\bQ_n$, we denote by $\bP_{k,n}$ and $\bQ_{k,n}$ the marginal distribution of the first $k$ coordinates:
\begin{align*}
	\bP_{k, n} & = \Law(X_1, \dots, X_k) \quad \quad \text{for $(X_1, \dots, X_n) \sim \bP_n$} \\
	\bQ_{k, n} & = \Law(X_1, \dots, X_k) \quad \quad \text{for $(X_1, \dots, X_n) \sim \bQ_n$}\,.
\end{align*}
Of course, $\bQ_{k ,n}$ is still a product distribution.
The following theorem establishes bounds on the statistical distance between $\bP_{k, n}$ and this i.i.d. counterpart.
\begin{thm}\label{thm:deFinetti}
	Let $1\le k\le n$.
	Then
	\begin{align*}
		\chi^2\pth{\bP_{k, n} \| \bQ_{k, n}} \le \frac{k^2}{n^2}\cdot \begin{cases}
			20\Capa^2 &\text{if } \Capa \le \frac{n}{2k}, \\
			(e\AffH)^{\Capa} - 1 &\text{if } \Capa, \AffH < \infty.
		\end{cases}
	\end{align*}
\end{thm}
By a standard convexity argument, the same bound holds when comparing the marginal distribution of an arbitrary finite exchangeable sequence (which can be written as a \textit{mixture} of permutation mixtures) to a mixture of product distributions.
%
%


In the absence of upper bounds on $\Capa$ and $\AffH$, some well-known results due to Stam \cite{stam1978distance}, Diaconis--Freedman \cite{diaconis1980finite}, and several recent works \cite{gavalakis2021information,johnson2024relative} yield bounds of the form $\KL(\bP_{k, n} \| \bQ_{k, n}) = O(\frac{k^2}{n})$, and $O(\frac{k^2}{n(n+1-k)})$ when $|\calP| = O(1)$.
In particular, $\bP_{k, n}$ and $\bQ_{k, n}$ are, in general, statistically indistinguishable when $k = o(\sqrt n)$, and this requirement can be improved to $k = o(n)$ when $\calP$ is small.\footnote{Note that if no restriction is made on the size of $\calP$ or on $\AffH$, then the birthday paradox shows that the requirement $k = o(\sqrt n)$ is tight.}
\Cref{thm:deFinetti} generalizes these results from two directions. First, if $|\calP|=O(1)$, we have $\Capa\le |\calP|-1$ (corollary of \Cref{lemma:chi2-capacity}) even when the elements of $\calP$ are mutually singular. Therefore, our first upper bound gives $\chi^2(\bP_{k, n} \| \bQ_{k, n}) = O(\frac{k^2}{n^2})$ for $k\le cn$, generalizing Stam's results \cite{stam1978distance} to the $\chi^2$ divergence, which cannot be analyzed using his techniques. Second, the same $O(\frac{k^2}{n^2})$ scaling also holds when $\AffH < \infty$.
This corresponds to the setting in which the distributions in $\calP$ are ``noisy'' enough that the pairwise squared Hellinger distances are bounded away from $2$. In this case, the mixtures are statistically indistinguishable for any $k = o(n)$, without any restriction on $|\calP|$.
This observation suggests that the addition of noise in the form of a bound on $\AffH$ has a similar qualitative effect as a bound on the size of $\calP$. 

\ifthenelse{\boolean{arxiv}}{\paragraph{Neighboring permutation mixtures.}}{\vspace{0.2cm} \noindent \textbf{Neighboring permutation mixtures.}}
The next result concerns the statistical distance between two ``neighboring'' permutation mixtures. Given $P_1,\cdots,P_n\in \calP$, let $P'_1\in \calP$ be arbitrary and set $P'_i = P_i$ for all $i \geq 2$.
The permutation mixtures $\bP_n$ and $\bP'_n$ are then defined as
\begin{align*}
(X_1,\cdots,X_n) &\sim \bE_{\pi\sim \Unif(S_n)}\qth{ \otimes_{i=1}^n P_{\pi(i)}} \text{ under } \mathbb{P}_n; \\
(X_1,\cdots,X_n) &\sim \bE_{\pi\sim \Unif(S_n)}\qth{ \otimes_{i=1}^n P'_{\pi(i)}} \text{ under } \mathbb{P}'_n. 
\end{align*}
In other words, the components of $\bP_n$ and $\bP'_n$ only differ in one coordinate. The following theorem establishes an upper bound on the squared total variation (TV) distance. 

\begin{thm}\label{thm:two_mixtures}
If $\Diam<\infty$, then
\begin{align*}
\TV(\bP_n, \bP'_n)^2 \le \frac{1}{4}\int \frac{(\rmd \bP_n - \rmd \bP'_n)^2}{\rmd \overline{P}^{\otimes n} } \le \frac{3\Diam (e\AffH)^{3\Capa}}{n},
\end{align*}
where $\overline{P} := \frac{1}{n-1}\sum_{i=2}^n P_i$. In particular, for every test function $f\in L^2(\overline{P}^{\otimes n})$, it holds that
\begin{align*}
\pth{\bE_{\bP_n}[f] - \bE_{\bP'_n}[f]}^2 \le \bE_{\overline{P}^{\otimes n}}\qth{f^2}\cdot \frac{12\Diam (e\AffH)^{3\Capa}}{n}. 
\end{align*}
\end{thm}

\Cref{thm:two_mixtures} shows that the squared TV distance between $\bP_n$ and $\bP'_n$ is of the order $O_\calP(\frac{1}{n})$. In addition, the mean difference of any test function $f$ under $\bP_n$ and $\bP'_n$ is of order $O_\calP(n^{-1/2})$ as long as the second moment of $f$ under the i.i.d.\ distribution $\overline P^{\otimes n}$ is bounded. This result will find applications in compound decision problems and differential privacy in \Cref{sec:statistics}.

\subsection{Notation}
Throughout the paper all logarithms are in base $e$. For a positive integer $n$, let $[n]:=\{1,\ldots,n\}$, and $S_n$ be the permutation group over $[n]$. For a vector $x$ in $\bR^n$ or $\bC^n$, let $x^\top$ and $x^\sfH$ be its real and conjugate transposes, respectively, and $e_\ell(x) := \sum_{S\subseteq [n]: |S|=\ell} \prod_{i\in S} x_i$ be the elementary symmetric polynomial of degree $\ell$ in $x$. For a square matrix $A = (a_{ij})_{i,j\in [n]}\in \mathbb{R}^{n\times n}$, let $\trace(A) = \sum_{i=1}^n a_{ii}$ be its trace, and $\Perm(A) = \sum_{\pi \in S_n} \prod_{i=1}^n a_{i\pi(i)}$ be its permanent. Let $\calN(\mu,\sigma^2)$ denote the normal distribution with mean $\mu$ and variance $\sigma^2$, and $\calCN(0,1)$ denote the complex normal distribution with real and imaginary parts being independent $\calN(0,\frac{1}{2})$ random variables. For probability measures $\mu$ and $\nu$, let $L^2(\mu)$ be the space of all functions $f$ with $\bE_\mu\qth{f^2} < \infty$, $\mu \otimes \nu$ be the product measure, and $\mu\star \nu$ be the convolution defined as $\mu\star \nu(A) = \int \mu(A-x)\rmd \nu(x)$. For probability measures $P$ and $Q$ on the same probability space, let 
\begin{align*}
	\TV(P,Q) = \frac{1}{2}\int |\rmd P - \rmd Q|, \qquad H^2(P,Q) = \int \pth{\sqrt{\rmd P} - \sqrt{\rmd Q}}^2 \end{align*}
be the total variation (TV) and squared Hellinger distances, respectively, and 
\begin{align*}
	\KL(P\|Q) = \int \rmd P\log \frac{\rmd P}{\rmd Q}, \qquad \chi^2(P\|Q) = \int \frac{(\rmd P - \rmd Q)^2}{\rmd Q}
\end{align*}
be the Kullback--Leibler (KL) and $\chi^2$ divergences, respectively. A collection of inequalities between the above distances/divergences can be found in \cite[Chapter 7.6]{polyanskiy2024information}. 

We shall use the following standard asymptotic notations. For non-negative sequences $\{a_n\}$ and $\{b_n\}$, let $a_n = O(b_n)$ denote $\limsup_{n\to\infty} a_n/b_n < \infty$, and $a_n = o(b_n)$ denote $\limsup_{n\to\infty} a_n/b_n = 0$. In addition, we write $a_n = \Omega(b_n)$ for $b_n = O(a_n)$, $a_n = \omega(b_n)$ for $b_n = o(a_n)$, and $a_n = \Theta(b_n)$ for both $a_n = O(b_n)$ and $b_n = O(a_n)$. We will also use the notations $O_\theta, o_\theta$, etc. to denote that the hidden factor depends on some external parameter $\theta$. 

\subsection{Proof techniques}\label{sec:techniques}
The core of our approach is to express the $\chi^2$ divergence between $\bP_n$ and $\bQ_n$ in terms of matrix permanents. Let us denote the marginal law of the coordinates of $X$ under $\bQ_n$ by $\overline{P}$.
We may then write
\begin{equation}\label{eq:lr_is_perm}
 \frac{\rmd \bP_n}{\rmd \bQ_n}(x_1, \dots, x_n) = \bE_{\pi \sim \Unif(S_n)} \qth{\prod_{i=1}^{n} \frac{\rmd P_{\pi(i)}}{\rmd \overline P}(x_i)} = \frac{1}{n!} \Perm(M(x)),
\end{equation}
where $M: \bR^n \to \bR^{n \times n}$ is given by
\begin{equation*}
	M(x)_{ij} = \frac{\rmd P_{j}}{\rmd \overline P}(x_i).
\end{equation*}

Obtaining bounds on matrix permanents is a classical subject~\cite[Chapter 6]{minc1984permanents}.
This task is made challenging by the fact that no polynomial time algorithms to compute or approximate matrix permanents are known, which rules out the existence of general purpose bounds on $\Perm(M(x))$.
Our strategy is to exploit enough special structure of the matrix $M(x)$ to allow us to compute accurate upper bounds.

We present two approaches to analyzing~\eqref{eq:lr_is_perm}.
The first, simpler, approach involves expanding $\Perm(M(x))$ around the $n \times n$ all-ones matrix $J$.
This expansion, which we call the ``doubly centered'' expansion, has the appealing property that the summands are orthogonal in $L^2(\bQ_n)$ \emph{and} involve matrices whose rows sum to zero.
We then apply a key inequality, Lemma~\ref{lemma:hadamard}, which shows that permanents of matrices with centered rows are small.
This inequality is obtained via a new bound on elementary symmetric polynomials: for any vector $x \in \bR^n$, \emph{if $\sum_{i=1}^n x_i = 0$}, then
\begin{align}\label{eq:ineq_ESP}
	|e_\ell(x)| \leq C \sqrt{\binom{n}{\ell}} \left(\frac{1}{n}\sum_{i=1}^n x_i^2\right)^{\ell/2}.
\end{align}
This bound may be contrasted with Maclaurin's inequality~\cite{maclaurin1730iv}, a version of which implies the bound  $|e_\ell(x)| \leq \binom{n}{\ell} \left(\frac{1}{n}\sum_{i=1}^n x_i^2\right)^{\ell/2}$ for \emph{all} $x \in \bR^n$, which is easily seen to be tight when $x_1 = \dots = x_n = 1$.
The crucial difference between this classical bound and our new inequality is the square-root improvement in the leading coefficient from $\binom{n}{\ell}$ to $\sqrt{\binom{n}{\ell}}$, which may be viewed as a consequence of the cancellations induced by the condition  $\sum_{i=1}^n x_i = 0$.
Our new Maclaurin-type inequality, combined with the doubly centered expansion, gives rise to the first bound in Theorem~\ref{thm:main}.

Our second approach combines the above observations with additional spectral information.
Since $(X_1, \dots, X_n)$ are independent under $\bQ_n$, we have that for any $\pi, \pi' \in S_n$,
\begin{align*}
	\bE_{X \sim \bQ_n} \qth{\prod_{i=1}^n M_{i, \pi(i)}(X_i)\prod_{i=1}^n M_{i, \pi'(i)}(X_i)} & = \prod_{i=1}^n \bE_{X_i \sim \overline P} \qth{M_{i, \pi(i)}(X_i)M_{i, \pi'(i)}(X_i)} \\
	& = \prod_{i=1}^n \int \frac{\rmd P_{\pi(i)} \rmd P_{\pi'(i)}}{\rmd \overline P} =: n^n \prod_{i=1}^n A_{\pi(i), \pi'(i)},
\end{align*}
where we write $A \in \bR^{n \times n}$ for the matrix given by $A_{ij} = \frac 1n \int \frac{\rmd P_i \rmd P_j}{\rmd \overline P}$.
We obtain
\begin{align*}
	\chi^2(\bP_n \| \bQ_n) + 1 & =  \frac{1}{(n!)^2} \bE_{X \sim \bQ_n} \Perm(M(X))^2 = \frac{n^n}{n!} \Perm(A)\,.
\end{align*}

The matrix $A$ is doubly stochastic, therefore putting the question of bounding $\Perm(A)$ in the setting of the celebrated van der Waerden conjecture \cite{van1926aufgabe}, proved by Egorychev \cite{egorychev1981solution} and Falikman \cite{falikman1981proof}, which states that $\Perm(A)\ge \frac{n!}{n^n}$ for all $n\times n$ doubly stochastic matrices $A$.
We prove a new \emph{upper} bound on the permanent of such matrices: if $A$ is a doubly stochastic matrix with eigenvalues $1 = \lambda_1 > \lambda_2 \geq \dots \geq \lambda_n \geq 0$, then
\begin{equation*}
	\Perm(A) \leq  \frac{n!}{n^n} \prod_{i=2}^n \frac{1}{1 - \lambda_i}.
\end{equation*}
The second and third bounds in Theorem~\ref{thm:main} are obtained by relating the spectrum of $A$ to the information theoretic quantities in Definition~\ref{defn:capacity_diam}.

Both applications in \Cref{thm:deFinetti} and \ref{thm:two_mixtures} rely on a combination of the above two approaches, by first expanding $M(x)$ around the all-ones matrix $J$ and then relating each term in the expansion to a counterpart appearing in the expression of $\Perm(A)$, through the identities in \Cref{sec:identity}. In particular, a \emph{complex} version of the inequality \eqref{eq:ineq_ESP} for $x\in \bC^n$ turns out to be crucial in the proof of \Cref{thm:two_mixtures}, a technique which could be of independent interest. 

\subsection{Related work}\label{subsec:related_work}
Comparisons between $\bP_n$ and $\bQ_n$ first implicitly arose in the consideration of compound decision problems.
An early result in this direction is due to Hannan~\cite{hannan1953asymptotic} (see~\cite{hannan1955asymptotic}).
In our notation, his results reads as follows: suppose that $\calP = \{P_-, P_+\}$ consists of two elements, which are not orthogonal to each other.
Given natural numbers $n$ and $k$ with $k \leq n$, denote by $\bP^{(k)}_n$ the permutation mixture arising from taking $P_1 = \dots = P_k = P_+$ and $P_{k+1} = \dots = P_n = P_-$.
If $k / n$ is bounded away from $0$ and $1$ as $n, k \to \infty$, then
\begin{equation*}
	\TV(\bP_n^{(k)}, \bP_n^{(k+1)}) = o(1)\,.
\end{equation*}
In \cite{hannan1955asymptotic}, this result is used to establish asymptotic equivalence between the performance of permutation invariant and simple decision rules.
Subsequent empirical Bayes literature largely focused on simple decision rules and avoided explicit consideration of $\bP_n$, but the importance of comparing $\bP_n$ to $\bQ_n$ to obtain genuine oracle inequalities was emphasized by Greenshtein and Ritov~\cite{greenshtein2009asymptotic,Greenshtein_2019}, who obtained risk bounds for the two models under the squared loss.
The general power of permutation-invariant decision rules was recently investigated in~\cite{weinstein2021permutation}.

The fact that exchangeable distributions are ``essentially'' (mixtures of) product measures is an important heuristic which has its roots in the work of de Finetti~\cite{de1929funzione} and Hewitt--Savage~\cite{hewitt1955symmetric}.
Diaconis and Friedman observed that versions of this claim hold in a number of different models~\cite{diaconis1987dozen}, and precise quantitative versions have been established for finite exchangeable sequences under different assumptions~\cite{stam1978distance,diaconis1980finite,bobkov2005generalized,roos2015bobkov,gavalakis2021information,johnson2024relative}.
Recently, motivated by applications in conformal prediction (see, e.g.,~\cite{tibshirani2019conformal, barber2023conformal}), similar results have been obtained for ``weighted'' exchangeable sequences as well~\cite{barber2023finetti}.

Bounding the $\chi^2$ divergence for mixture distributions is a common task in high-dimensional statistics~\cite{lepski1999estimation,Bar02,cai2011testing,ingster2012nonparametric,wu2016minimax,ColComTsy17,banks2018information,jiao2018minimax,perry2018optimality,balakrishnan2019hypothesis,bandeira2020optimal}, and obtaining bounds for ``low-degree'' versions of the $\chi^2$ divergence has become an important technique in establishing statistical-computational gaps (see, e.g., \cite{kunisky2019notes}).
In most of these applications, $\bP_n$ is a mixture distribution but $\bQ_n$ is particularly simple (for instance, its coordinates are i.i.d.\ $\calN(0,1)$ or $\Bern(1/2)$); in this case, expanding the likelihood ratio $\frac{\rmd \bP_n}{\rmd \bQ_n}$ in a basis of $L^2(\bQ_n)$-orthogonal polynomials yields explicit bounds in terms of the moments of $\bP_n$.
Few general techniques exist to obtain sharp results outside of this simple setting, though Schramm and Wein~\cite{schramm2022computational} (see also~\cite{rush2023easier}) showed that a related expansion based on cumulants yields a (potentially loose) bound when $\bQ_n$ is a mixture distribution.
In a more general context, Kunisky~\cite{kunisky2024low} has shown that employing the Hoeffding decomposition~\cite{Hoe48} can sometimes obtain comparable results for more complicated choices of $\bQ_n$.
We employ a similar technique in \Cref{sec:basis}.

Bounds for permutation mixtures were first obtained by Ding~\cite{Din22} in a PhD thesis supervised by the second author.
He established contiguity in the setting of the example presented in~\eqref{eq:balanced_example}, by showing
\begin{equation*}
	\chi^2(\bP_n \| \bQ_n) = \begin{cases}
		O(\mu^{2}) &\text{if } \mu \leq 1, \\
		O(\exp(4 \mu^2)) &\text{if } \mu > 1.
	\end{cases}
\end{equation*}
He also developed analogous bounds for a Bernoulli version of the problem. A dimension-independent upper bound on the KL divergence $\KL(\bP_n \| \bQ_n)$ was also observed and obtained in \cite[Theorem 2]{tang2023capacity} for a finite class $\calP$, used as a crucial technical step towards establishing the capacity upper bound of the noisy permutation channel in \cite{makur2020coding}. Their proof relies on an anti-concentration property for the histogram of independent discrete observations with non-identical distributions, leading to an upper bound linear in $|\calP|$. They also prove a bound analogous to the one presented in \Cref{thm:deFinetti} for $k$-dimensional marginals with $k<n$ via a convexity argument~\cite[Appendix A]{tang2023capacity}; however, their bound scales as 
$O(\frac{k}{n})$ rather than $O(\frac{k^2}{n^2})$. 

\subsection{Organization}
The rest of this paper is organized as follows. 
In \Cref{sec:statistics}, we present several statistical applications of our main theorems, highlighting their usefulness in the analysis of statistical procedures.
We then turn to our main results.
In \Cref{sec:failure}, we review some existing methods of bounding the $\chi^2$ divergence between mixture distributions, and show their failure to obtain the tight bound \eqref{eq:balanced_example} for the toy Gaussian model at the beginning of the introduction. \Cref{sec:basis} and \ref{sec:permanents} provide the details of the doubly centered expansion and the matrix permanent approach, respectively, where a key inequality in the analysis is a new upper bound on the elementary symmetric polynomial for centered real or complex vectors in \Cref{subsec:key-inequality}. 
We defer broader discussion of our results (\Cref{sec:discussion}), 
further comments on our proof techniques (\Cref{sec:identity}), proofs of the new inequality (\Cref{append:saddlepoint}) and other main results (\Cref{append:proofs}) to the appendix. 
\section{Statistical applications}\label{sec:statistics}
In this section, we present several statistical applications of our main results.
These include the identification of the least-favorable prior for the $\ell^p$ constrained Gaussian sequence model, an improved differential privacy guarantee for the ``shuffled privacy model,'' and black-box consistency results for empirical Bayes procedures in compound decision problems.
In all three cases, our results yield more general results, with fewer assumptions, than were previously available.

\subsection{The least-favorable prior over $\ell^p$ balls}\label{subsec:GSM}
Our results give new insights into the classical Gaussian sequence model.
Let the observation vector be $X\sim \calN(\theta, I_n)$, with an unknown mean vector $\theta\in \bR^n$ lying in the following $\ell^p$ ball $\Theta_p(R)$: 
\begin{align*}
	\Theta_p(R) := \sth{ \theta\in \bR^n: \frac{1}{n}\sum_{i=1}^n |\theta_i|^p \le R^p }. 
\end{align*}
Here $p>0$ is a given norm parameter, and $R = R_n$ is a given radius; we will assume that $R = o(1)$ to promote sparsity in this problem. Our target is to understand the minimax $\ell^q$ risk $r_{p,q}^\star(n,R)$ for the mean estimation: 
\begin{align*}
	r_{p,q}^\star(n,R) := \inf_{\widehat{\theta}}\sup_{\theta\in \Theta_p(R)} \bE_\theta \qth{\frac{1}{n}\sum_{i=1}^n |\widehat{\theta}_i - \theta_i|^q}. 
\end{align*}
Sharp bounds on this minimax risk were first established by Dohono and Johnstone~\cite{donoho1994minimax}, who showed that if $q\ge p\vee 1$, a soft-thresholding estimator with threshold $\mu = \sqrt{2\log(1/R^p)}$ is asymptotically minimax, with risk $r_{p,q}^\star(n,R) = (1+o(1))\mu^{q-p} R^p$, \textit{provided that} the density condition $s:=\frac{nR^p}{\mu^p} = \omega(1)$ holds.
This condition was later removed in \cite{zhang2012minimax}, who showed by a different argument that the same result holds in the sparse case, when $s = O(1)$.

The most challenging aspect of these results, highlighted in~\cite{donoho1994minimax}, is to obtain a lower bound on the minimax risk $r_{p,q}^\star(n,R)$ by bounding the Bayes risk under (a sequence of) carefully chosen priors.
Identifying these ``least-favorable priors'' gives fundamental insight into the statistical structure of the problem, since they capture the core difficulty of the estimation task.
Due to the permutation invariance of the parameter space $\Theta_p(R)$, it is natural to conjecture that a permutation prior is the asymptotically least favorable, under which the marginal distribution of the observations is a permutation mixture.
However, the technical difficulties of analyzing such mixtures meant that~\cite{donoho1994minimax} used the i.i.d.\ prior $(1-\alpha)\delta_0 + \frac{\alpha}{2}(\delta_{-\mu} + \delta_{\mu})$ instead, with $\alpha = (1+o(1))\frac{s}{n}$.\footnote{We note for comparison with what follows that their argument applied to the two-point prior $(1-\alpha)\delta_0 + \alpha\delta_\mu$ gives the same lower bound.}
Of course, this prior is not supported on $\Theta_p(R)$; however, the crucial density condition $s = \omega(1)$ is used to guarantee that it approximately concentrates on $\Theta_p(R)$ as $n \to \infty$.
Nevertheless, their argument does not reveal whether a permutation prior is indeed least favorable.

To analyze the sparse case, Zhang~\cite{zhang2012minimax} successfully analyzed the permutation prior in the $s = O(1)$ regime.
He established that the prior given by $\theta=(v_{\pi(1)},\dots,v_{\pi(n)})$ with $\pi\sim \Unif(S_n)$, where $v=(\mu,\dots,\mu,0,\dots,0)$ has $s$ nonzero entries, is indeed asymptotically least favorable.
The major benefit of this analysis is that it explicitly identifies a \textit{bona fide} prior on $\Theta_p(R)$.
However, Zhang's key step~\cite[Prop. 1]{zhang2012minimax} fails unless $s = O(1)$.
What limits the scope of his results is the lack of tools for analyzing the permutation mixture.

To summarize, these results leave open the question of what drives the statistical difficulty of the problem. Is there a fundamental difference between the two regimes $s = \omega(1)$ and $s = O(1)$, or does the same least-favorable prior saturate the minimax risk in both cases?
Our results show that the latter is the case.

\begin{lemma}\label{lemma:gaussian-seq-model}
	Let $q\ge p\vee 1$ and $R = o(1)$. For $\mu = \sqrt{2\log(1/R^p)}$ and $s=\frac{nR^p}{\mu^p}=\omega(1)$, the Bayes $\ell^q$ risk under the permutation prior $\theta=(v_{\pi(1)},\dots,v_{\pi(n)})$ with $\pi\sim \Unif(S_n)$ and
	\begin{align*}
		v = (\sqrt{1-\varepsilon}\mu,\dots,\sqrt{1-\varepsilon}\mu,0,\dots,0), \quad \|v\|_0 = \lfloor s\rfloor
	\end{align*}
	is lower bounded by $(1-o(1))(1-c(\varepsilon))\mu^{q-p}R^p$, with $c(\varepsilon)\to 0$ as $\varepsilon\to 0^+$. 
\end{lemma}
When combined with the analysis in~\cite{zhang2012minimax} for the $s= O(1)$ regime, \Cref{lemma:gaussian-seq-model} shows upon taking $\varepsilon \to 0^+$ that permutation priors are asymptotically least favorable for the minimax risk $r_{p,q}^\star(n,R)$ for any choice of $s$.

	\subsection{Amplification by shuffling for the Gaussian mechanism}
	\Cref{thm:two_mixtures} also finds applications to differential privacy \cite{dwork2006calibrating}, particularly in the context of the \emph{shuffled privacy model} \cite{erlingsson2019amplification}. In this model, each client sends her (randomized) message to a secure shuffler, which randomly permutes all incoming messages before forwarding them to the server.
	This stylized setting captures the behavior of a system in which users' submissions are stripped of identifying metadata before aggregation, so that an adversary does not know which user is associated with which data point.
	It is known that random shuffling can strengthen privacy guarantees, a phenomenon that~\cite{erlingsson2019amplification} calls ``amplification by shuffling.''
	Our techniques give a very simple proof of this fact, including for the Gaussian mechanism, for which no such guarantees were previously known.
	
	To this end, consider the standard differential privacy requirement with two neighboring datasets, where the private data of client~1 differs between them. Let $P_1,\dots,P_n$ denote the output distributions of the clients under the first dataset. Under the second dataset, only client~1's distribution changes, resulting in $(P_1', P_2,\dots,P_n)$. Finally, let $\bP_n$ and $\bP_n'$ denote the distributions of the \emph{shuffled} outputs under these datasets. We have the following result. 
	
	\begin{lemma}\label{lemma:privacy}
		Suppose $P_i = x_i + Z$ and $P_1' = x_1' + Z$ for some noise mechanism $Z$, with bounded private data $x_1,\dots,x_n,x_1'\in [0,1]$. For $\varepsilon\in (0,1)$, both the Laplace mechanism $Z\sim \mathrm{Lap}(\frac{1}{\varepsilon})$ and the Gaussian mechanism $Z\sim \calN(0,\frac{1}{\varepsilon^2})$ achieve
		\begin{align*}
			\TV(\bP_n, \bP_n') = O\pth{\frac{\varepsilon}{\sqrt{n}}}. 
		\end{align*}
	\end{lemma}
	
	Without random shuffling, it is well known that both mechanisms only achieve a TV distance $\Theta(\varepsilon)$ between the output distributions. In other words, random shuffling amplifies privacy. This phenomenon has been studied in \cite{erlingsson2019amplification,girgis2021renyi,feldman2022hiding}, where the results require each client's local mechanism to satisfy $\varepsilon$-LDP (local differential privacy). Our requirement in \Cref{lemma:privacy} recovers this condition for the Laplace mechanism (which is $\varepsilon$-LDP), but our result also applies to the Gaussian mechanism (which does not satisfy $\varepsilon$-LDP for any $\varepsilon>0$), which gives a novel demonstration of amplification by shuffling in this setting.
\subsection{Consistency of empirical Bayes procedures for the compound decision problem}\label{subsec:EB-quadratic}
	Our results can also be applied to obtain general results for the compound decision problem~\cite{robbins1951asymptotically}.
	In this problem, the statistician observes independent data $X_1, \dots, X_n$, where $X_i \sim P_{\theta_i}$ for $i=1, \dots, n$; the goal is to estimate  the vector of parameters under a separable loss $L (\theta, \widehat \theta) := \frac 1 n \sum_{i=1}^n \ell(\theta_i, \widehat \theta_i)$.
	This setting is the starting point of the theory of empirical Bayes~\cite{Efr19}, which views the parameters as having been drawn i.i.d.\ from a prior distribution, which is in turn estimated from the observations.

	An oracle version of the compound decision problem was suggested by Brown and Greenshtein~\cite{brown2009nonparametric} (see also~\cite{Jiang_2009}), in which the statistician knows the unordered set $\{\theta_1, \dots, \theta_n\}$, but not their correspondence to the observations $(X_1, \dots, X_n)$.
	Equivalently, the observations may be viewed as arising from the permutation mixture $\bP_n$.
	As Greenshtein and Ritov~\cite{greenshtein2009asymptotic} observe, however, most oracle inequalities in the empirical Bayes literature compare to an oracle that knows $\{\theta_1, \dots, \theta_n\}$ but is restricted to the use of ``simple'' decision rules, of the form $\widehat \theta = (\Delta(X_1), \dots, \Delta(X_n))$ for some fixed univariate function $\Delta$.
	Such rules are in fact optimal when the observations arise not from $\bP_n$ but from its independent counterpart $\bQ_n$.
	Denoting by $\calD^{\mathrm{S}}$ the set of such simple rules, the ``Fundamental Theorem of Compound Decisions'' (see, e.g.,~\cite{copas1969compound}) states that for any $\widehat \theta \in \calD^{\mathrm{S}}$,
	\begin{equation}\label{eq:fundamental_theorem}
		\bE[L(\theta, \widehat \theta)] =  \bE_{\vartheta \sim \frac{1}{n}\sum_{i=1}^n \delta_{\theta_i}} \ell(\vartheta, \Delta(X))\,,
	\end{equation}
	where the final expression denotes the risk in a univariate Bayesian estimation problem under which $\vartheta$ is a random element of $\{\theta_1, \dots, \theta_n\}$ and, conditional on $\vartheta$, $X \sim P_{\vartheta}$.
	Denoting the optimal risk for simple rules by $r^{\mathrm{S}}(\theta) = 	\inf_{\widehat \theta \in \calD^{\mathrm{S}}} \bE[L(\theta, \widehat \theta)]$, the empirical Bayes method yields estimators $\widehat \theta^{\mathrm{EB}}$ \textit{depending on the whole observation} (i.e., is not simple) whose risk satisfies
	\begin{equation}\label{eq:oracle_inequality}
		\bE[L(\theta, \widehat \theta^{\mathrm{EB}})] = r^{\mathrm{S}}(\theta) + o(1)\,.
	\end{equation}
	For example, Robbins~\cite{robbins1951asymptotically} showed that in the Gaussian location model with $\theta_i \in \{\pm 1\}$ for $i = 1, \dots, n$ and the zero one loss, an empirical Bayes estimator has excess risk $O(n^{-1/2})$ over $r^{\mathrm{S}}(\theta)$.  
	(See~\cite{zhang2003compound} for a summary of related results.)
	In light of~\eqref{eq:fundamental_theorem}, such procedures asymptotically match the performance of an oracle that knows the parameters $\theta$ but is restricted to the use of simple rules.

	As emphasized by~\cite{greenshtein2009asymptotic}, however, such guarantees are only partially convincing.
    Indeed, two important questions remain.
	In the first place, since $\widehat \theta^{\mathrm{EB}}$ is not simple, comparison with the oracle risk $r^{\mathrm{S}}(\theta)$ is not justified.
	However, since empirical Bayes procedures satisfy the permutation invariance property
	\begin{align*}
		\widehat{\theta}^{\mathrm{EB}}_{\pi(i)}(X_{\pi(1)},\dots,X_{\pi(n)}) = \widehat{\theta}^{\mathrm{EB}}_{i}(X_{1},\dots,X_n), \quad i\in [n], 
	\end{align*}
	for all $\pi \in S_n$, it is more honest to compare against the class $\calD^{\mathrm{PI}}$ of permutation invariant estimators.
    This raises the question of characterizing the difference in performance between the best simple decision rule (which is optimal for the i.i.d.\ setting) and the best permutation-invariant decision rule (which is optimal for the compound decision setting).
	Moreover, once the restriction to simple estimators is dropped, there is no clear analogue to~\eqref{eq:fundamental_theorem}; indeed, in the general setting two natural multivariate Bayesian problems arise (see, e.g.~\cite{weinstein2021permutation}), the first under which the random parameter $(\vartheta_1, \dots, \vartheta_n)$ is a uniform random permutation of $(\theta_1, \dots, \theta_n)$, and the second under which its coordinates are i.i.d.\ draws from $\frac{1}{n}\sum_{i=1}^n \delta_{\theta_i}$.
	It is the latter problem that appears in~\eqref{eq:fundamental_theorem}, but the relationship between the risks of estimators in the two different Bayesian settings is not clear.
    This raises the question of whether estimators that achieve a good risk bound in one setting also achieve a good risk bound in the other.
	
	Our results allow us to address both concerns.
	First, \Cref{cor:EB} shows that the optimal risks of simple and permutation invariant decision rules agree up to $O(n^{-1/2})$ for general models and bounded losses.
	Denote the optimal risk among permutation invariant decision rules by $\calD^{\mathrm{PI}}$:
	\begin{equation*}
		r^{\mathrm{PI}}(\theta) = \inf_{\widehat{\theta}\in \calD^{\mathrm{PI}}} \bE\qth{L(\theta, \widehat{\theta})}\,.
	\end{equation*}
	
	\Cref{thm:two_mixtures} implies the following. 
	\begin{lemma}\label{cor:EB}
		Let $P_{\theta_1},\cdots,P_{\theta_n}\in \calP$ with $\Diam < \infty$, and $0\le \ell(\cdot,\cdot) \le M$. Then
		\begin{align*}
			r^{\mathrm{S}}(\theta) - r^{\mathrm{PI}}(\theta) \le M\sqrt{6\Diam (e\AffH)^{3\Capa}/n}. 
		\end{align*}
	\end{lemma}
	
	\Cref{cor:EB} shows that for a large class of models and losses, $r^\mathrm{S}(\theta) - r^{\mathrm{PI}}(\theta) = O_{\calP}(n^{-1/2})$.
	This strengthens a ``folklore'' result in~\cite{hannan1955asymptotic} that shows that the gap is $o(1)$.
	Combined with~\eqref{eq:oracle_inequality}, these results yield honest oracle inequalities for empirical Bayes procedures.
	Our improvement is salient for the majority of parametric empirical Bayes problems (like that of Robbins) for which the excess risk over simple decision rules is $O(n^{-1/2})$.
	
	The connection between \Cref{cor:EB} and \Cref{thm:two_mixtures} is the Bayesian perspective alluded to above.
	Both $r^{\mathrm{S}}(\theta)$ and $r^{\mathrm{PI}}(\theta)$ can be analyzed within the Bayesian framework under which $\vartheta$ is a uniform random permutation of $\theta$.
	In this formulation, comparing $r^{\mathrm{S}}(\theta)$ with $r^{\mathrm{PI}}(\theta)$ reduces to comparing the conditional distributions $P_{\vartheta_i \mid X_i}$ and $P_{\vartheta_i \mid X^n}$, where \Cref{thm:two_mixtures} plays a central role.
	
	Second, \cref{thm:permutation_prior} gives a general comparison between the two different Bayesian settings described above for arbitrary estimators.
	In particular, it implies that for general bounded losses, estimators that have vanishing risk under the i.i.d.\ prior also have vanishing risk under the more complicated permutation prior.
		\begin{lemma}\label{thm:permutation_prior}
	For any estimator $\widehat{\vartheta} = \widehat{\vartheta}(X)$ and separable loss $L(\vartheta,\widehat{\vartheta}) = \frac 1n \sum_{i=1}^n \ell(\vartheta_i, \widehat{\vartheta}_i)$, we have
			\begin{align}\label{eq:mutual_info_1}
				|\bE_\bP[L(\vartheta,\widehat{\vartheta})] - \bE_\bQ[L(\vartheta,\widehat{\vartheta})]| \le \sqrt{ e(\chi^2(\calP)+1)} \cdot \frac 1n \sum_{i=1}^n \sqrt{ \var_{\bQ}[\ell(\vartheta_i,\widehat{\vartheta}_i)]}, 
			\end{align}
			where $\chi^2(\calP)$ is any upper bound in \Cref{thm:main} for the distribution family $\calP=\{P_{\theta_1},\dots,P_{\theta_n}\}$, and $\bP$ and $\bQ$ denote the joint distributions of $(\vartheta, X)$ under the permutation prior and i.i.d. prior, respectively. 
	
	Moreover, if $0 \leq \ell(\cdot, \cdot) \leq M$, then
	\begin{equation}\label{eq:consistency}
		\bE_\bP[L(\vartheta,\widehat{\vartheta})] \leq \bE_\bQ[L(\vartheta,\widehat{\vartheta})] + \sqrt{ eM (\chi^2(\calP)+1) \bE_\bQ[L(\vartheta,\widehat{\vartheta})]}\,.
	\end{equation}
	In particular, if $M (\chi^2(\calP)+1)$ is bounded, then $\bE_\bQ[L(\vartheta,\widehat{\vartheta})] \to 0$ implies $\bE_\bP[L(\vartheta,\widehat{\vartheta})] \to 0$.
	\end{lemma}
	
	We note that further improvements to \Cref{cor:EB} are possible for specific losses, via more specialized arguments.
	For example, an alternative strategy is developed in \cite{greenshtein2009asymptotic} for the quadratic loss $L(\theta,\widehat{\theta}) = \frac{1}{n}\|\theta - \widehat{\theta}\|_2^2$. For several classes $\calP$ (such as the Gaussian location model), \cite{greenshtein2009asymptotic} shows that $r^{\mathrm{S}}(\theta) - r^{\mathrm{PI}}(\theta) = O_{\calP}(n^{-1})$ under quadratic loss, improving upon the $O_{\calP}(n^{-1/2})$ upper bound in \Cref{cor:EB}, which holds for general losses. 
	An adaptation of their proof technique can be used to extend this result to broader classes $\calP$ under weaker assumptions. 
	
	\begin{lemma}\label{lemma:EB-quadratic}
		Let $P_{\theta_1},\dots,P_{\theta_n}\in \calP$ with $\Diam < \infty$, and $|\theta_i|\le M$ for all $i\in [n]$. Then under the quadratic loss $L(\theta,\widehat{\theta}) = \frac{1}{n}\|\theta - \widehat{\theta}\|_2^2$, 
		\begin{align*}
			r^{\mathrm{S}}(\theta) - r^{\mathrm{PI}}(\theta)  \le \frac{4M^2\Diam (1+\Capa)}{n}. 
		\end{align*}
	\end{lemma}
	
	We remark that the proof of \Cref{lemma:EB-quadratic} follows a different strategy than the other arguments in this paper.
	The key step in the proof of \Cref{lemma:EB-quadratic} is to apply convexity with a clever coupling of \cite{greenshtein2009asymptotic} to obtain a Hellinger upper bound between two permutation mixtures, as stated in \Cref{lemma:Greenshtein-Ritov}; however, as witnessed in \Cref{sec:failure}, this convexity-based idea is insufficient to establish \eqref{eq:balanced_example}.
	Moreover, even for the quadratic loss, \Cref{lemma:EB-quadratic} is still not optimal: in our follow-up work~\cite{han2025best}, we use the techniques of the present paper to obtain the first major improvement of \Cref{lemma:EB-quadratic} for Gaussian location models since \cite{greenshtein2009asymptotic}; see \cite[Theorem 4.1]{han2025best}.

\section{Failure of existing approaches}\label{sec:failure}
Before we present our proof of \Cref{thm:main}, we review several existing approaches to upper bound the statistical distance between mixture distributions. Unfortunately, we will show that all of them fail to yield the bound \eqref{eq:balanced_example}. For these approaches, occasionally we will show the failure under other statistical distances such as the squared TV distance or the KL divergence, both of which are no larger than the $\chi^2$ divergence. 

\subsection{Reduction to two simple distributions}
The simplest strategy to deal with two mixture distributions is a reduction to ``simple'' distributions via coupling and convexity. Taking the KL divergence for an example and returning to the setting of \eqref{eq:balanced_example}, it holds that
\begin{align*}
\KL(\bP_n \| \bQ_n) &= \KL(\bE_{\vartheta\sim \nu_\bP}[ \calN(\vartheta,I_n)] \| \bE_{\vartheta'\sim \nu_\bQ}[ \calN(\vartheta',I_n)]) \\
&\le \min_{\rho \in \Pi(\nu_\bP, \nu_\bQ)} \bE_{(\vartheta,\vartheta')\sim \rho} \qth{\KL(\calN(\vartheta,I_n) \| \calN(\vartheta',I_n))} = \frac{W_2^2(\nu_\bP,\nu_\bQ)}{2}, 
\end{align*}
where the inequality is due to the joint convexity of the KL divergence, $\Pi(\nu_\bP, \nu_\bQ)$ denotes all possible couplings with marginals $\nu_\bP$ and $\nu_\bQ$, and $W_2$ is the Wasserstein-$2$ distance. For the last term, since $\nu_\bP$ is supported on the set $L := \{\theta\in \{\pm\mu\}^n: \sum_{i=1}^n \theta_i=0\}$, we can lower bound it as
\begin{align*}
    W_2^2(\nu_\bP, \nu_\bQ) \ge \bE_{\vartheta\sim\nu_\bQ} \qth{ \min_{\theta_0\in L} \|\vartheta-\theta_0\|_2^2 } = \bE_{\vartheta\sim\nu_\bQ} \bqth{ 4\mu^2 \Big|\sum_{i=1}^n \mathbbm{1}_{\vartheta_i=\mu} - \frac{n}{2}\Big|  } = \Omega(\sqrt{n}\mu^2), 
\end{align*}
where $\sum_{i=1}^n \mathbbm{1}_{\vartheta_i=\mu}$ follows a binomial distribution $\mathsf{B}(n,\frac{1}{2})$ under $\nu_\bQ$, and the last step is due to the CLT. Consequently, this coupling can at best provide an upper bound of $O(\sqrt{n}\mu^2)$, which not only grows with $n$ but also exhibits a loose dependence on $\mu$ (i.e. $O(\mu^2)$ instead of $O(\mu^4)$ for small $\mu$). 

\subsection{Reduction to one simple distribution}\label{subsec:mixture-simple}
Instead of reducing to the $\chi^2$ divergence between two simple distributions, a more careful coupling approach could reduce to the $\chi^2$ divergence between a mixture distribution and a simple distribution, where the standard second moment computation \cite{ingster2012nonparametric} could then be applied. Still taking the KL divergence for an example, such a general strategy takes the form
\begin{align*}
\KL(\bP_n \| \bQ_n) &\le \min_{ \{\nu_{\theta'}\}_{\theta'\in \{\pm \mu\}^n } }\bE_{\vartheta'\sim \nu_\bQ}\qth{ \KL\pth{ \bE_{\vartheta\sim \nu_{\vartheta'}}\qth{\calN(\vartheta,I_n)} \| \calN(\vartheta',I_n) } },
\end{align*}
where the minimization is over all possible families of distributions $\{\nu_{\theta'}\}_{\theta'\in \{\pm \mu\}^n }$ such that $\bE_{\vartheta'\sim \nu_\bQ}[\nu_{\vartheta'}] = \nu_{\bP}$. 
This is the strategy developed in~\cite{Din22}, where it is employed with a judicious choice of the family $\{\nu_{\theta'}\}_{\theta'\in \{\pm \mu\}^n }$ to show that $\bP_n$ is contiguous to $\bQ_n$.

However, even though it succeeds in proving a dimension-free upper bound of $\KL(\bP_n \| \bQ_n)$, this strategy cannot yield the correct dependence on $\mu$.
Indeed, we have the following lower bound:
\begin{align*}
&\bE_{\vartheta'\sim \nu_\bQ}\qth{ \KL\pth{ \bE_{\vartheta\sim \nu_{\vartheta'}}\qth{\calN(\vartheta,I_n)} \| \calN(\vartheta',I_n) } } \\
&\stepa{\ge} \frac{1}{2}\bE_{\vartheta'\sim \nu_\bQ}\qth{ W_2^2\pth{ \bE_{\vartheta\sim \nu_{\vartheta'}}\qth{\calN(\vartheta,I_n)}, \calN(\vartheta',I_n) } } \\
&\stepb{\ge} \frac{1}{2}\bE_{\vartheta'\sim \nu_\bQ}\qth{ \left\| \bE_{\vartheta\sim \nu_{\vartheta'}}[\vartheta] - \vartheta' \right\|_2^2} \stepc{\ge} \frac{1}{2}\bE_{\vartheta'\sim \nu_\bQ}\bqth{ \frac{1}{n}\bpth{\sum_{i=1}^n \vartheta_i'}^2 } = \frac{\mu^2}{2}, 
\end{align*}
where (a) is the transportation cost inequality under the Gaussian measure $\calN(\theta',I_n)$ \cite[Theorem 6.6]{ledoux2001concentration}, (b) uses the inequality $W_2^2(P,Q)\ge \|\bE_P[X] - \bE_Q[X]\|_2^2$ due to convexity, and (c) notes that the vector $\bE_{\vartheta\sim \nu_{\vartheta'}}[\vartheta]$ always lies on the hyperplane $\{\theta\in \bR^n: \sum_{i=1}^n \theta_i = 0\}$ and uses the orthogonal projection onto this hyperplane. Therefore, this approach can at best lead to an upper bound $O(\mu^2)$, still larger than the correct dependence $O(\mu^4)$ for small $\mu$. On a high level, this is because the original mixtures $\bP_n$ and $\bQ_n$ have the same mean, but after applying the convexity, the means of $\bE_{\theta\sim \nu_{\theta'}}\qth{\calN(\theta,I_n)}$ and $\calN(\theta',I_n)$ no longer match. 

\subsection{Method of moments}
The previous failures illustrate the importance of directly comparing the mixtures $\bP_n$ and $\bQ_n$ rather than reducing one or both measures to simple distributions.
The method of moments is a powerful approach for performing such a direct comparison.
Expanding the likelihood ratio in the orthogonal basis of Hermite polynomials shows that (see, e.g. \cite[Lemma 2.2]{hardt2015tight} and \cite[Lemma 9]{wu2020optimal})
\begin{align}\label{eq:method_of_moments}
\TV(\nu_\bP \star \calN(0,I_n), \nu_\bQ \star \calN(0,I_n))^2 \le \sum_{\alpha\ge 0} \frac{(m_\alpha(\nu_\bP)-m_\alpha(\nu_\bQ))^2}{\alpha!}, 
\end{align}
where $\alpha=(\alpha_1,\cdots,\alpha_n)\in \naturals^n$ is a multi-index, $\alpha! \triangleq \prod_{i=1}^n (\alpha_i!)$, and $m_\alpha(\nu) = \bE_{X\sim \nu}\qth{\prod_{i=1}^n X_i^{\alpha_i}}$ denotes the $\alpha$-th joint moment of $X\sim \nu$. Since the moment difference in \eqref{eq:method_of_moments} vanishes when $|\alpha| = \sum_{i=1}^n \alpha_i = 1$, the method of moments indeed suggests an $O(\mu^4)$ dependence on $\mu$. 

However, the bound in \eqref{eq:method_of_moments} is not dimension free and tends to infinity as $n\to\infty$. To see this, fix some $\ell\in \naturals$ and consider the multi-index $\alpha=(1,1,2,\cdots,2,0,\cdots,0)$ with $2$ appearing $\ell$ times. By simple algebra, it is easy to check that 
\begin{align*}
m_\alpha(\nu_\bP) = \mu^{2\ell} \bE_{\vartheta \sim \nu_\bP}[\vartheta_{1}\vartheta_{2}] = -\frac{\mu^{2\ell+2}}{n-1}, \qquad m_\alpha(\nu_\bQ) = 0. 
\end{align*}
Therefore, the total contribution of all permutations of $\alpha$ to the RHS of \eqref{eq:method_of_moments} is
\begin{align*}
\frac{\mu^{4\ell +4}}{2^{\ell}(n-1)^2}\cdot \binom{n}{2,\ell,n-\ell-2} = \Theta_{\ell}\pth{\mu^{4\ell+4} n^{\ell}}, 
\end{align*}
which grows polynomially with $n$ for any constant $\ell\ge 1$. On a high level, this shows that although the method of moments works well in one dimension, the sum of squared moment differences might become unbounded in high dimensions due to a large number of cross terms. 

\subsection{Method of cumulants}
To address a similar problem arising in the recovery problem under a low-degree framework, a recent work \cite{schramm2022computational} established an upper bound on the $\chi^2$ divergence based on cumulants. Specifically, it was shown in \cite[Theorem 2.2]{schramm2022computational} and \cite[Proposition 2.1]{rush2023easier} that
\begin{align}\label{eq:cumulants}
\chi^2(\nu_\bP \star \calN(0,I_n) \| \nu_\bQ \star \calN(0,I_n))\le \sum_{\alpha\ge 0} \frac{\kappa_\alpha^2}{\alpha!}, 
\end{align}
where $\kappa_\alpha$ denotes the \emph{joint cumulant}
\begin{align*}
\kappa_\alpha = \kappa_{\vartheta\sim \nu_\bQ}\pth{\frac{\rmd\nu_\bP}{\rmd\nu_{\bQ}}(\vartheta), \vartheta_1, \ldots, \vartheta_1, \vartheta_2, \ldots, \vartheta_2, \ldots, \vartheta_n}, 
\end{align*}
where $\vartheta_i$ appears $\alpha_i$ times in the joint cumulant. We refer to \cite{schramm2022computational} for an in-depth discussion of cumulants; we shall only need the following recursive formula to compute $\kappa_\alpha$: 
\begin{align}\label{eq:cumulants_recursive}
\kappa_\alpha = \bE_{\vartheta \sim \nu_\bP}\qth{\prod_{i=1}^n \vartheta_{i}^{\alpha_i}} - \sum_{0\le \beta \lneq \alpha} \kappa_\beta \prod_{i=1}^n \qth{ \binom{\alpha_i}{\beta_i} \bE_{\vartheta_i\sim \Unif(\{\pm \mu\})} \qth{\vartheta_i^{\alpha_i - \beta_i}} }. 
\end{align}

Although the upper bound \eqref{eq:cumulants} enjoys several advantages over the moment-based approach such as a better behavior for product distributions, it still fails in our problem in a delicate way. Consider a multi-index $\alpha=(1,\ell,0,\ldots,0)$ and let $a_\ell = \kappa_{(1,\ell,0,\ldots,0)}$. Based on \eqref{eq:cumulants_recursive} and some algebra, it is easy to show $a_{2\ell}=0$ for all $\ell\in \naturals$, and that the modified sequence $b_{2\ell+1} = (-1)^{\ell+1}(n-1)a_{2\ell+1}/\mu^{2\ell+2}$ satisfies
\begin{align*}
b_{2\ell+1} - \binom{2\ell+1}{2}b_{2\ell-1} + \binom{2\ell+1}{4}b_{2\ell-3} - \cdots = (-1)^{\ell}, \qquad b_1 = 1. 
\end{align*}
The sequence $\{b_n\}$ is well known to be the number of alternating permutations,\footnote{Entry A000182 in The On-Line Encyclopedia of Integer Sequences, \url{https://oeis.org/A000182}} with the asymptotic growth \cite[Page 5]{flajolet2009analytic}
\begin{align*}
b_{2\ell+1} \sim 2\pth{\frac{2}{\pi}}^{2\ell+1}\cdot (2\ell+1)! \qquad \text{as } \ell\to\infty.
\end{align*}
We obtain that
\begin{align*}
\kappa_{(1,2\ell+1,0,\ldots,0)} = a_{2\ell+1} \sim (-1)^{\ell+1}\frac{\pi}{n-1}\left(\frac{2\mu}{\pi}\right)^{2\ell+2}\cdot (2\ell+1)! \qquad \text{as } \ell\to\infty.
\end{align*}
As the growth of $(2\ell+1)!$ is much faster than exponential, summing along this subsequence in \eqref{eq:cumulants} gives a diverging result, indicating the failure of this approach. 
\section{Upper bound via the doubly centered expansion}\label{sec:basis}
In this section we prove the first part \eqref{eq:upper_bound_1} of \Cref{thm:main}, by developing an orthogonal expansion of the likelihood ratio in terms of a set of ``doubly centered'' functions. In \Cref{subsec:warmup} we first provide a direct proof of the previous toy example and discuss the intuitions, and then present the proof for the general case in \Cref{subsec:general-case}. One key step is a new upper bound of elementary symmetric polynomials for a sequence summing into zero, which we present in \Cref{subsec:key-inequality}. 

\subsection{A warm-up example}\label{subsec:warmup}
Similar to \Cref{sec:failure}, this section focuses on the toy Gaussian model in \eqref{eq:balanced_example} as a warm-up example. The proof relies critically on the following representation of the Gaussian likelihood ratio: for $\theta\in \{\pm \mu\}$, 
\begin{align}\label{eq:Gaussian-LR}
\frac{\varphi(x-\theta)}{\varphi(x)} = \exp\pth{\theta x - \frac{\theta^2}{2}} = \cosh(\mu x)\exp\pth{-\frac{\mu^2}{2}} + \sinh(\mu x)\frac{\theta}{\mu}\exp\pth{-\frac{\mu^2}{2}},
\end{align}
where $\varphi(x)$ is the density function of $\calN(0,1)$. In addition, the marginal distribution under $\bQ_n$ is
\begin{align}\label{eq:varphi_0}
\varphi_0(x) := \frac{\varphi(x-\mu) + \varphi(x+\mu)}{2} = \varphi(x) \cosh(\mu x)\exp\pth{-\frac{\mu^2}{2}}, 
\end{align}
corresponding to the first term on the RHS of \eqref{eq:Gaussian-LR}. Based on \eqref{eq:Gaussian-LR} and \eqref{eq:varphi_0}, we can express the likelihood ratio between $\bP_n$ and $\bQ_n$ as
\begin{align}\label{eq:overall-LR}
\frac{\rmd \bP_n}{\rmd \bQ_n}(X) &= \bE_{\pi\sim\Unif(S_n)}\qth{\prod_{i=1}^n\frac{\varphi(X_i - \theta_{\pi(i)})}{\varphi_0(X_i)}} \nonumber \\
&\overset{\prettyref{eq:Gaussian-LR}}{=} \bE_{\pi\sim\Unif(S_n)}\qth{\prod_{i=1}^n\pth{1 + \tanh(\mu X_i)\frac{\theta_{\pi(i)}}{\mu}}} \nonumber \\
&= \sum_{S\subseteq [n]} \pth{\prod_{i\in S} \tanh(\mu X_i)} \bE_{\pi\sim\Unif(S_n)}\qth{\prod_{i\in S} \frac{\theta_{\pi(i)}}{\mu}}. 
\end{align}
Consequently, the likelihood ratio in \eqref{eq:overall-LR} is decomposed into a sum over all subsets $S\subseteq [n]$, and the effects of $X$ and $\pi$ are decoupled in each summand. Some observations are in order: 
\begin{enumerate}
    \item The functions $\{\prod_{i\in S}\tanh(\mu X_i): S\subseteq [n]\}$ are orthogonal under $\bQ_n$. To see this, let $S\neq T$ be two different subsets of $[n]$, so that there exists some $i_0\in S\Delta T$. As
    \begin{align*}
    \bE_{\bQ_n}\qth{\tanh(\mu X_{i_0})} = \int_{\bR} \varphi_0(x)\tanh(\mu x)\rmd x \overset{\prettyref{eq:varphi_0}}{=} \exp\pth{-\frac{\mu^2}{2}}\bE_{Z\sim \calN(0,1)} \qth{\sinh(\mu Z)} = 0, 
    \end{align*}
    we have by the product structure of $\bQ_n$ that
    \begin{align*}
    &\bE_{\bQ_n}\qth{ \prod_{i\in S} \tanh(\mu X_i) \prod_{i \in T}\tanh(\mu X_i) } \\
    &= \bE_{\bQ_n}\qth{\tanh(\mu X_{i_0})} \cdot \bE_{\bQ_n}\qth{ \prod_{i\in S\backslash \{i_0\}} \tanh(\mu X_i) \prod_{i \in T\backslash \{i_0\}}\tanh(\mu X_i) } = 0. 
    \end{align*}
    \item The second moment of $\prod_{i\in S} \tanh(\mu X_i)$ under $\bQ_n$ admits a simple expression. In fact, by the i.i.d. structure of $\bQ_n$, we have
    \begin{align*}
    \bE_{\bQ_n}\qth{ \prod_{i\in S} \tanh^2(\mu X_i) } = \left(\exp\pth{-\frac{\mu^2}{2}} \int_{\bR} \frac{\sinh^2(\mu x)}{\cosh(\mu x)}\varphi(x)\rmd x \right)^{|S|} =: f(\mu)^{|S|}. 
    \end{align*}
    We can also derive an upper bound of $f(\mu)$: 
    \begin{align}\label{eq:ub_f_mu}
        f(\mu) &= \exp\pth{-\frac{\mu^2}{2}} \int_{\bR} \pth{\cosh(\mu x) - \frac{1}{\cosh(\mu x)}}\varphi(x)\rmd x  \nonumber\\
        &= 1 - \exp\pth{-\frac{\mu^2}{2}} \int_{\bR} \frac{\varphi(x)}{\cosh(\mu x)}\rmd x \nonumber \\
        &\le 1 - \exp\pth{-\frac{\mu^2}{2}}  \frac{1}{\int_{\bR}\cosh(\mu x)\varphi(x) \rmd x } = 1 - \exp\pth{-\mu^2}. 
    \end{align}
    \item The expectation $\bE_{\pi\sim\Unif(S_n)}\qth{\prod_{i\in S} \frac{\theta_{\pi(i)}}{\mu}}$ depends only on $|S|$, and can be explicitly computed in the toy example. In fact, the generating function $\prod_{i=1}^n \pth{1+z\theta_i/\mu} = (1-z^2)^{n/2}$ gives
    \begin{align}\label{eq:balanced_expectation}
    \bE_{\pi\sim\Unif(S_n)}\qth{\prod_{i\in S} \frac{\theta_{\pi(i)}}{\mu}} = (-1)^{\ell/2} \frac{\binom{n/2}{\ell/2}}{\binom{n}{\ell}}\indc{\ell\text{ is even}} =: g_\ell, 
    \end{align}
    where $\ell = |S|$.
    
    In particular, since the mean vector is balanced, $g_1 = \bE_{\pi\sim\Unif(S_n)} \qth{\frac{\theta_{\pi(1)}}{\mu}} = 0$, and we have the simple inequality: 
    \begin{align*}
        g_{\ell}^2 \le \frac{1}{\binom{n}{\ell}}. 
    \end{align*}
\end{enumerate}

Based on the above observations, we are in a position to compute the second moment of \eqref{eq:overall-LR}: 
\begin{align*}
\chi^2(\bP_n \| \bQ_n) &= \bE_{\bQ_n}\qth{\pth{\frac{\rmd \bP_n}{\rmd \bQ_n}}^2} - 1 = \sum_{S\subseteq [n]} f(\mu)^{|S|}g_{|S|}^2 - 1 \stepa{=} \sum_{S\subseteq [n]: |S|\ge 2} f(\mu)^{|S|}g_{|S|}^2 \\
&\le \sum_{\ell=2}^n \binom{n}{\ell}\cdot \frac{f(\mu)^\ell}{\binom{n}{\ell}} \le \frac{f(\mu)^2}{1-f(\mu)} = \begin{cases}
O(\mu^4) &\text{if } \mu \le 1, \\
O(\exp(\mu^2)) &\text{if } \mu > 1. 
\end{cases}
\end{align*}
where in (a), the contribution of $S = \varnothing$ cancels with $-1$, and $g_1 = 0$ is used for $|S|=1$. This proves \eqref{eq:balanced_example}. 

We make some comments on the intuition behind this approach and the challenges moving forward.
The methods of moments and cumulants discussed in Section~\ref{sec:failure} rely on expanding the ratio $\varphi(x- \theta)/\varphi(x)$ in the basis of Hermite polynomials, which are orthogonal under $L^2(\varphi)$; this expansion facilitates taking expectations with respect to the standard Gaussian measure---as is done implicitly in deriving~\eqref{eq:method_of_moments} and~\eqref{eq:cumulants}---but computing the $\chi^2$ divergence between $\bP_n$ and $\bQ_n$ by comparing both measures to the standard Gaussian is inherently loose.

By contrast, the above approach works directly with the ratio $\varphi(x - \theta)/\varphi_0(x) = 1 + \Psi(x, \theta)$, where $\Psi(x, \theta) := \tanh(\mu x) \frac{\theta}{\mu}$.
This representation has two benefits with respect to an expansion in Hermite polynomials.
First, the product structure of $\bQ_n$ guarantees that for any $\theta_i \in \{\pm \mu\}$, the functions $\{\Pi_{i \in S} \Psi(X_i, \theta_i) : S \subseteq [n]\}$ are automatically orthogonal in $L^2(\bQ_n)$, so no change of measure is required.
Second, and more subtly, for any $x \in \bR^d$, the function $\theta \mapsto \Psi(x, \theta)$ is automatically mean-zero when $\theta \sim \mathrm{Unif}(\{\pm \mu\})$.
This property guarantees that the first-order term in~\eqref{eq:overall-LR} vanishes, which is crucial to obtaining the correct dependence on $\mu$, and indirectly ensures that the coefficients $g_\ell$ are typically small.
On the other hand, as \Cref{sec:failure} makes clear, large values of the moments $m_\alpha$ or cumulants $\kappa_\alpha$ lead to the failure of the moment and cumulant methods. Further comparisons between the above two approaches are discussed in \Cref{subsec:basis_expansion}. 

Several challenges arise when attempting to implement this idea in general.
There will typically not be a simple expression for the analogue of $\Psi(x, \theta)$ in the general case, and the coefficients no longer admit an explicit formula such as~\eqref{eq:balanced_expectation}.
We address these challenges in \Cref{subsec:general-case}.

\subsection{Proof of the general case}\label{subsec:general-case}
This section is devoted to the proof of the first part \eqref{eq:upper_bound_1} of \Cref{thm:main}. Motivated by \eqref{eq:Gaussian-LR} and \eqref{eq:varphi_0} in the warm-up example, we write
\begin{align}\label{eq:marginal}
\overline{P}(\rmd x) &= \frac{1}{n}\sum_{i=1}^n P_i(\rmd x), \\
P_i(\rmd x) &= \overline{P}(\rmd x) + \Psi_{i}(\rmd x). \label{eq:Psi}
\end{align}
It is clear that $\Psi_i \ll \overline{P}$, so the derivative $\frac{\rmd \Psi_i}{\rmd \overline{P}}$ exists. In addition, $\{\Psi_1,\cdots,\Psi_n\}$ have the following crucial ``doubly centered'' property:
\begin{align}
	    &\int \Psi_i(\rmd x) = \int (P_i(\rmd x) - \overline{P}(\rmd x)) = 0, \qquad \text{for all }i;\label{eq:int_zero} \\
    &\sum_{i=1}^n \Psi_i(\rmd x) = \sum_{i=1}^n (P_i(\rmd x) - \overline{P}(\rmd x)) = 0, \qquad \text{for }\overline{P}\text{-a.e. } x. \label{eq:sum_zero}
\end{align}
As we shall see, the first of these equations guarantees that products of functions in $\{\Psi_i\}_{i \in [n]}$ are orthogonal in $L^2(\bQ_n)$, whereas the second guarantees that $\bE_{I \sim  \mathrm{Unif}[n]} \Psi_I(\rmd x) = 0$.
Based on \eqref{eq:marginal} and \eqref{eq:Psi}, we express the likelihood ratio as 
\begin{align*}
\frac{\rmd \bP_n}{\rmd \bQ_n}(X) &= \bE_{\pi\sim\Unif(S_n)}\qth{\prod_{i=1}^n \frac{\rmd P_{\pi(i)}}{\rmd \overline{P}}(X_i)} = \bE_{\pi\sim\Unif(S_n)}\qth{\prod_{i=1}^n \pth{1+\frac{\rmd \Psi_{\pi(i)}}{\rmd \overline{P}}(X_i)}} \\
&= \sum_{S\subseteq [n]} \bE_{\pi\sim\Unif(S_n)}\qth{\prod_{i\in S} \frac{\rmd \Psi_{\pi(i)}}{\rmd \overline{P}}(X_i)}. 
\end{align*}
Similar to the warm-up example, the condition \eqref{eq:int_zero} ensures that the above summands are orthogonal under $\bQ_n = \overline{P}^{\otimes n}$. Therefore,
\begin{align}\label{eq:chi-squared-basis}
\chi^2(\bP_n \| \bQ_n) = \bE_{\bQ_n}\qth{\pth{\frac{\rmd \bP_n}{\rmd \bQ_n}}^2} - 1 = \sum_{S\subseteq [n], S\neq \varnothing} \bE_{\bQ_n}\qth{\pth{\bE_{\pi\sim\Unif(S_n)}\qth{\prod_{i\in S} \frac{\rmd \Psi_{\pi(i)}}{\rmd \overline{P}}(X_i)}}^2}.
\end{align}

To proceed, we fix any set $S$ with $|S|=\ell$; by symmetry we assume that $S = [\ell]$. For a fixed vector $(X_1,\cdots,X_\ell)$, construct a matrix $A = (a_{ij})\in \bR^{\ell\times n}$, with $a_{ij} = \frac{\rmd \Psi_j}{\rmd \overline{P}}(X_i)$. It is clear that 
\begin{align*}
    \bE_{\pi\sim\Unif(S_n)}\qth{\prod_{i\in S} \frac{\rmd \Psi_{\pi(i)}}{\rmd \overline{P}}(X_i)} = \frac{1}{
    \ell! \binom{n}{\ell}}\sum_{T\subseteq [n], |T| = \ell} \Perm(A_T), 
\end{align*}
where $A_T\in \mathbb{R}^{\ell\times \ell}$ is the submatrix of $A$ by taking the columns of $A$ with indices in $T$. In addition, thanks to \eqref{eq:sum_zero}, the matrix $A$ has all row sums zero. The key to upper bounding \eqref{eq:chi-squared-basis} is the following technical lemma. 

\begin{lemma}\label{lemma:hadamard}
Let $A=(a_{ij})\in \bR^{\ell\times n}$ be a real matrix with $1\le \ell \le n$ and all row sums being zero. Then the following inequality holds: 
\begin{align*}
\left| \frac{1}{\ell!}\sum_{T\subseteq [n], |T| = \ell} \Perm(A_T) \right| \le \sqrt{10\binom{n}{\ell}}\cdot \prod_{i=1}^{\ell} \pth{\frac{1}{n}\sum_{j=1}^n a_{ij}^2}^{\frac{1}{2}}. 
\end{align*}
\end{lemma}

In the special case $\ell = n$, \Cref{lemma:hadamard} coincides with the Hadamard-type inequality for permanents \cite{carlen2006inequality}, even without the condition that all row sums are zero. However, for $\ell < n$, this condition becomes crucial: without this condition, if $A$ is the all-ones matrix, the LHS of \Cref{lemma:hadamard} would be a much larger quantity $\binom{n}{\ell}$. The proof of \Cref{lemma:hadamard} turns out to be involved and is the central theme of \Cref{subsec:key-inequality}. 

Applying \Cref{lemma:hadamard} to \eqref{eq:chi-squared-basis}, for a given subset $S$ with $|S|=\ell\ge 2$ we get
\begin{align*}
&\bE_{\bQ_n}\qth{\pth{\bE_{\pi\sim\Unif(S_n)}\qth{\prod_{i\in S} \frac{\rmd \Psi_{\pi(i)}}{\rmd \overline{P}}(X_i)}}^2} \le \frac{10}{\binom{n}{\ell}} \bE_{\bQ_n}\qth{ \prod_{i=1}^{\ell} \pth{\frac{1}{n}\sum_{j=1}^n \frac{\rmd \Psi_j}{\rmd \overline{P}}(X_i)^2 }} \\
&= \frac{10}{\binom{n}{\ell}} \pth{\bE_{Z\sim \overline{P}}\qth{\frac{1}{n}\sum_{j=1}^n \frac{\rmd \Psi_j}{\rmd \overline{P}}(Z)^2}}^{\ell} = \frac{10}{\binom{n}{\ell}} \pth{ \frac{1}{n}\sum_{j=1}^n \chi^2(P_j\|\overline{P})}^{\ell} \le \frac{10}{\binom{n}{\ell}} \Capa^\ell. 
\end{align*}
In addition, for $\ell=1$ the LHS of \Cref{lemma:hadamard} is zero since all row sums are zero. Therefore, \eqref{eq:chi-squared-basis} gives that
\begin{align*}
\chi^2(\bP_n \| \bQ_n) \le \sum_{S\subseteq [n]: |S|\ge 2} \frac{10}{\binom{n}{|S|}} \Capa^{|S|} = \sum_{\ell=2}^n \binom{n}{\ell}\cdot \frac{10}{\binom{n}{\ell}} \Capa^{\ell} = 10\sum_{\ell=2}^n \Capa^{\ell}, 
\end{align*}
which is the first part \eqref{eq:upper_bound_1} of \Cref{thm:main}. 


\subsection{A key inequality}\label{subsec:key-inequality}
This section is devoted to the proof of \Cref{lemma:hadamard}, which consists of several steps. First we show that it suffices to prove \Cref{lemma:hadamard} with identical rows. Let $r_1,\cdots,r_\ell\in \calH$ be the rows of $A$, where $\calH = \{x\in \bR^n: \sum_{i=1}^n x_i = 0\}$ is a Hilbert space equipped with the Euclidean inner product. 
Writing the LHS of \Cref{lemma:hadamard} as a function of $(r_1, \dots, r_\ell)$, it is clear that
\begin{equation*}
	L(r_1, \dots, r_\ell) := \frac{1}{\ell!}\sum_{T\subseteq [n], |T| = \ell} \Perm(A_T)
\end{equation*}
is multilinear in $(r_1,\ldots,r_\ell)$. We invoke the following deep result due to S. Banach~\cite{banach1938homogene}. 

\begin{lemma}\label{lemma:banach}
Let $L(x_1,\cdots,x_n)$ be a symmetric multilinear form from a Hilbert space $(\calH, \langle \cdot , \cdot \rangle)$ to either $\bR$ or $\bC$. Suppose that
\begin{align*}
\sup\sth{ |L(x,x,\ldots,x)|: |x|\le 1 } \le M.
\end{align*}
Then it also holds that
\begin{align*}
\sup\sth{ |L(x_1,x_2,\ldots,x_n)|: |x_1|\le 1, \ldots, |x_n|\le 1 } \le M. 
\end{align*}
\end{lemma}

To apply \Cref{lemma:banach}, we note that the multilinear form $L$ in \Cref{lemma:hadamard} is clearly symmetric, and \Cref{lemma:hadamard} precisely asks for an upper bound on the operator norm of $L$. By \Cref{lemma:banach}, it suffices to establish \Cref{lemma:hadamard} when all rows of $A$ are identical, denoted by a vector $x\in \calH$; also, by scaling we may assume that $\sum_{i=1}^n x_i^2 = n$. In this case, the target quantity in \Cref{lemma:hadamard} becomes $e_\ell(x) := \sum_{S\subseteq [n]: |S|=\ell} \prod_{i\in S} x_i$, the elementary symmetric polynomial of $x=(x_1,\cdots,x_n)$. The central inequality of this section is summarized in the following theorem. 

\begin{thm}\label{thm:ESP}
For $0\le \ell \le n$, the following upper bound holds: 
\begin{enumerate}
    \item If $x\in \bC^n$ with $\sum_{i=1}^n x_i = 0$ and $\sum_{i=1}^n |x_i|^2 = n$, then (define $0^0 := 1$)
    \begin{align}\label{eq:complex_case}
       |e_\ell(x)|^2 \le \frac{n^n}{\ell^\ell (n-\ell)^{n-\ell}} < 3\sqrt{\ell+1}\cdot \binom{n}{\ell}. 
    \end{align}
    \item If in addition the condition $x\in \bR^n$ holds, then an improved upper bound is available: 
    \begin{align}\label{eq:real_case}
        |e_\ell(x)| \le \sqrt{10\binom{n}{\ell}}. 
    \end{align}
\end{enumerate}
\end{thm}

It is clear that the inequality \eqref{eq:real_case} in the real case proves \Cref{lemma:hadamard} after the reduction through \Cref{lemma:banach}. The complex inequality \eqref{eq:complex_case} will also be useful in \Cref{sec:permanents}. Again, we remark that the condition $\sum_{i=1}^n x_i = 0$ is crucial for \Cref{thm:ESP}, for $x=(1,\ldots,1)$ would lead to $e_\ell(x) = \binom{n}{\ell}$.
\Cref{thm:ESP} gives a bound on $|e_\ell(x)|$ when $e_1(x)$ and $e_2(x)$ are known.
It therefore adds to a recent line of work \cite{gopalan2014inequalities, meka2019pseudorandom, doron2020log, tao2023maclaurin} proving bounds on elementary symmetric polynomials given the values of two consecutive lower-order elementary symmetric polynomials.
For real $x\in \mathbb{R}^n$, this prior work can be used to deduce the bound
\begin{align*}
|e_\ell(x)|^2 \le C\cdot A^{\ell}\binom{n}{\ell}, 
\end{align*}
for some constant $A>1$. The best known constant is obtained via the differentiation trick in \cite{tao2023maclaurin}, giving
\begin{align*}
|e_\ell(x)|^2 \le \binom{n}{\ell}^2 \pth{\frac{\ell-1}{n-1}}^{\ell}, 
\end{align*}
corresponding to a constant $A=e$ by Stirling's approximation. This idea could also be applied to the complex case, where using the Schoenberg conjecture/Malamud--Pereira theorem \cite{malamud2005inverse,pereira2003differentiators} leads to the same upper bound.  

However, the application to our problem requires to have the best possible constant $A=1$ (note that the example in \eqref{eq:balanced_expectation} implies that $A<1$ is impossible). For example, the quantity $f(\mu)$ in \eqref{eq:ub_f_mu} could be arbitrarily close to $1$, and only $A=1$ makes the geometric sum $\sum_{\ell} A^{\ell} f(\mu)^{\ell}$ converge; we will also see similar scenarios in the application of \Cref{lemma:UB-individual-sum} in \Cref{sec:permanents}. As shown in \Cref{thm:ESP}, this constant turns out to be achievable (possibly at some tolerable price of $\mathrm{poly}(\ell)$) using the saddle point method. 

\begin{proof}[Proof of \Cref{thm:ESP} (First Part)]
It is clear that $e_\ell(x)$ is the coefficient of $z^\ell$ in $\prod_{k=1}^n (1+x_kz)$. By Cauchy's formula, for any $r>0$ we have
\begin{align*}
|e_\ell(x)| = \left|\frac{1}{2\pi \icom} \oint_{|z|=r}\frac{\prod_{k=1}^n (1+x_kz)}{z^{\ell}}\frac{\rmd z}{z}\right| \le \max_{|z|=r} \left|\frac{\prod_{k=1}^n (1+x_kz)}{z^{\ell}}\right|. 
\end{align*}
By AM-GM and the assumptions on $x$, 
\begin{align*}
\prod_{k=1}^n |1+x_kz|^2 &= \prod_{k=1}^n \pth{1+|x_k|^2 |z|^2 + 2\Re(x_k z)} \\
&\le \pth{\frac{1}{n}\sum_{k=1}^n \pth{1+|x_k|^2 |z|^2 + 2\Re(x_k z)}}^n = \pth{1+|z|^2}^n. 
\end{align*}
A combination of the previous steps gives 
\begin{align*}
|e_{\ell}(x)|^2 \le \inf_{r>0} \frac{(1+r^2)^n}{r^{2\ell}} = \frac{n^n}{\ell^{\ell}(n-\ell)^{n-\ell}}, 
\end{align*}
by choosing $r^2 = \frac{\ell}{n-\ell}$ for $1\le \ell\le n-1$, $r\to 0$ for $\ell = 0$, and $r\to\infty$ for $\ell = n$, establishing the first inequality of \eqref{eq:complex_case}. The second step of \eqref{eq:complex_case} simply follows from Stirling's approximation. 
\end{proof}

After the initial draft of this paper, we were informed that the same programs in \Cref{lemma:banach} and the first part of \Cref{thm:ESP} have already been employed in \cite{roos2015bobkov}. However, the second part \eqref{eq:real_case} of \Cref{thm:ESP} is still new, with a more involved proof. First, relying critically on a property of real-rooted polynomials (which fails for the complex case), we argue that the coordinates of the real maximizer $x$ of $|e_\ell(x)|$ can only take two values. The evaluation of $|e_\ell(x)|$ in the resulting simplified scenario is still challenging, and we apply a more careful saddle point analysis to arrive at \eqref{eq:real_case}. We defer the details to \Cref{append:saddlepoint}.
\section{Upper bound via matrix permanents}\label{sec:permanents}
The bounds presented in \Cref{sec:basis} are based on the identity
\begin{align*}
	\frac{\rmd \bP_n}{\rmd \bQ_n}(X) = \bE_{\pi\sim\Unif(S_n)}\qth{\prod_{i=1}^n \frac{\rmd P_{\pi(i)}}{\rmd \overline{P}}(X_i)} = \sum_{S\subseteq [n]} \bE_{\pi\sim\Unif(S_n)}\qth{\prod_{i\in S} \frac{\rmd \Psi_{\pi(i)}}{\rmd \overline{P}}(X_i)}. 
\end{align*}
As mentioned in \Cref{sec:techniques}, this expression may be viewed as an expansion of the matrix permanent in~\eqref{eq:lr_is_perm} around the all-ones matrix.
The benefit of this approach is that each of the terms in this expansion are orthogonal in $L^2(\bQ_n)$.
However, our proof relies crucially on \Cref{lemma:banach}, which is used to crudely upper bound each permanent appearing in the above expansion by the permanent of a matrix with identical rows.
This bound therefore ignores any additional structure in the likelihood ratio, and fails to give tight results when $\Capa > 1$.

In this section, we derive refined bounds for $\Capa > 1$ by treating the $\chi^2$ divergence between $\bP_n$ and $\bQ_n$ directly.
This gives proofs of the other upper bounds \eqref{eq:upper_bound_2} and \eqref{eq:upper_bound_3} in \Cref{thm:main}.
We also show how to use similar ideas to obtain \Cref{thm:two_mixtures}. 
In \Cref{subsec:permanent} we express the $\chi^2$ divergence between $\bP_n$ and $\bQ_n$ as a proper matrix permanent, and upper bound it in two ways, described in \Cref{subsec:entire_sum} and \Cref{subsec:individual_sum} respectively, both of which rely on the use of complex random variables.  

\subsection{Divergence as a matrix permanent}\label{subsec:permanent}
We first recall the expression given in \Cref{sec:techniques} for $\chi^2(\bP_n \| \bQ_n)$ as a matrix permanent.

\begin{lemma}\label{lemma:permanent}
Given $P_1,\cdots,P_n\in \calP$, define a matrix $A\in \bR^{n\times n}$ as
\begin{align}\label{eq:A-matrix}
A_{ij} = \frac{1}{n}\int \frac{\rmd P_i \rmd P_j}{\rmd \overline{P}}, \qquad \text{with } \overline{P} := \frac{1}{n}\sum_{i=1}^n P_i. 
\end{align}
Then under the setting of $\Cref{thm:main}$, the following identity holds:
\begin{align*}
\chi^2(\bP_n \| \bQ_n) = \frac{n^n}{n!}\Perm(A) - 1. 
\end{align*}
\end{lemma}

The following lemma summarizes some properties of the matrix $A$ in \eqref{eq:A-matrix}. 

\begin{lemma}\label{lemma:A-property}
The matrix $A$ in \eqref{eq:A-matrix} satisfies the following properties: 
\begin{enumerate}
    \item it is doubly stochastic and PSD; 
    \item its trace satisfies $\trace(A) \le 1 + \Capa$; 
    \item its eigenvalues satisfy $1=\lambda_1>\lambda_2\ge \ldots \ge \lambda_n\ge 0$, with leading eigenvector $v_1 = {\bf 1}$ and a spectral gap
    \begin{align*}
    1 - \lambda_2 \ge \frac{1}{\AffH}. 
    \end{align*}
\end{enumerate}
\end{lemma}

By \Cref{lemma:permanent} and \Cref{lemma:A-property}, the computation of the $\chi^2$ divergence $\chi^2(\bP_n\|\bQ_n)$ reduces to upper bounding the permanent of $A$. Moreover, in light of the van der Waerden conjecture/Egorychev--Falikman theorem \cite{van1926aufgabe,egorychev1981solution,falikman1981proof}, which states that $\Perm(A)\ge \frac{n!}{n^n}$ for all $n\times n$ doubly stochastic matrices $A$, showing that $\chi^2(\bP_n\|\bQ_n) = O(1)$ amounts to showing that $\Perm(A)$ is nearly as small as is possible for a doubly-stochastic matrix. Motivated by the eigenstructure of $A$ established in \Cref{lemma:A-property}, we write $A = UDU^\top$ for an orthogonal matrix $U\in \bR^{n\times n}$ and a diagonal matrix $D = \mathrm{diag}(\lambda_1,\cdots,\lambda_n)$ consisting of the eigenvalues. By a generalized Cauchy--Binet formula for permanents \cite[Theorem 2.1]{marcus1965generalized}, it holds that
\begin{align*}
	\Perm(A) = \sum_{\substack{(\ell_1,\ldots,\ell_n)\in \naturals^n \\ \sum_{i=1}^n \ell_i = n}} \frac{\Perm(U_{\ell_1,\ldots,\ell_n})^2}{\ell_1!\cdots \ell_n!}\lambda_1^{\ell_1}\cdots \lambda_n^{\ell_n}, 
\end{align*}
where $U_{\ell_1,\ldots,\ell_n}$ denotes the $n\times n$ matrix with the first column of $U$ appearing $\ell_1$ times, second column appearing $\ell_2$ times, and so on. By taking $\lambda_1=1$ into account, we obtain the equality
\begin{align}\label{eq:polynomial-eigenvalue}
\chi^2(\bP_n \| \bQ_n) +1 = \frac{n^n}{n!} \sum_{\ell=0}^n \sum_{\substack{(\ell_2,\ldots,\ell_n)\in \naturals^{n-1} \\ \sum_{i=2}^n \ell_i = \ell}} \frac{\Perm(U_{n-\ell,\ell_2,\ldots,\ell_n})^2}{(n-\ell)!\ell_2!\cdots \ell_n!}\lambda_2^{\ell_2}\cdots \lambda_n^{\ell_n} =: \sum_{\ell=0}^n S_\ell, 
\end{align}
which is a sum of homogeneous polynomials $S_\ell$ in $(\lambda_2,\ldots,\lambda_n)$ of total degree $\ell=0,1,\ldots,n$. 

To relate the permanent view \eqref{eq:polynomial-eigenvalue} with the doubly centered expansion in \Cref{sec:basis}, we refer to \Cref{sec:identity} for some identities between $S_\ell$ and related quantities. In particular, we note that the permanent $\Perm(U_{n-\ell,\ell,0,\ldots,0})$ is exactly a multiple of $e_\ell(u_2)$, the elementary symmetric polynomial of second column vector $u_2$ of $U$. Therefore, \Cref{thm:ESP} could be used to upper bound this quantity. In addition, \Cref{lemma:banach} reduces the general quantity $\Perm(U_{n-\ell,\ell_2,\ldots,\ell_n})$ to the case $\Perm(U_{n-\ell,\ell,0,\ldots,0})$ with identical columns, so the arguments in \Cref{sec:basis} leads to an upper bound of \eqref{eq:polynomial-eigenvalue} in terms of $\sum_{i=2}^n \lambda_i\le \Capa$ by \Cref{lemma:A-property}. However, this approach does not capture additional structures of $(\lambda_2,\dots,\lambda_n)$ such as the spectral gap. To this end, it might be a natural idea to establish pointwise upper bounds on $\Perm(U_{n-\ell,\ell_2,\ldots,\ell_n})$.
We have not succeeded in this approach and leave it as an open direction; instead, we choose to upper bound the individual sum $S_\ell$ or the entire sum $\sum_{\ell=0}^n S_\ell$. These form the topics of the subsequent sections.  

\subsection{Proof of \Cref{thm:main} via bounding the entire sum}\label{subsec:entire_sum}
In this section, we upper bound the entire sum $\sum_{\ell=0}^n S_{\ell}$ in \eqref{eq:polynomial-eigenvalue} in order to prove \Cref{thm:main}. The central result of this section is the following upper bound:

\begin{lemma}\label{lemma:UB-entire-sum}
Using the notation of \eqref{eq:polynomial-eigenvalue}, for any $\lambda_2,\dots,\lambda_n\ge 0$ it holds that
\begin{align*}
\sum_{\ell=0}^n S_{\ell} \le \sum_{\ell=0}^n \sum_{\substack{(\ell_2,\ldots,\ell_n)\in \naturals^{n-1} \\ \sum_{i=2}^n \ell_i = \ell}} \lambda_2^{\ell_2}\cdots \lambda_n^{\ell_n}. 
\end{align*}
\end{lemma}

Comparing this expression with \eqref{eq:polynomial-eigenvalue}, \Cref{lemma:UB-entire-sum} shows that replacing all coefficients in \eqref{eq:polynomial-eigenvalue} by $1$ gives an upper bound. Using the eigenvalue properties in \Cref{lemma:A-property}, this upper bound implies the upper bounds \eqref{eq:upper_bound_2} and \eqref{eq:upper_bound_3} of \Cref{thm:main}; see \Cref{append:proof_thm_main}. The proof of \Cref{lemma:UB-entire-sum} relies on the use of complex normal random variables. Recall that $z\sim \calCN(0,1)$ denotes $z=x+\icom y$ with independent $x,y\sim \calN(0,\frac{1}{2})$, and $z\sim \calCN(0,\Sigma)$ denotes a vector $z=x+\icom y$ with independent $x,y\sim \calN(0,\frac{\Sigma}{2})$ for \emph{real} PSD matrices $\Sigma$. It is known that the moments of $z\sim \calCN(0,1)$ are
\begin{align}\label{eq:complex-normal-moments}
\bE[z^m \bar{z}^n] = n!\mathbbm{1}_{m=n}. 
\end{align}
The following identity between matrix permanents and complex normal random vectors is observed in \cite[Lemma 2]{anari2017simply} and earlier in \cite{gurvits2003classical}.

\begin{lemma}\label{lemma:wick-formula}
For $P\in \bR^{m\times n}$ and $z\sim \calCN(0,I_n)$, the following identity holds:
\begin{align*}
\Perm(PP^\top) = \bE\qth{\prod_{i=1}^m |(Pz)_i|^2}. 
\end{align*}
\end{lemma}

\Cref{lemma:wick-formula} follows from the classical Isserlis' theorem \cite{isserlis1918formula} (or Wick's formula) for products of joint normal random variables, and we include a proof in \Cref{append:wick_formula_proof} for completeness. Now we present the proof of \Cref{lemma:UB-entire-sum}. 

\begin{proof}[Proof of \Cref{lemma:UB-entire-sum}]
Let $A = UDU^\top$ be the matrix in \Cref{lemma:permanent}, and $P = UD^{1/2}$. Then
\begin{align*}
\frac{n!}{n^n}\sum_{\ell=0}^n S_\ell = \Perm(A) = \Perm(PP^\top) = \bE\qth{\prod_{i=1}^n |(Pz)_i|^2}
\end{align*}
by \Cref{lemma:wick-formula}. Since
\begin{align*}
\sum_{i=1}^n |(Pz)_i|^2 =  \|Pz\|_2^2 = z^\sfH P^\top P z = z^\sfH Dz = \sum_{i=1}^n \lambda_i |z_i|^2, 
\end{align*}
the AM-GM inequality then gives that
\begin{align*}
\bE\qth{\prod_{i=1}^n |(Pz)_i|^2} &\le \bE\qth{\pth{\frac{1}{n}\sum_{i=1}^n |(Pz)_i|^2 }^n} = \bE\qth{\pth{\frac{1}{n}\sum_{i=1}^n \lambda_i |z_i|^2 }^n} \\
& \stepa{=} \frac{1}{n^n}\sum_{\substack{(\ell_1,\ldots,\ell_n)\in \naturals^n \\ \sum_{i=1}^n \ell_i = n}} \binom{n}{\ell_1,\ldots,\ell_n}\lambda_1^{\ell_1}\cdots \lambda_n^{\ell_n}\bE\qth{|z_1|^{2\ell_1}\cdots |z_n|^{2\ell_n}} \\
& \stepb{=} \frac{n!}{n^n}\sum_{\substack{(\ell_1,\ldots,\ell_n)\in \naturals^n \\ \sum_{i=1}^n \ell_i = n}} \lambda_1^{\ell_1}\cdots \lambda_n^{\ell_n}, 
\end{align*}
where (a) is the multinomial theorem, and (b) follows from \eqref{eq:complex-normal-moments}. The proof is complete. 
\end{proof} 

\subsection{Proof of \Cref{thm:two_mixtures} via bounding the individual sum}\label{subsec:individual_sum}
In this section, we upper bound the individual sum $S_\ell$ in \eqref{eq:polynomial-eigenvalue} in order to prove \Cref{thm:two_mixtures}. The central result of this section is the following upper bound: 
\begin{lemma}\label{lemma:UB-individual-sum}
Using the notation of \eqref{eq:polynomial-eigenvalue}, for any $\lambda_2,\cdots,\lambda_n\ge 0$ and $\ell=0,\ldots,n$ it holds that
\begin{align*}
 S_{\ell} \le 3\sqrt{\ell+1}\cdot \sum_{\substack{(\ell_2,\ldots,\ell_n)\in \naturals^{n-1} \\ \sum_{i=2}^n \ell_i = \ell}} \lambda_2^{\ell_2}\cdots \lambda_n^{\ell_n}. 
\end{align*}
\end{lemma}

Compared with \Cref{lemma:UB-entire-sum}, the result of \Cref{lemma:UB-individual-sum} has an additional $O(\sqrt{\ell})$ factor, due to the use of \Cref{thm:ESP} in the complex case. However, such an individual bound is crucial to deal with weighted sums $\sum_{\ell=0}^n w_\ell S_\ell$ in applications such as \Cref{thm:two_mixtures}, as witnessed by the next lemma. Recall the definitions of the neighboring mixtures $\bP_n$ and $\bP_n'$ in \Cref{thm:two_mixtures}. Similar to \eqref{eq:A-matrix} and \eqref{eq:polynomial-eigenvalue}, let the matrix $A\in \bR^{(n-1)\times (n-1)}$ be constructed based on $P_2,\ldots,P_n\in \calP$, and $S_\ell$ be the degree-$\ell$ homegeneous polynomial in its eigenvalues $(\lambda_2,\ldots,\lambda_{n-1})$. Applying a similar likelihood ratio computation to \Cref{lemma:permanent} leads to the following upper bound. 

\begin{lemma}\label{lemma:UB-empirical-bayes}
Under the above notations, the following inequality holds:
\begin{align*}
\TV(\bP_n, \bP_n')^2 \le \frac{1}{4}\int \frac{(\rmd \bP_n - \rmd \bP_n')^2}{\rmd \overline{P}^{\otimes n}} \le \frac{\Diam}{n}\sum_{\ell=1}^n \ell S_{\ell-1}. 
\end{align*}
\end{lemma}

We defer the proofs of \Cref{lemma:UB-empirical-bayes} and \Cref{thm:two_mixtures} to \Cref{append:proof_UB-empirical-bayes} and \Cref{append:proof_two_mixtures}, respectively, and present the proof of \Cref{lemma:UB-individual-sum} relying again on complex normal random variables. Recalling the notation $A = UDU^\top$ from \Cref{subsec:permanent}, let $\widetilde{U}\in \bR^{n\times (n-1)}$ be the matrix $U$ with its first column (which is $\frac{\bf 1}{\sqrt{n}}$ by \Cref{lemma:A-property}) removed, and $\widetilde{D} = \mathrm{diag}(\lambda_2,\ldots,\lambda_n)$ with the first eigenvalue $\lambda_1 = 1$ removed. The following lemma establishes an alternative expression of $S_\ell$. 

\begin{lemma}\label{lemma:wick-formula-2}
Let $\widetilde{P} = \widetilde{U}\widetilde{D}^{1/2}\in \bR^{n\times (n-1)}$ and $z\sim \calCN(0,I_{n-1})$. Then for $\ell=0,\cdots,n$, 
\begin{align*}
S_{\ell} = n^\ell \frac{(n-\ell)! }{n!} \cdot \bE\qth{|e_\ell(\widetilde{P}z)|^2}, 
\end{align*}
where as usual $e_\ell$ denotes the elementary symmetric polynomial of order $\ell$. 
\end{lemma}

The proof of \Cref{lemma:wick-formula-2} again follows from the Isserlis' theorem and several permutation identities in \Cref{lemma:RST}, and is deferred to \Cref{append:wick-formula-2-proof}. Now we present the proof of \Cref{lemma:UB-individual-sum}. 
\begin{proof}[Proof of \Cref{lemma:UB-individual-sum}]
Since the matrix $U$ is orthogonal with the first column proportional to ${\bf 1}$, all columns of $\widetilde{U}$ are orthogonal to ${\bf 1}$. Consequently, the vector $\widetilde{P}z=\widetilde{U}\widetilde{D}^{1/2}z$ sums to zero, so \eqref{eq:complex_case} in \Cref{thm:ESP} gives
\begin{align*}
|e_{\ell}(\widetilde{P}z)|^2 \le 3\sqrt{\ell+1}\binom{n}{\ell}\pth{\frac{1}{n}\sum_{i=1}^n \left|(\widetilde{P}z)_i\right|^2}^{\ell}.
\end{align*}
Since
\begin{align*}
\sum_{i=1}^n \left|(\widetilde{P}z)_i\right|^2 = z^\sfH \widetilde{P}^\top \widetilde{P} z = z^\sfH \widetilde{D} z = \sum_{i=2}^{n} \lambda_{i}|z_{i-1}|^2, 
\end{align*}
by \Cref{lemma:wick-formula-2} and \eqref{eq:complex-normal-moments} we get
\begin{align*}
S_{\ell} &\le n^{\ell}\frac{(n-\ell)!}{n!}\cdot 3\sqrt{\ell+1}\binom{n}{\ell}\bE\qth{\pth{\frac{1}{n}\sum_{i=2}^{n} \lambda_i |z_{i-1}|^2}^{\ell}} \\
&= \frac{3\sqrt{\ell+1}}{\ell!} \sum_{\substack{(\ell_2,\ldots,\ell_n)\in \naturals^{n-1} \\ \sum_{i=2}^n \ell_i = \ell}} \binom{\ell}{\ell_2,\ldots,\ell_n}\lambda_2^{\ell_2}\cdots \lambda_n^{\ell_n} \bE\qth{|z_1|^{2\ell_2}\cdots |z_{n-1}|^{2\ell_{n}}} \\
&= 3\sqrt{\ell+1}\sum_{\substack{(\ell_2,\ldots,\ell_n)\in \naturals^{n-1} \\ \sum_{i=2}^n \ell_i = \ell}} \lambda_2^{\ell_2}\cdots \lambda_n^{\ell_n}, 
\end{align*}
which is the claimed upper bound. 
\end{proof}


\section{Discussion}\label{sec:discussion}

\subsection{Tightness of upper bounds}\label{subsec:tightness}

In this section we discuss the tightness of the upper bounds in \Cref{thm:main}. Recall that the upper bound \eqref{eq:upper_bound_1} is quadratic in $\Capa$ when $\Capa$ is small, and the upper bound \eqref{eq:upper_bound_2} or \eqref{eq:upper_bound_3} is a power function of $\AffH$ when $\Capa$ is large. The next lemma shows that this is essentially the right behavior of $\chi^2(\bP_n \| \bQ_n)$ for every family $\calP$. 

\begin{lemma}\label{lemma:tightness}
In the setting of \Cref{thm:main}, the following lower bound holds:
\begin{align*}
\sup_{n}\sup_{P_1,\ldots,P_n\in \calP} \chi^2(\bP_n \| \bQ_n) \ge \sup_{n}\sup_{P_1,\ldots,P_n\in \calP} \frac{1}{\sqrt{1-\lambda_2(A)^2}} - 1, 
\end{align*}
where $A= A(P_1,\dots,P_n)$ is defined in \eqref{eq:A-matrix}, and $\lambda_2(A)$ denotes the second largest eigenvalue of $A$. In particular, it holds that
\begin{align*}
\sup_{n}\sup_{P_1,\ldots,P_n\in \calP} \chi^2(\bP_n \| \bQ_n) \ge \max\sth{ \frac{\DiamH^2}{8}, \frac{\AffH^{1/4}}{\sqrt{2}}-1 }. 
\end{align*}
\end{lemma}

Combined with \Cref{thm:main}, \Cref{lemma:tightness} shows that when the capacity of the family $\calP$ is small, we have $\DiamH^2 \lesssim \sup_{n,P_1,\dots,P_n}\chi^2(\bP_n \| \bQ_n)\lesssim \Capa^2$, which is often tight as $\Capa\le \Diam \approx \DiamH$ in view of the fact that most $f$-divergences are locally $\chi^2$-like \cite[Chapter 7.10]{polyanskiy2024information}. At the other extreme, when the family $\calP$ is rich, we have $\AffH^{\Omega(1)}-1\le \sup_{n,P_1,\dots,P_n}\chi^2(\bP_n \| \bQ_n) \le (e\AffH)^{\Capa}-1$. Therefore, up to a gap in the exponent, the maximum $H^2$ singularity $\AffH$ plays a central role when it is large. The proof of \Cref{lemma:tightness} computes the variance of a linear test function $\frac{1}{n}\sum_{i=1}^n f(X_i)$ under $\bP_n$ and $\bQ_n$. 

The next natural question is on the tightness of the exponent $\Capa$ in the upper bound \eqref{eq:upper_bound_2}. Using the permanent representation of the $\chi^2$ divergence in \Cref{lemma:permanent}, we can equivalently ask whether the following bound in \Cref{sec:permanents} is tight: 
\begin{align}\label{eq:permanent-upper-bound}
\Perm(A) \le \frac{n!}{n^n}\pth{\frac{1}{\Delta}}^{1+\sfC},
\end{align}
where $A\in \mathbb{R}^{n\times n}$ is PSD and doubly stochastic, with $\trace(A)\le 1+\sfC$ and $1-\lambda_2(A)\ge \Delta$. Our next lemma shows that the permanent upper bound \eqref{eq:permanent-upper-bound} is essentially tight if we only make use of the trace and spectral gap of $A$. 

\begin{lemma}\label{lemma:tightness-perm}
There exist absolute constants $r,r',\sfC_0>0$ and $\Delta_0 \in (0,1)$ such that for any $\sfC \ge \sfC_0$ and $\Delta \le \Delta_0$, one may find some $n\in \mathbb{N}$ and $A\in \mathbb{R}^{n\times n}$ satisfying: 1) $A$ is PSD and doubly stochastic; 2) $\trace(A)\le 1+\sfC$; 3) $1-\lambda_2(A)\ge \Delta$; and 4)
\begin{align*}
\Perm(A) \ge \frac{n!}{n^n} \pth{\frac{1}{\Delta}}^{r\sfC}. 
\end{align*}
Similarly, one may also find a family $\calP$ of distributions with $\Capa \le \sfC$, $\AffH \le \Delta^{-1}$, and
\begin{align*}
\sup_{n}\sup_{P_1,\dots,P_n\in \calP} \chi^2(\bP_n \| \bQ_n) \ge \pth{\frac{1}{\Delta}}^{r' \sfC} - 1. 
\end{align*}
\end{lemma}

By \Cref{lemma:tightness-perm}, the exponent $\sfC$ in both the permanent upper bound \eqref{eq:permanent-upper-bound} and the $\chi^2$ divergence upper bound \eqref{eq:upper_bound_2} in \Cref{thm:main} is tight in the worst case, up to multiplicative constants. However, as opposed to \Cref{lemma:tightness}, the statement of \Cref{lemma:tightness-perm} is not pointwise in $\calP$; in other words, this does \emph{not} mean that the upper bounds of \Cref{thm:main} are always tight for a specific class $\calP$. For example, for the Gaussian family (cf. point 1 of \Cref{cor:specific_families}) with $\mu > 1$, \Cref{thm:main} and \Cref{lemma:tightness} only imply that $\exp(\Omega(\mu^2)) \le \sup_{n,P_1,\dots,P_n} \chi^2(\bP_n \| \bQ_n) \le \exp(O(\mu^3))$, still exhibiting a gap on the exponent. An intuitive reason is that, for specific $\calP$, further properties of the matrix $A$ (besides the trace and spectral gap) could potentially be exploited to lead to better results.


\subsection{The permutation channel}
We also discuss the implications of our results on the capacity of the \emph{noisy permutation channel} in information theory. Motivated by mobile networks and DNA coding systems, the noisy permutation channel \cite{makur2020coding} applies a uniformly random permutation to the outputs of a usual communication channel. Mathematically, let $\calP = (\sfK_x)_{x\in \calX}$ be a class of conditional distributions $\sfK_x := P_{Z|X=x}$ (i.e., a channel), and $(Z_1,\dots,Z_n)$ be the channel output of an input sequence $(X_1,\dots,X_n)$. The final output $(Y_1,\dots,Y_n)$ of the permutation channel is a uniformly random permutation of $(Z_1,\dots,Z_n)$, i.e., $Y_i = Z_{\pi(i)}$ for $\pi\sim \Unif(S_n)$. To understand and design statistically optimal coding schemes for the permutation channel, a key task is to determine the channel capacity
\begin{align}\label{eq:channel_capacity}
    C_n = \max_{p(x^n)} I(X^n; Y^n), 
\end{align}
where the maximum is over all possible distributions of $X^n$. Using our result that the distribution $P_{Y^n|X^n}$ is approximately $\pth{\frac{1}{n}\sum_{i=1}^n \sfK_{X_i}}^{\otimes n}$, we obtain the following bounds of $C_n$. 

\begin{lemma}\label{lemma:permutation_channel}
Let $\delta(\calP):= \Capa(1+\AffH)$, the following bounds hold for $C_n$ in \eqref{eq:channel_capacity}: 
\begin{align*}
 \sup_{\varepsilon>0} \min\sth{ \frac{n\varepsilon^2}{4}, \log M_{\mathrm{H}}(\calP_n,\varepsilon)} - \frac{\delta(\calP)}{2} - \log 2 \le C_n \le \inf_{\varepsilon>0} \pth{n\varepsilon^2 + \log N_{\mathrm{KL}}(\calP_n, \varepsilon) } + \delta(\calP). 
\end{align*}
Here $\calP_n := \frac{1}{n}(\calP+\dots+\calP)$ with the set addition $A+B:=\{a+b:a\in A, b\in B\}$, and 
\begin{align*}
N_{\mathrm{KL}}(\calP, \varepsilon) &:= \min\sth{m: \inf_{\calP_0\subseteq \calP: |\calP_0|=m} \sup_{P\in \calP} \min_{P_0\in \calP_0} \KL(P\|P_0)\le \varepsilon^2 }, \\
M_{\mathrm{H}}(\calP, \varepsilon) &:= \max\sth{m: \sup_{\calP_0\subseteq \calP: |\calP_0|=m} \min_{P,P'\in \calP_0: P\neq P'} H^2(P, P')\ge \varepsilon^2 }
\end{align*}
are the covering and packing numbers under the KL divergence and Hellinger distance, respectively. 
\end{lemma}

In the above result, the quantity $\inf_{\varepsilon>0} \pth{n\varepsilon^2 + \log N_{\mathrm{KL}}(\mathrm{conv}(\calP), \varepsilon) }$ is an entropic upper bound on the \emph{minimax redundancy} of $\calP_n^{\otimes n}$ \cite{haussler1997mutual,yang1999information}, and the quantity $\sup_{\varepsilon>0} \min\{n\varepsilon^2/4, \log M_{\mathrm{H}}(\calP_n,\varepsilon)\}$ is an entropic lower bound for redundancy \cite{haussler1997mutual}. Therefore, \Cref{lemma:permutation_channel} shows that an i.i.d. approximation of $P_{Y^n|X^n}$ is quantitatively accurate up to an $\calO_\calP(1)$ additive factor. When $\calP$ consists of $d$ linearly independent discrete distributions with strictly positive pmfs, both entropic bounds scale as $\frac{d-1}{2}\log n$, so \Cref{lemma:permutation_channel} recovers the capacity bound $C_n = \frac{d-1}{2}\log n + O(1)$ in \cite{makur2020coding,tang2023capacity}. Moreover, unlike the method-of-types technique in \cite{tang2023capacity} which requires a finite class $\calP$ of discrete distributions, \Cref{lemma:permutation_channel} gives meaningful results even if $|\calP|=\infty$ or $\calP$ is a family of continuous distributions.

\subsection{General basis expansions}\label{subsec:basis_expansion}
As mentioned in Section~\ref{sec:basis}, in the Gaussian case where $P_i = \calN(\theta_i, 1)$, the methods of moments and cumulants described in Section~\ref{sec:failure} rely on the expansion $\varphi(x- \theta)/\varphi(x) = 1 + \sum_{k \geq 1} h_k(x) \frac{\theta^k}{k!}$, where $\varphi$ is the density of $\calN(0, 1)$ and $h_k$ is the $k$th Hermite polynomial.
By contrast, the doubly centered expansion is based on the representation $\varphi(x - \theta)/\varphi_0(x) = 1 + \Psi(x, \theta)$, where $\varphi_0(x) = \frac 1n \sum_{i=1}^n \varphi(x - \theta_i)$ denotes the marginal density of each coordinate under $\bP_n$ and $\bQ_n$.
This representation possesses the important property that the functions $(x_1, \dots, x_n) \mapsto \prod_{i \in S} \Psi(x_i, \theta_i)$ are orthogonal with respect to $L^2(\bQ_n)$ and $\frac 1n \sum_{i=1}^n \Psi(x, \theta_i) = 0$.

In principle, this expansion could be developed further as in the Hermite case by writing 
\begin{equation}\label{eq:full_expansion}
	\frac{\varphi(x - \theta)}{\varphi_0(x)} = 1 + \sum_{k \geq 1} \psi_k(x) g_k(\theta)
\end{equation}
for some functions $\psi_k$ which are orthogonal in $L^2(\varphi_0)$ and some suitable coefficient functions $g_k$.
Such an expansion would still possess the doubly centered property described above: indeed, averaging both sides of~\eqref{eq:full_expansion} over the mixing measure shows that in such an expansion, $\frac 1n \sum_{i=1}^n g_k(\theta_i) = 0$ always holds.
However, this development seems to offer no benefits in the setting of this work.
The intuitive explanation of this fact is that $\bP_n$ and $\bQ_n$ differ starting from the second moment, so the specific choices of the higher-order basis functions become unimportant (where every choice works).

Nevertheless, we conjecture that~\eqref{eq:full_expansion} may be useful in more general scenarios.
For example, if $\bP_n = \nu_\bP \star \calN(0, I_n)$ for a model in which $\bE_{\vartheta \sim \bP} \qth{\prod_{i=1}^n g_{\alpha_i}(\vartheta_i)} =0$ for all multi-indices $\alpha$ satisfying $|\alpha| \leq K$, then we anticipate that a full basis expansion as in~\eqref{eq:full_expansion} can be used to take advantage of additional cancellations if it is possible to give good bounds on $g_k(\theta)$ for large $k$. 

\paragraph{Acknowledgments.} We are grateful to Terence Tao for suggesting the use of the saddle point method for proving \Cref{thm:ESP}, to Cun-Hui Zhang for introducing to us the question of identifying the least favorable prior in \Cref{lemma:gaussian-seq-model} and the reference \cite{zhang2012minimax}, to Yury Polyanskiy for pointing us to the reference \cite{tang2023capacity}, to Bero Roos for telling us the references \cite{bobkov2005generalized,roos2015bobkov}, and to Yunzi Ding and Cheng Mao for discussions about an earlier version of this project. JNW is partially supported by a Sloan Research Fellowship and NSF grant DMS-2210583.

\appendix

\section{Useful identities}\label{sec:identity}
We discuss several identities involving permutations and matrix permanents which will be useful in the proofs of several results. Fix an integer $\ell\in [n]$ and $n$ distributions $P_1,\cdots,P_n\in \calP$. The first quantity appears in \eqref{eq:chi-squared-basis} of the doubly centered expansion in \Cref{sec:basis}, and is formally defined as
\begin{align}\label{eq:R_ell}
R_{\ell} = \bE_{X_1,\ldots,X_{\ell} \sim \overline{P}}\qth{\pth{ \bE_{\pi\sim \Unif(S_n)}\qth{\prod_{i=1}^{\ell} \frac{\rmd (P_{\pi(i)} - \overline{P})}{\rmd \overline{P}}(X_i)} }^2}, 
\end{align}
with $\overline{P} = \frac{1}{n}\sum_{i=1}^n P_i$ as usual. The second quantity is the degree-$\ell$ homogeneous polynomial $S_{\ell}$ in \eqref{eq:polynomial-eigenvalue} of the matrix permanent approach in \Cref{sec:permanents}: for the matrix $A\in \bR^{n\times n}$ defined in \eqref{eq:A-matrix} and its eigen-decomposition $A = UDU^\top$ with $D = \mathrm{diag}(\lambda_1,\cdots,\lambda_n)$, the quantity $S_{\ell}$ is defined as
\begin{align}\label{eq:S_ell}
S_{\ell} = \frac{n^n}{n!}\sum_{\substack{(\ell_2,\ldots,\ell_n)\in \naturals^{n-1} \\ \sum_{i=2}^n \ell_i = \ell}} \frac{\Perm(U_{n-\ell,\ell_2,\ldots,\ell_n})^2}{(n-\ell)!\ell_2!\cdots \ell_n!}\lambda_2^{\ell_2}\cdots \lambda_n^{\ell_n}. 
\end{align}
Here we recall that $U_{\ell_1,\ldots,\ell_n}$ denotes the $n\times n$ matrix with the first column of $U$ appearing $\ell_1$ times, second column appearing $\ell_2$ times, and so on. The last quantity $T_{\ell}$ is a useful intermediate quantity in our proof: let $\overline{A} = A - \frac{1}{n}{\bf 1}{\bf 1}^\top$ be the centered version of $A$, and define
\begin{align}\label{eq:T_ell}
T_{\ell} = \sum_{\substack{S,S'\subseteq [n] \\ |S|=|S'|=\ell }} \Perm(\overline{A}_{S,S'}), 
\end{align}
where $\overline{A}_{S,S'}$ denotes the submatrix of $\overline{A}$ with row indices in $S$ and column indices in $S'$. 

The following lemma shows that, within constant multiples of each other, the above quantities are essentially the same. 

\begin{lemma}\label{lemma:RST}
For the quantities defined in \eqref{eq:R_ell}--\eqref{eq:T_ell}, the following identity holds:
\begin{align*}
S_{\ell} = \binom{n}{\ell} \cdot R_{\ell} = \frac{(n-\ell)!}{n!}n^{\ell} \cdot T_{\ell}. 
\end{align*}
\end{lemma}
\begin{proof}
We first prove the relation between $R_{\ell}$ and $T_{\ell}$. By introducing an independent copy $\pi'\sim \Unif(S_n)$ of $\pi$, for \eqref{eq:R_ell} it holds that
\begin{align*}
R_{\ell} &= \bE_{\pi,\pi'\sim \Unif(S_n)} \sth{ \bE_{X_1,\ldots,X_{\ell} \sim \overline{P}} \qth{ \prod_{i=1}^{\ell} \frac{\rmd (P_{\pi(i)} - \overline{P})}{\rmd \overline{P}}(X_i)\prod_{i=1}^{\ell} \frac{\rmd (P_{\pi'(i)} - \overline{P})}{\rmd \overline{P}}(X_i)} } \\
&= \bE_{\pi,\pi'\sim \Unif(S_n)} \qth{ \prod_{i=1}^{\ell} \pth{\int \frac{\rmd (P_{\pi(i)} - \overline{P})\times \rmd (P_{\pi'(i)} - \overline{P})}{\rmd \overline{P}} } } \\
&= \bE_{\pi,\pi'\sim \Unif(S_n)} \qth{ \prod_{i=1}^{\ell}\pth{\int \frac{\rmd P_{\pi(i)} \rmd P_{\pi'(i)}}{\rmd \overline{P}} - 1 }} \\
&\overset{\prettyref{eq:A-matrix}}{=} \bE_{\pi,\pi'\sim \Unif(S_n)} \qth{\prod_{i=1}^{\ell}(nA_{\pi(i),\pi'(i)}-1)} \\
&= n^{\ell} \cdot \bE_{\pi,\pi'\sim \Unif(S_n)} \qth{\prod_{i=1}^{\ell}\overline{A}_{\pi(i),\pi'(i)}} \\
&\stepa{=} \frac{n^{\ell}}{\binom{n}{\ell}^2 \ell!} \sum_{\substack{S,S'\subseteq [n] \\ |S|=|S'|=\ell }} \Perm(\overline{A}_{S,S'}) = \frac{n^{\ell}}{\binom{n}{\ell}^2 \ell!}\cdot T_{\ell}, 
\end{align*}
where (a) uses $\bP(\pi([\ell]) = S) = \bP(\pi'([\ell])=S') = \binom{n}{\ell}^{-1}$ for all $S,S'$ of size $\ell$ and averages over $\ell!$ bijections between $S$ and $S'$. 

Next we prove the relation between $S_\ell$ and $T_\ell$. Fix any $t\in \mathbb{R}$, let $D_t := \mathrm{diag}(\lambda_1,t\lambda_2,\ldots,t\lambda_n)$. By the Cauchy--Binet formula for permanents \cite[Theorem 2.1]{marcus1965generalized}, it holds that
\begin{align*}
\Perm(UD_tU^\top) &= \sum_{\substack{(\ell_1,\ldots,\ell_n)\in \naturals^n \\ \sum_{i=1}^n \ell_i = n}} \frac{\Perm(U_{\ell_1,\ldots,\ell_n})^2}{\ell_1!\cdots \ell_n!}\lambda_1^{\ell_1}(t\lambda_2)^{\ell_2}\cdots (t\lambda_n)^{\ell_n} \\
&= \sum_{\ell=0}^n \sum_{\substack{(\ell_2,\ldots,\ell_n)\in \naturals^{n-1} \\ \sum_{i=2}^n \ell_i = \ell}} \frac{\Perm(U_{n-\ell,\ell_2,\ldots,\ell_n})^2}{(n-\ell)!\ell_2!\cdots \ell_n!}(t\lambda_2)^{\ell_2}\cdots (t\lambda_n)^{\ell_n} = \frac{n!}{n^n}\sum_{\ell=0}^n S_{\ell} t^{\ell} 
\end{align*}
is a polynomial in $t$. An alternative expression of $\Perm(UD_tU^\top)$ is also possible: by the structure of leading eigenvalue and eigenvector in \Cref{lemma:A-property}, $UD_tU^\top = U(tD - (t-1)\mathrm{diag}(1,0,\ldots,0))U^\top = tA - \frac{t-1}{n}J = t\overline{A} + \frac{J}{n}$ for $J := {\bf 1}{\bf 1}^\top$ being the all-ones matrix. Consequently,
\begin{align*}
\Perm(UD_tU^\top) &= \Perm\pth{ t\overline{A} + \frac{J}{n} } \\
&= \sum_{S,S'\subseteq [n]: |S| = |S'|} \Perm\pth{(t\overline{A})_{S,S'}}\Perm\pth{(\frac{J}{n})_{[n]\backslash S, [n]\backslash S'}} \\
&=\sum_{\ell=0}^n  t^{\ell} \frac{(n-\ell)!}{n^{n-\ell}}\sum_{\substack{S,S'\subseteq [n] \\ |S|=|S'|=\ell }} \Perm(\overline{A}_{S,S'}) = \sum_{\ell=0}^n \frac{(n-\ell)!}{n^{n-\ell}}T_{\ell} t^{\ell}. 
\end{align*}
The above two expressions must give the same polynomial in $t$, so $S_{\ell} = \frac{(n-\ell)!}{n!}n^{\ell}\cdot T_{\ell}$. 
\end{proof}
\ifthenelse{\boolean{arxiv}}{\section{Proof of \Cref{thm:ESP}}\label{append:saddlepoint}}{\section{Proof of Theorem 4.3}\label{append:saddlepoint}}
Since the complex case \eqref{eq:complex_case} has been established in \Cref{subsec:key-inequality}, in this section we are devoted to the proof of \eqref{eq:real_case} in the real case. We split the analysis into several steps. 

\subsection{The case of binary support}
We first consider the special case where the coordinates of $x\in \mathbb{R}^n$ can only take two values. Thanks to the assumptions $\sum_{i=1}^n x_i = 0$ and $\sum_{i=1}^n x_i^2 = n$, the only cases are 
\begin{align*}
    x^{(k)} = \Big( \underbrace{\sqrt{\frac{k}{n-k}}, \ldots, \sqrt{\frac{k}{n-k}}}_{n-k\text{ copies}}, \underbrace{-\sqrt{\frac{n-k}{k}}, \ldots, -\sqrt{\frac{n-k}{k}}}_{k\text{ copies}} \Big), \quad \text{for some } k = 1, \cdots, n-1. 
\end{align*}
The target of this section is to show that
\begin{align}\label{eq:target_binary}
|e_{\ell}(x^{(k)})| \le \sqrt{10\binom{n}{\ell}} \qquad \text{for all }k,\ell \in [n-1].
\end{align}
Fix the choice of $r = \sqrt{\ell/(n-\ell)}$. By Cauchy's formula,
\begin{align*}
|e_{\ell}(x^{(k)})| &= \left|\frac{1}{2\pi \icom}\oint_{|z|=r} \pth{1+z\sqrt{\frac{k}{n-k}}}^{n-k}\pth{1-z\sqrt{\frac{n-k}{k}}}^k \frac{\rmd z}{z^{\ell+1}} \right| \\
&\stepa{\le} \frac{1}{2\pi r^{\ell}} \int_0^{2\pi} \pth{1+\frac{k}{n-k}r^2 + 2r\sqrt{\frac{k}{n-k}}\cos\theta}^{\frac{n-k}{2}} \pth{1+\frac{n-k}{k}r^2-2r\sqrt{\frac{n-k}{k}}\cos\theta}^{\frac{k}{2}}\rmd \theta \\
&\stepb{=} \frac{1}{\pi r^\ell} \int_{-1}^1 \pth{1+\frac{k}{n-k}r^2 + 2r\sqrt{\frac{k}{n-k}}t}^{\frac{n-k}{2}} \pth{1+\frac{n-k}{k}r^2-2r\sqrt{\frac{n-k}{k}}t}^{\frac{k}{2}}\frac{\rmd t}{\sqrt{1-t^2}}, 
\end{align*}
where (a) uses the triangle inequality and a change of variable $z=re^{\icom \theta}$ with $\theta\in [0,2\pi)$, and (b) applies another change of variable $t=\cos\theta\in [-1,1]$. We break the integral into two parts:
\begin{itemize}
    \item Fixing some $\delta\in \qth{0,\frac{1}{2}}$ to be specified later, define
    \begin{align*}
        A_1(\delta):= \int_{1-\delta \le |t|\le 1} \pth{1+\frac{k}{n-k}r^2 + 2r\sqrt{\frac{k}{n-k}}t}^{\frac{n-k}{2}} \pth{1+\frac{n-k}{k}r^2-2r\sqrt{\frac{n-k}{k}}t}^{\frac{k}{2}}\frac{\rmd t}{\sqrt{1-t^2}}.
    \end{align*}
    Note that by the AM-GM inequality, for all $t\in [-1,1]$ it holds that
    \begin{align*}
    &\pth{1+\frac{k}{n-k}r^2 + 2r\sqrt{\frac{k}{n-k}}t}^{\frac{n-k}{2}} \pth{1+\frac{n-k}{k}r^2-2r\sqrt{\frac{n-k}{k}}t}^{\frac{k}{2}} \\
    & \le \qth{\frac{n-k}{n} \pth{1+\frac{k}{n-k}r^2 + 2r\sqrt{\frac{k}{n-k}}t} + \frac{k}{n}\pth{1+\frac{n-k}{k}r^2-2r\sqrt{\frac{n-k}{k}}t}}^{\frac{n}{2}} \\
    &= (1+r^2)^{\frac{n}{2}}. 
    \end{align*}
    Therefore,
    \begin{align}\label{eq:A1-delta}
        A_1(\delta) \le (1+r^2)^{\frac{n}{2}} \int_{1-\delta \le |t|\le 1} \frac{\rmd t}{\sqrt{1-t^2}} = 2(1+r^2)^{\frac{n}{2}} \arccos(1-\delta)\le \frac{2\pi}{3}(1+r^2)^{\frac{n}{2}}\sqrt{2\delta}, 
    \end{align}
    where the last step follows from $1-\cos t\ge \frac{1}{2}\pth{\frac{3t}{\pi}}^2$ for all $t\in \qth{0,\frac{\pi}{3}}$.
    \item For the remaining integral 
    \begin{align*}
        A_2(\delta) := \int_{-1+\delta}^{1-\delta} \pth{1+\frac{k}{n-k}r^2 + 2r\sqrt{\frac{k}{n-k}}t}^{\frac{n-k}{2}} \pth{1+\frac{n-k}{k}r^2-2r\sqrt{\frac{n-k}{k}}t}^{\frac{k}{2}}\frac{\rmd t}{\sqrt{1-t^2}}, 
    \end{align*}
    we note that $\sqrt{1-t^2}\ge \sqrt{1-(1-\delta)^2}\ge \sqrt{3\delta /2}$ for $\delta\in \qth{0,\frac{1}{2}}$. In addition, 
    \begin{align*}
        t_- := -\frac{1+\frac{k}{n-k}r^2}{2r\sqrt{\frac{k}{n-k}}}\le -1, \qquad t_+ := \frac{1+\frac{n-k}{k}r^2}{2r\sqrt{\frac{n-k}{k}}} \ge 1
    \end{align*}
    by AM-GM, with
    \begin{align*}
        t_+ - t_- = \frac{n}{\sqrt{k(n-k)}}\cdot \frac{1+r^2}{2r}. 
    \end{align*}
    Using a change of variable $t=t_- + (t_+-t_-)u$ for $u\in [0,1]$, we upper bound $A_2(\delta)$ as
    \begin{align*}
        A_2(\delta)&\le \sqrt{\frac{2}{3\delta}} \int_{-1+\delta}^{1-\delta} \pth{1+\frac{k}{n-k}r^2 + 2r\sqrt{\frac{k}{n-k}}t}^{\frac{n-k}{2}} \pth{1+\frac{n-k}{k}r^2-2r\sqrt{\frac{n-k}{k}}t}^{\frac{k}{2}}\rmd t \\
        &\le \sqrt{\frac{2}{3\delta}}\int_{t_-}^{t_+} \pth{1+\frac{k}{n-k}r^2 + 2r\sqrt{\frac{k}{n-k}}t}^{\frac{n-k}{2}} \pth{1+\frac{n-k}{k}r^2-2r\sqrt{\frac{n-k}{k}}t}^{\frac{k}{2}}\rmd t \\
        &= \sqrt{\frac{2}{3\delta}}(t_+-t_-)^{\frac{n}{2}+1} \int_0^1 \pth{2r\sqrt{\frac{k}{n-k}} u}^{\frac{n-k}{2}}\pth{2r\sqrt{\frac{n-k}{k}}(1-u)}^{\frac{k}{2}}\rmd u \\
        &= \sqrt{\frac{2}{3\delta}} \frac{(1+r^2)^{\frac{n}{2}+1}}{2r}\cdot \frac{n^{\frac{n}{2}+1}}{k^{\frac{k+1}{2}}(n-k)^{\frac{n-k+1}{2}}}\frac{\Gamma\pth{\frac{k}{2}+1}\Gamma\pth{\frac{n-k}{2}+1}}{\Gamma\pth{\frac{n}{2}+2}}, 
    \end{align*}
    where $\Gamma(x) = \int_0^\infty t^{x-1} e^{-t}\rmd t$ is the Gamma function. Using Stirling's approximation
    \begin{align*}
        \sqrt{2\pi x}\pth{\frac{x}{e}}^x \le \Gamma(x+1) \le \sqrt{2\pi x}\pth{\frac{x}{e}}^x\cdot e^{\frac{1}{12x}}, \qquad x>0, 
    \end{align*}
    we get
    \begin{align*}
    &\frac{n^{\frac{n}{2}+1}}{k^{\frac{k+1}{2}}(n-k)^{\frac{n-k+1}{2}}}\frac{\Gamma\pth{\frac{k}{2}+1}\Gamma\pth{\frac{n-k}{2}+1}}{\Gamma\pth{\frac{n}{2}+2}}\\
    &\le \frac{n^{\frac{n}{2}+1}}{k^{\frac{k+1}{2}}(n-k)^{\frac{n-k+1}{2}}} \cdot \frac{\sqrt{\pi k}\pth{\frac{k}{2e}}^{\frac{k}{2}}\sqrt{\pi(n-k)}\pth{\frac{n-k}{2e}}^{\frac{n-k}{2}} }{\pth{\frac{n}{2}+1} \sqrt{\pi n}\pth{\frac{n}{2e}}^{\frac{n}{2}}}\exp\left(\frac{1}{6k}+\frac{1}{6(n-k)}\right) \\
    &\le \frac{e^{1/3}\sqrt{\pi n}}{\frac{n}{2}+1} \le \frac{2e^{1/3}\sqrt{\pi}}{\sqrt{n}}, 
    \end{align*}
    which is an upper bound independent of $k$. Plugging it into the upper bound of $A_2(\delta)$ gives
    \begin{align}\label{eq:A2-delta}
    A_2(\delta)\le e^{1/3}\sqrt{\frac{2\pi}{3\delta}} \frac{(1+r^2)^{\frac{n}{2}+1}}{r\sqrt{n}} < \frac{2\pi}{3}\frac{(1+r^2)^{\frac{n}{2}+1}}{r\sqrt{n\delta}}. 
    \end{align}
\end{itemize}

Combining \eqref{eq:A1-delta} and \eqref{eq:A2-delta}, we get
\begin{align*}
|e_\ell(x^{(k)})|\le \frac{2(1+r^2)^{\frac{n}{2}}}{3r^{\ell}}\pth{\sqrt{2\delta} + \frac{1+r^2}{r\sqrt{n\delta}}} = \frac{2(1+r^2)^{\frac{n}{2}}}{3r^{\ell}}\pth{\sqrt{2\delta} + \frac{1}{\sqrt{\delta}}\cdot\sqrt{\frac{n}{\ell(n-\ell)}}}, 
\end{align*}
where the final step plugs in the choice of $r=\sqrt{\ell/(n-\ell)}$. Choosing
\begin{align*}
\delta = \frac{1}{2}\sqrt{\frac{n}{2\ell(n-\ell)}} \le \frac{1}{2},
\end{align*}
together with Stirling's approximation, leads to the upper bound
\begin{align*}
|e_\ell(x^{(k)})| &\le \frac{2(1+r^2)^{\frac{n}{2}}}{r^{\ell}}\pth{\frac{n}{2\ell(n-\ell)}}^{\frac{1}{4}} = \frac{2n^{\frac{n}{2}}}{\ell^{\frac{\ell}{2}}(n-\ell)^{\frac{n-\ell}{2}}}\pth{\frac{n}{2\ell(n-\ell)}}^{\frac{1}{4}} \\
&\le 2^{\frac{3}{4}}\pth{ \frac{n^{n+\frac{1}{2}}}{\ell^{\ell+\frac{1}{2}}(n-\ell)^{n-\ell+\frac{1}{2}}} }^{\frac{1}{2}} \le 2^{\frac{3}{4}}\pth{\sqrt{2\pi}e^{1/3} \binom{n}{\ell}}^{\frac{1}{2}} < \sqrt{10\binom{n}{\ell}}, 
\end{align*}
which proves the desired result \eqref{eq:target_binary}. 

\begin{remark}
The following insights of the saddle point method are used in the above proof. Note that $e_\ell(x)$ is the coefficient of $z^{\ell}$ in $f(z) = \prod_{i=1}^n (1+x_i z)$, with corresponding saddle point equation
\begin{align*}
    \frac{\ell}{z} = \frac{\rmd}{\rmd z}\log f(z) = \sum_{i=1}^n \frac{x_i}{1+x_iz}. 
\end{align*}
Specializing to the vector $x=x^{(k)}$ gives the saddle points
\begin{align*}
    z_{\pm} = \frac{1}{2(n-\ell)}\pth{-\frac{\ell(n-2k)}{\sqrt{k(n-k)}} \pm \sqrt{\frac{(n-2k)^2}{k(n-k)}\ell^2 - 4\ell(n-\ell)}}. 
\end{align*}
We observe that whenever the discriminant (the term under the squared root) is not positive, both saddle points $z_\pm$ have magnitude $\sqrt{\ell/(n-\ell)}$, motivating our choice of $r=\sqrt{\ell/(n-\ell)}$ in the proof. In addition, the Laplace approximation around $z_\pm$ suggests that $|e_\ell(x^{(k)})|$ attains the maximum when $z_+$ and $z_-$ coincide on the real axis, motivating us to decompose the integral into two quantities $A_1(\delta)$ and $A_2(\delta)$ as in \eqref{eq:A1-delta} and \eqref{eq:A2-delta}.
\end{remark}

\subsection{Reduction to binary support}
In this section, we show that the upper bound of $|e_\ell(x)|$ for general $x\in \bR^n$ could be reduced to the case of binary support in \eqref{eq:target_binary}. Specifically, for $\ell\ge 4$, we will prove by induction on $n\ge \ell$ that 
\begin{align}\label{eq:ind_hyp}
|e_\ell(x)| \le \sqrt{10\binom{n}{\ell}}, \qquad \text{for all } x\in \bR^n \text{ with } \sum_{i=1}^n x_i = 0, \sum_{i=1}^n x_i^2 = n.
\end{align}
The remaining cases $\ell\le 3$ will be dealt with in the next section. 

The base case $n=\ell$ of \eqref{eq:ind_hyp} is clear: by the AM-GM inequality, 
\begin{align*}
|e_\ell(x)|^2 = \prod_{i=1}^\ell x_i^2 \le \pth{\frac{1}{\ell}\sum_{i=1}^\ell x_i^2}^\ell = 1. 
\end{align*}
For the inductive step, assume that $n\ge \ell+1$ and the induction hypothesis \eqref{eq:ind_hyp} holds for all values smaller than $n$. As the constraint set of \eqref{eq:ind_hyp} is compact, the optimization program \eqref{eq:ind_hyp} must admit a maximizer $x^\star$. By the method of Lagrangian multipliers (for maximizing $e_\ell(x)$ or $-e_\ell(x)$), for any maximizer $x^\star$, there must exist $\lambda,\mu\in \bR$ such that 
\begin{align}\label{eq:stationarity}
e_{\ell-1}\pth{x^\star\backslash \sth{x_i^\star}} = \lambda x_i^\star + \mu, \qquad \text{for all } i \in [n]. 
\end{align}
Based on \eqref{eq:stationarity}, we split the analysis into three steps. 

\ifthenelse{\boolean{arxiv}}{\paragraph{Step I: if $x_i^\star = 0$ for some $i\in [n]$.}}{\vspace{0.2cm} \noindent \emph{Step I: if $x_i^\star = 0$ for some $i\in [n]$.}} In this case, the rescaled vector $\sqrt{1-n^{-1}}(x^\star \backslash \sth{x_i^\star})$ satisfies the constraint in \eqref{eq:ind_hyp} with $n$ replaced by $n-1$. By the induction hypothesis, 
\begin{align*}
|e_\ell(x^\star)| = \pth{\frac{n}{n-1}}^{\frac{\ell}{2}}\left|e_\ell\pth{\sqrt{1-n^{-1}}(x^\star \backslash \sth{x_i^\star})}\right| \le \pth{\frac{n}{n-1}}^{\frac{\ell}{2}}\sqrt{10\binom{n-1}{\ell}} \le \sqrt{10\binom{n}{\ell}}, 
\end{align*}
where the last step follows from simple algebra and Bernoulli's inequality $\pth{1-\frac{1}{n}}^{\ell}\ge 1-\frac{\ell}{n}$. Therefore, the inductive hypothesis \eqref{eq:ind_hyp} holds for $n$ in this case, and we are done. Subsequently, we may assume that every maximizer $x^\star$ of \eqref{eq:ind_hyp} must have $x_i^\star\neq 0$ for all $i\in [n]$. 

\ifthenelse{\boolean{arxiv}}{\paragraph{Step II: show that $|\mathrm{supp}(x^\star)|\le 3$.}}{\vspace{0.2cm} \noindent \emph{Step II: show that $|\mathrm{supp}(x^\star)|\le 3$.}} Assume by contradiction that $|\mathrm{supp}(x^\star)|\ge 4$, and that the values of $(x_1^\star,x_2^\star,x_3^\star,x_4^\star)$ are distinct. Choosing $i=1,2$ in \eqref{eq:stationarity}, a subtraction gives
\begin{align*}
(x_2^\star - x_1^\star)e_{\ell-2}(x^\star\backslash \sth{x_1^\star,x_2^\star}) = \lambda(x_1^\star - x_2^\star) \Longrightarrow  e_{\ell-2}(x^\star\backslash \sth{x_1^\star,x_2^\star}) = -\lambda. 
\end{align*}
Similarly we also have $e_{\ell-2}(x^\star\backslash \sth{x_1^\star,x_3^\star}) = -\lambda$, and a further subtraction gives
\begin{align}\label{eq:zero_123}
(x_3^\star - x_2^\star) e_{\ell-3}(x^\star\backslash \sth{x_1^\star,x_2^\star,x_3^\star}) = 0 \Longrightarrow e_{\ell-3}(x^\star\backslash \sth{x_1^\star,x_2^\star,x_3^\star}) = 0. 
\end{align}
Again, subtracting \eqref{eq:zero_123} with $e_{\ell-3}(x^\star\backslash \sth{x_1^\star,x_2^\star,x_4^\star}) = 0$ leads to
\begin{align}\label{eq:zero_1234}
(x_4^\star - x_3^\star) e_{\ell-4}(x^\star\backslash \sth{x_1^\star,x_2^\star,x_3^\star,x_4^\star}) = 0 \Longrightarrow e_{\ell-4}(x^\star\backslash \sth{x_1^\star,x_2^\star,x_3^\star,x_4^\star}) = 0. 
\end{align}
Moreover, since
\begin{align*}
e_{\ell-3}(x^\star\backslash \sth{x_1^\star,x_2^\star,x_3^\star}) = x_4^\star e_{\ell-4}(x^\star\backslash \sth{x_1^\star,x_2^\star,x_3^\star,x_4^\star}) + e_{\ell-3}(x^\star\backslash \sth{x_1^\star,x_2^\star,x_3^\star,x_4^\star}), 
\end{align*}
by \eqref{eq:zero_123} and \eqref{eq:zero_1234} we also have
\begin{align}\label{eq:zero_1234_another}
e_{\ell-3}(x^\star\backslash \sth{x_1^\star,x_2^\star,x_3^\star,x_4^\star}) = 0.
\end{align}
Based on \eqref{eq:zero_1234} and \eqref{eq:zero_1234_another}, we invoke the following result in \cite[Fact B]{gopalan2014inequalities}, which is a property for real-rooted polynomials. 

\begin{lemma}\label{lemma:real-rooted-poly}
For a real vector $x\in \bR^n$ and $0\le k\le n-1$, if $e_k(x) = e_{k+1}(x) = 0$, then $e_\ell(x)=0$ for all $\ell\ge k$. 
\end{lemma}

As $\ell \ge 4$ and $n\ge \ell+1$, \Cref{lemma:real-rooted-poly} applied to \eqref{eq:zero_1234} and \eqref{eq:zero_1234_another} gives $\prod_{i=5}^n x_i^\star = e_{n-4}(x^\star\backslash \sth{x_1^\star,x_2^\star,x_3^\star,x_4^\star}) = 0$, i.e. one of $(x_5^\star, \ldots, x_n^\star)$ must be zero. However, it is assumed at the end of Step I that $x_i^\star\neq 0$ for all $i\in [n]$, a contradiction! So $|\mathrm{supp}(x^\star)|\le 3$, as desired. 

\ifthenelse{\boolean{arxiv}}{\paragraph{Step III: show that $|\mathrm{supp}(x^\star)|\le 2$.}}{\vspace{0.2cm} \noindent \emph{Step III: show that $|\mathrm{supp}(x^\star)|\le 2$.}} We proceed to show that $|\mathrm{supp}(x^\star)|=3$ is also impossible. Assume by contradiction that $|\mathrm{supp}(x^\star)|=3$, and the values of $(x_1^\star,x_2^\star,x_3^\star)$ are distinct. We propose to find a triple $(x_1,x_2,x_3)\in \bR^3$ of distinct elements such that
\begin{align}
x_1 + x_2 + x_3 &= x_1^\star + x_2^\star + x_3^\star, \label{eq:e_1} \\
x_1x_2 + x_2x_3 + x_1x_3 &= x_1^\star x_2^\star + x_2^\star x_3^\star + x_1^\star x_3^\star, \label{eq:e_2} \\
\{x_1, x_2, x_3\} \cap \{x_1^\star, x_2^\star, x_3^\star, 0\} &= \varnothing. \label{eq:disjointness}
\end{align}
We show that such a triple $(x_1,x_2,x_3)$ exists. Note that \eqref{eq:e_1} and \eqref{eq:e_2} define an intersection of a hyperplane and a sphere in $\mathbb{R}^3$, which is a circle. This circle is nondegenerate (and thus has infinitely many points on it), for $(x_1^\star,x_2^\star,x_3^\star)$ and $(x_2^\star,x_1^\star,x_3^\star)$ are two distinct points on this circle. In addition, this circle belongs to none of the hyperplanes $\{x_i = x_j\}$, $\{x_i = x_j^\star\}$, or $\{x_i = 0\}$, so each hyperplane only intersects the circle at finitely many points. Therefore, we can choose any point $(x_1,x_2,x_3)$ on the circle other than the above intersections, and this triple satisfies distinctness and \eqref{eq:e_1}--\eqref{eq:disjointness}.

Given such a triple, we define a new vector $x' = (x_1,x_2,x_3,x_4^\star, \ldots, x_n^\star)\in \bR^n$. By \eqref{eq:e_1} and \eqref{eq:e_2}, it is clear that $x'$ satisfies the constraints in \eqref{eq:ind_hyp}. In addition, viewing both $e_{\ell}(x')$ and $e_\ell(x^\star)$ as polynomials of the first three elements, we get
\begin{align*}
e_{\ell}(x') - e_\ell(x^\star) &= e_{\ell-1}(x^\star \backslash \sth{x_1^\star,x_2^\star,x_3^\star})(x_1+x_2+x_3 - x_1^\star-x_2^\star-x_3^\star) \\
&\quad + e_{\ell-2}(x^\star \backslash \sth{x_1^\star,x_2^\star,x_3^\star})(x_1x_2 + x_2x_3 + x_1x_3 -x_1^\star x_2^\star - x_2^\star x_3^\star - x_1^\star x_3^\star) \\
&\quad + e_{\ell-3}(x^\star \backslash \sth{x_1^\star,x_2^\star,x_3^\star})(x_1x_2x_3 - x_1^\star x_2^\star x_3^\star) \\
&= 0, 
\end{align*}
where the last step is due to \eqref{eq:e_1}, \eqref{eq:e_2}, and \eqref{eq:zero_123}. In other words, $x'$ is also a maximizer of $|e_\ell(\cdot)|$, while $|\mathrm{supp}(x')|\ge 4$ and $x'$ does not contain zero by distinctness of $(x_1,x_2,x_3)$ and \eqref{eq:disjointness}. Proceeding as Step II would lead to the conclusion $|\mathrm{supp}(x')|\le 3$, which is a contradiction. Therefore, we have established that $|\mathrm{supp}(x^\star)|\le 2$, and the desired result \eqref{eq:ind_hyp} follows from the inequality \eqref{eq:target_binary} in the binary case. 

\subsection{Remaining corner cases}
The only cases not covered in \eqref{eq:ind_hyp} are the scenarios $\ell\le 3$, which we handle separately.  
\begin{itemize}
    \item For $\ell \in \{0,1\}$, the results $e_0(x) = 1$ and $e_1(x) = 0$ are trivial. 
    \item For $\ell=2$, Newton's identity gives that
    \begin{align*}
        |e_2(x)| = \frac{1}{2}\left| e_1(x)^2 - \sum_{i=1}^n x_i^2\right| = \frac{n}{2} \le \sqrt{\binom{n}{2}}. 
    \end{align*}
    \item For $\ell=3$, Newton's identity gives that
    \begin{align*}
        e_3(x) &= \frac{1}{3}\pth{ e_1(x)\pth{e_2(x) - \sum_{i=1}^n x_i^2} + \sum_{i=1}^n x_i^3 } = \frac{1}{3}\sum_{i=1}^n x_i^3. 
    \end{align*}
    As $\sum_{i=1}^n x_i^2 = n$, we have $|x_i|\le \sqrt{n}$ for all $i\in [n]$, and consequently
    \begin{align*}
    |e_3(x)| \le \frac{\sqrt{n}}{3}\sum_{i=1}^n x_i^2 = \frac{n^{\frac{3}{2}}}{3} \le \sqrt{3\binom{n}{3}}. 
    \end{align*}
\end{itemize}
In all these scenarios, the inequality \eqref{eq:real_case} of \Cref{thm:ESP} holds. The proof of \Cref{thm:ESP} is therefore complete. 
\section{Deferred proofs}\label{append:proofs}

\subsection{Completing the proof of \Cref{thm:main}}\label{append:proof_thm_main}
Since the upper bound \eqref{eq:upper_bound_1} of \Cref{thm:main} has been established in \Cref{sec:basis}, we prove the remaining upper bounds \eqref{eq:upper_bound_2} and \eqref{eq:upper_bound_3} here from \Cref{lemma:UB-entire-sum}. Based on \eqref{eq:polynomial-eigenvalue} and \Cref{lemma:UB-entire-sum}, we have
\begin{align*}
\chi^2(\bP_n \| \bQ_n) + 1 \le \sum_{\ell=0}^n \sum_{\substack{(\ell_2,\ldots,\ell_n)\in \naturals^{n-1} \\ \sum_{i=2}^n \ell_i = \ell}} \lambda_2^{\ell_2}\cdots \lambda_n^{\ell_n} \le \sum_{(\ell_2,\ldots,\ell_n)\in \naturals^{n-1}} \lambda_2^{\ell_2}\cdots \lambda_n^{\ell_n} = \prod_{i=2}^n \frac{1}{1-\lambda_i}. 
\end{align*}
Thanks to the spectral gap bound in \Cref{lemma:A-property}, we have $\lambda_i\in \qth{0, 1 - \frac{1}{\AffH}}$, so that
\begin{align*}
    \log \frac{1}{1-\lambda_i} \le \lambda_i \cdot \frac{\AffH}{\AffH- 1}\log \AffH
\end{align*}
by the convexity of $\lambda\mapsto \log\frac{1}{1-\lambda}$ on $[0,1)$. Summing over $i$ gives the upper bound
\begin{align*}
\log\pth{\chi^2(\bP_n \| \bQ_n) + 1} &\le \sum_{i=2}^n \lambda_i\cdot \frac{\AffH}{\AffH- 1}\log \AffH \\
&\stepa{\le} \Capa\cdot \frac{\AffH}{\AffH- 1}\log \AffH \stepb{\le} \Capa\pth{1+\log \AffH}, 
\end{align*}
where (a) follows from \Cref{lemma:A-property}, (b) uses $\log x\le x-1$ for all $x>0$. For the inequality \eqref{eq:upper_bound_3}, we have
\begin{align*}
    \log\pth{\chi^2(\bP_n \| \bQ_n) + 1}&\le \Capa\cdot \frac{\AffH}{\AffH- 1}\log \AffH \\
    &\stepc{\le} \Capa\cdot \frac{\Diam+1}{\Diam}\log (\Diam+1) \stepd{\le} (\Capa+1)\log(\Diam+1),
\end{align*}
where (c) is due to $\AffH\le 1+\Diam$ and 
\begin{align*}
\frac{\rmd}{\rmd x} \pth{\frac{x\log x}{x-1}} = \frac{x-1-\log x}{(x-1)^2} \ge 0, 
\end{align*}
and (d) uses $\Capa\le \Diam$. Therefore the upper bounds \eqref{eq:upper_bound_2} and \eqref{eq:upper_bound_3} are established. 

\subsection{Proof of \Cref{cor:specific_families}}\label{append:cor_family}
By \Cref{thm:main}, it suffices to upper bound $\Capa$ and $\AffH$ for the given families $\calP$. While the computation of $\AffH$ is typically straightforward, the evaluation of $\Capa$ may require some effort. We present a useful lemma to upper bound $\Capa$. 

\begin{lemma}\label{lemma:chi2-capacity}
Let $\calP_1, \ldots, \calP_m$ be families of probability distributions over the same space. Then
\begin{align*}
\sfC_{\chi^2}\pth{\bigcup_{i=1}^m \calP_i} \le \sum_{i=1}^m \sfC_{\chi^2}\pth{\calP_i} + m - 1.
\end{align*}
\end{lemma}
\begin{proof}
By induction it suffices to prove the lemma for $m=2$. In addition, since $\sfC_{\chi^2}\pth{\calP} \le \sfC_{\chi^2}\pth{\calP'}$ for $\calP\subseteq \calP'$, without loss of generality we may assume that $\calP_1$ and $\calP_2$ are disjoint. 

Let $\rho$ be a probability distribution over $\calP_1\cup \calP_2$, with $\rho(\calP_1)=a$ and $\rho(\calP_2) = 1-a$. Let $\rho_1, \rho_2$ be the restriction (conditional distribution) of $\rho$ to $\calP_1$ and $\calP_2$, respectively, it holds that
\begin{align*}
    \rho = a\rho_1 + (1-a)\rho_2. 
\end{align*}
We upper bound the $\chi^2$ mutual information $I_{\chi^2}(P;X)$ as follows: 
\begin{align*}
I_{\chi^2}(P;X) &= \bE_{P\sim \rho}\qth{\chi^2(P\| \bE_{P'\sim \rho}[P'] ) } \\
&= \bE_{P\sim \rho}\qth{\int \frac{(\rmd P)^2}{ \bE_{P'\sim \rho}[\rmd P']} } - 1 \\
&= a \bE_{P\sim \rho_1}\qth{\int \frac{(\rmd P)^2}{ \bE_{P'\sim \rho}[\rmd P']} } + (1-a) \bE_{P\sim \rho_2}\qth{\int \frac{(\rmd P)^2}{ \bE_{P'\sim \rho}[\rmd P']} } - 1\\
&\le a \bE_{P\sim \rho_1}\qth{\int \frac{(\rmd P)^2}{ a\bE_{P'\sim \rho_1}[\rmd P']} } + (1-a) \bE_{P\sim \rho_2}\qth{\int \frac{(\rmd P)^2}{ (1-a)\bE_{P'\sim \rho_2}[\rmd P']} } - 1 \\
&= \bE_{P\sim \rho_1}\qth{\chi^2(P\| \bE_{P'\sim \rho_1}[P'] ) } + \bE_{P\sim \rho_2}\qth{\chi^2(P\| \bE_{P'\sim \rho_2}[P'] ) } + 1 \\
&\le \sfC_{\chi^2}\pth{\calP_1} + \sfC_{\chi^2}\pth{\calP_2} + 1.
\end{align*}
Taking the supremum over $\rho\in \Delta(\calP_1\cup \calP_2)$ completes the proof. 
\end{proof}

Next we prove \Cref{cor:specific_families}. 

\ifthenelse{\boolean{arxiv}}{\paragraph{Gaussian family $\calP=\sth{ \calN(\theta,1): |\theta|\le \mu }$.}}{\vspace{0.2cm} \noindent \emph{Gaussian family $\calP=\sth{ \calN(\theta,1): |\theta|\le \mu }$.}} Since
\begin{align*}
\chi^2\pth{ \calN(\theta,1) \| \calN(\theta',1) } = \exp\pth{(\theta-\theta')^2} - 1, 
\end{align*}
it is clear that $\Diam = \exp(4\mu^2) - 1$. We show that $\Capa = O(\mu \wedge \mu^2)$. For $\mu\le 1$, we simply use $\Capa\le \Diam = O(\mu^2)$. For $\mu>1$, we split the class $\calP$ into $\calP \subseteq \cup_{m\in \mathbb{Z}, |m|\le \mu+1} \calP_m$, with $\calP_m := \sth{\calN(\theta,1): m\le \theta<m+1 }$. Clearly $\sfC_{\chi^2}(\calP_m) \le \sfD_{\chi^2}(\calP_m) = e - 1 = 
O(1)$, so \Cref{lemma:chi2-capacity} yields
\begin{align*}
\Capa \le \sum_{m\in \mathbb{Z}, |m|\le \mu+1} \sfC_{\chi^2}(\calP_m) + \left|\sth{m\in \mathbb{Z}: |m|\le \mu+1}\right| - 1 = O(\mu). 
\end{align*}
Then the claimed result follows from \eqref{eq:upper_bound_1} and \eqref{eq:upper_bound_3}. 

\ifthenelse{\boolean{arxiv}}{\paragraph{Gaussian family with finite support $\calP=\sth{ \calN(\theta,1): |\theta|\le \mu, \theta\in \Theta }$, $|\Theta|<\infty$.}}{\vspace{0.2cm} \noindent \emph{Gaussian family with finite support $\calP=\sth{ \calN(\theta,1): |\theta|\le \mu, \theta\in \Theta }$, $|\Theta|<\infty$.}} We only need to prove an additional upper bound $\Capa \le |\Theta|-1$. In fact, let $\Theta = \{\theta_1,\ldots,\theta_m\}$, \Cref{lemma:chi2-capacity} leads to
\begin{align*}
\Capa \le \sum_{i=1}^m \sfC_{\chi^2}\pth{ \sth{\calN(\theta_i,1)} } + m - 1 = m - 1.
\end{align*}

\ifthenelse{\boolean{arxiv}}{\paragraph{Bernoulli family $\calP = \{\Bern(p): p\in [\varepsilon,1-\varepsilon]\}$.}}{\vspace{0.2cm} \noindent \emph{Bernoulli family $\calP = \{\Bern(p): p\in [\varepsilon,1-\varepsilon]\}$.}} For $\varepsilon\ge \frac{1}{4}$, 
\begin{align*}
\Capa\le \Diam = \chi^2\pth{\Bern(\varepsilon) \| \Bern(1-\varepsilon)} = \frac{(1-2\varepsilon)^2}{\varepsilon(1-\varepsilon)} = O\pth{(1-2\varepsilon)^2}, 
\end{align*}
so that \Cref{eq:upper_bound_1} gives the first upper bound. For the second upper bound, we show that $\Capa \le 1-2\varepsilon$ for general $\varepsilon>0$, so that \eqref{eq:upper_bound_1} gives
\begin{align*}
\chi^2(\bP_n\|\bQ_n)\le 10\sum_{\ell=2}^n (1-2\varepsilon)^{\ell} \le 10\sum_{\ell=2}^\infty (1-2\varepsilon)^{\ell} = O\pth{\frac{1}{\varepsilon}}. 
\end{align*}
To see this, for any prior $\rho \in \Delta([\varepsilon, 1-\varepsilon])$, we upper bound the $\chi^2$ mutual information as
\begin{align*}
    I_{\chi^2}(p; X) &= \bE_{p\sim \rho}\qth{\chi^2\pth{\Bern(p) \| \Bern(\bE_{p'\sim \rho}[p'] } } \\
    &= \frac{\bE_{p\sim \rho}[p^2]}{\bE_{p\sim \rho}[p]} + \frac{\bE_{p\sim \rho}[(1-p)^2]}{\bE_{p\sim \rho}[1-p]} - 1 \\
    &\le 1-\varepsilon + 1-\varepsilon - 1 = 1-2\varepsilon. 
\end{align*}
Taking the supremum over the prior $\rho$ leads to $\Capa \le 1-2\varepsilon$. 

\ifthenelse{\boolean{arxiv}}{\paragraph{Poisson family $\calP = \sth{\Poi(\lambda): \lambda \in [0,M]}$.}}{\vspace{0.2cm} \noindent \emph{Poisson family $\calP = \sth{\Poi(\lambda): \lambda \in [0,M]}$.}} First, we have $\AffH = e^{M}$, for
\begin{align*}
\sup_{P_1, P_2\in \calP} \frac{1}{\int \sqrt{\rmd P_1\rmd P_2}} = \frac{1}{\int \sqrt{\rmd \Poi(0) \rmd \Poi(M)}} = e^{M/2}. 
\end{align*}
Next we prove $\Capa = O(\sqrt{M}\wedge M)$, by considering the cases $M\le 1$ and $M>1$. For $M\le 1$, pick any prior $\rho$ on $[0,M]$. We upper bound the $\chi^2$ mutual information as
\begin{align*}
I_{\chi^2}(\lambda; X) &= \bE_{\lambda\sim \rho}\qth{\chi^2\pth{\Poi(\lambda) \| \bE_{\lambda'\sim \rho}[\Poi(\lambda')] } } \\
&= \frac{\bE_{\lambda\sim \rho}\qth{ \bP(\Poi(\lambda)=0)^2 } }{\bE_{\lambda\sim \rho}\qth{ \bP(\Poi(\lambda)=0) }} + \sum_{k=1}^\infty \frac{\bE_{\lambda\sim \rho}\qth{ \bP(\Poi(\lambda)=k)^2 } }{\bE_{\lambda\sim \rho}\qth{ \bP(\Poi(\lambda)=k) }} - 1 \\
&\stepc{\le} 1 + \sum_{k=1}^\infty \bP(\Poi(M) = k) - 1 = 1 - e^{-M} \le M, 
\end{align*}
where (c) uses the monotonicity $\bP(\Poi(\lambda)=k) \le \bP(\Poi(M) = k)$ for all $\lambda\le M\le 1$ and $k\ge 1$. Taking the supremum over the prior $\rho$ gives that $\Capa\le M$. For $M>1$, we write $\calP = \cup_{i=1}^m \calP_i$, with
\begin{align*}
    \calP_i := \sth{ \Poi(\lambda): (i-1)^2 \le \lambda \le i^2 \wedge M }. 
\end{align*}
Clearly we can choose $m = O(\sqrt{M})$ to ensure that $\calP = \cup_{i=1}^m \calP_i$. In addition, our previous argument shows that $\sfC_{\chi^2}(\calP_1) \le 1$, and for $i\ge 2$, 
\begin{align*}
\sfC_{\chi^2}(\calP_i) \le \sfD_{\chi^2}(\calP_i) \le \max_{(i-1)^2\le \lambda_1, \lambda_2\le i^2} \exp\pth{\frac{(\lambda_1-\lambda_2)^2}{\lambda_2}}- 1 \le \exp\pth{\frac{ (i^2 - (i-1)^2)^2 }{(i-1)^2}} - 1 = O(1).  
\end{align*}
Therefore, by \Cref{lemma:chi2-capacity}, we have
\begin{align*}
\Capa \le \sum_{i=1}^m \sfC_{\chi^2}(\calP_i) + m - 1 = O(\sqrt{M}). 
\end{align*}
Then the claimed result follows from \eqref{eq:upper_bound_1} and \eqref{eq:upper_bound_2}.

\subsection{Proof of \Cref{thm:deFinetti}}

Let $\overline{P} := \frac{1}{n}\sum_{i=1}^n P_i$. Expanding the likelihood ratio gives
\begin{align*}
\chi^2(\bP_{k,n} \| \bQ_{k,n}) + 1 
&= \bE_{X_1,\ldots,X_k\sim \overline{P}}\qth{ \pth{\bE_{\pi\sim\Unif(S_n)}\qth{ \prod_{i=1}^k \frac{\rmd P_{\pi(i)}}{\rmd \overline{P}}(X_i)} }^2 } \\
&\stepa{=} \bE_{\pi,\pi'\sim \Unif(S_n)}\qth{ \prod_{i=1}^k \bE_{X_i\sim \overline{P}}\qth{\frac{\rmd P_{\pi(i)}}{\rmd \overline{P}}(X_i)\frac{\rmd P_{\pi'(i)}}{\rmd \overline{P}}(X_i)} } \\
&\stepb{=} \bE_{\pi,\pi'\sim \Unif(S_n)} \qth{n^k\prod_{i=1}^k A_{\pi(i),\pi'(i)} } \\
&\stepc{=} \frac{n^k}{\binom{n}{k}^2 k!}\sum_{\substack{ S,S'\subseteq [n] \\ |S| = |S'| = k }} \Perm(A_{S,S'}), 
\end{align*}
where in (a) we introduce an independent copy $\pi'\sim \Unif(S_n)$ of $\pi$, (b) recalls the definition of $A$ in \eqref{eq:A-matrix}, and (c) uses $\bP(\pi([k]) = S) = \bP(\pi'([k])=S') = \binom{n}{k}^{-1}$ for all subsets $S,S'$ of size $k$ and averages over $k!$ bijections between $S$ and $S'$. Next we relate $A$ to its centered version $\overline{A} = A - \frac{1}{n}{\bf 1}{\bf 1}^\top$. Let $J = {\bf 1}{\bf 1}^\top$ denote the all-ones matrix, then
\begin{align*}
\sum_{\substack{ S,S'\subseteq [n] \\ |S| = |S'| = k }} \Perm(A_{S,S'}) &= \sum_{\substack{ S,S'\subseteq [n] \\ |S| = |S'| = k }} \Perm\pth{\overline{A}_{S,S'} + \frac{J_{S,S'}}{n}} \\
&= \sum_{\substack{ S,S'\subseteq [n] \\ |S| = |S'| = k }} \sum_{\substack{ T\subseteq S, T'\subseteq S' \\ |T| = |T'|}} \Perm(\overline{A}_{T,T'}) \Perm\pth{\frac{J_{S\backslash T, S'\backslash T'}}{n}} \\
&= \sum_{\substack{ S,S'\subseteq [n] \\ |S| = |S'| = k }} \sum_{\ell=0}^k \sum_{\substack{ T\subseteq S, T'\subseteq S' \\ |T| = |T'|=\ell}} \Perm(\overline{A}_{T,T'}) \frac{(k-\ell)!}{n^{k-\ell}} \\
&\stepd{=} \sum_{\ell=0}^k \sum_{\substack{ T,T'\subseteq [n] \\ |T| = |T'|=\ell}} \binom{n-\ell}{k-\ell}^2 \Perm(\overline{A}_{T,T'}) \frac{(k-\ell)!}{n^{k-\ell}}, 
\end{align*}
where (d) swaps the sum and counts the number of $S\supseteq T, S'\supseteq T'$ of size $k$. After some algebra, the above two identities give that
\begin{align*}
\chi^2(\bP_{k,n} \| \bQ_{k,n}) + 1 &= \frac{n^{\ell} k!}{[n!]^2}\sum_{\ell=0}^k \frac{[(n-\ell)!]^2}{(k-\ell)!}  \sum_{\substack{ T,T'\subseteq [n] \\ |T| = |T'|=\ell}} \Perm(\overline{A}_{T,T'}) \\
&\stepe{=} \frac{n^{\ell} k!}{[n!]^2}\sum_{\ell=0}^k \frac{[(n-\ell)!]^2}{(k-\ell)!}\cdot \frac{n!}{(n-\ell)! n^{\ell}}S_{\ell} = \sum_{\ell=0}^k \frac{\binom{k}{\ell}}{\binom{n}{\ell}}S_{\ell}, 
\end{align*}
where (e) follows from the identity between $T_{\ell}$ and $S_{\ell}$ in \Cref{lemma:RST}. 

Finally, to prove the first upper bound, by \eqref{eq:polynomial-eigenvalue} we have
\begin{align*}
S_{\ell} &= \frac{n^n}{n!}\sum_{\substack{(\ell_2,\ldots,\ell_n)\in \naturals^{n-1} \\ \sum_{i=2}^n \ell_i = \ell}} \frac{\Perm(U_{n-\ell,\ell_2,\ldots,\ell_n})^2}{(n-\ell)!\ell_2!\cdots \ell_n!}\lambda_2^{\ell_2}\cdots \lambda_n^{\ell_n} \\
&\stepf{\le} \sum_{\substack{(\ell_2,\ldots,\ell_n)\in \naturals^{n-1} \\ \sum_{i=2}^n \ell_i = \ell}} \frac{10\cdot \ell!(n-\ell)!}{(n-\ell)!\ell_2!\cdots \ell_n!}\lambda_2^{\ell_2}\cdots \lambda_n^{\ell_n}  = 10\pth{\sum_{i=2}^n \lambda_i}^{\ell} \le 10\Capa^{\ell}, 
\end{align*}
where (f) follows from the same program (\Cref{lemma:banach} and \Cref{thm:ESP}) in \Cref{sec:basis}, and the last inequalty uses \Cref{lemma:A-property}. As $S_0 = 1$, $S_1 =0$, 
\begin{align*}
\chi^2(\bP_{k,n} \| \bQ_{k,n}) \le \sum_{\ell=2}^k \pth{\frac{k}{n}}^\ell \cdot 10 \Capa^{\ell} \le 20\pth{\frac{k\Capa}{n}}^2, \quad \text{if }\frac{k\Capa}{n}\le \frac{1}{2}. 
\end{align*}
For the second upper bound, we use \eqref{eq:polynomial-eigenvalue} and $S_\ell\ge 0$ to conclude
\begin{align*}
\chi^2(\bP_{k,n} \| \bQ_{k.n}) = \sum_{\ell=2}^k \frac{\binom{k}{\ell}}{\binom{n}{\ell}}S_{\ell} \le \frac{k(k-1)}{n(n-1)}\sum_{\ell=2}^n S_{\ell} = \frac{k(k-1)}{n(n-1)}\chi^2(\bP_{n} \| \bQ_{n}).  
\end{align*} 
Now the claimed result follows from the upper bound of $\chi^2(\bP_{n} \| \bQ_{n})$ in \Cref{thm:main}. 

\subsection{Proof of \Cref{thm:two_mixtures}}\label{append:proof_two_mixtures}
By \Cref{lemma:UB-individual-sum}, we have
\begin{align*}
    \sum_{\ell=1}^n \ell S_{\ell-1} &\le 3\sum_{\ell=1}^n \ell^{3/2} \sum_{\substack{(\ell_2,\dots,\ell_{n-1})\in \naturals^{n-2} \\ \sum_{i=2}^{n-1} \ell_i = \ell-1}} \lambda_2^{\ell_2}\cdots \lambda_{n-1}^{\ell_{n-1}} \\
    &\stepa{\le} 3\sum_{\ell=1}^n \sum_{\substack{(\ell_2,\dots,\ell_{n-1})\in \naturals^{n-2} \\ \sum_{i=2}^{n-1} \ell_i = \ell-1}} (\ell_2+1)^{3/2}\lambda_2^{\ell_2}\cdots (\ell_{n-1}+1)^{3/2}\lambda_{n-1}^{\ell_{n-1}} \\
    &\le 3\prod_{i=2}^{n-1}\pth{\sum_{\ell=0}^\infty (\ell+1)^{3/2}\lambda_i^\ell} \stepb{\le} 3\prod_{i=2}^{n-1}\pth{\frac{1}{2}\sum_{\ell=0}^\infty (\ell+2)(\ell+1)\lambda_i^\ell} \stepc{=} 3\prod_{i=2}^{n-1}\frac{1}{(1-\lambda_i)^3}, 
\end{align*}
where (a) is due to $\prod_{i=2}^{n-1}(1+\ell_i)\ge 1 + \sum_{i=2}^{n-1} \ell_i = \ell$, (b) uses $2\sqrt{\ell+1}\le \ell+2$ for all $\ell\ge 0$, and (c) follows from the identity
\begin{align*}
\sum_{\ell=0}^\infty (\ell+2)(\ell+1)x^{\ell} = \frac{\rmd^2}{\rmd x^2}\sum_{\ell=0}^\infty x^{\ell+2} = \frac{2}{(1-x)^3}, \qquad \text{for } |x| < 1. 
\end{align*}
By the same arguments in \Cref{append:proof_thm_main}, \Cref{lemma:UB-empirical-bayes} leads to
\begin{align*}
\TV(\bP_n, \bP_n')^2 \le \frac{3\Diam}{n}(e\AffH)^{3\Capa}, 
\end{align*}
which establishes the first claim of \Cref{thm:two_mixtures}. The second claim of \Cref{thm:two_mixtures} simply follows from the Cauchy--Schwarz inequality. 

\subsection{Proof of \Cref{lemma:gaussian-seq-model}}
Let $\bQ$ be the joint distribution of $(\theta, X)$ under the i.i.d. prior $\theta\sim \pth{(1-\alpha)\delta_0 + \alpha\delta_{\mu_0}}^{\otimes n}$, and $\bP$ be the counterpart under the permutation prior. By the arguments in \cite[Proof of Theorem 1]{zhang2012minimax}, to establish the claimed result, it suffices to prove that
\begin{align}\label{eq:posterior_target}
\bE_\bP\qth{ \bE_\bP[\theta_1 | X^n]^2 } = o\pth{\frac{s\mu^2}{n}}. 
\end{align}
For completeness we show how \eqref{eq:posterior_target} implies the target Bayes risk lower bound. Let $\mu_0 := \sqrt{1-\varepsilon}\mu$. For $p_1\in (0,1)$, consider the minimizer $a^\star$ of the map $a\mapsto f(a):=p_1|\mu_0-a|^q + (1-p_1)|a|^q$. Clearly $a^\star\in [0,\mu_0]$, and since $(1-p_1)(a^\star)^q\le f(a^\star)\le f(0)=p_1\mu_0^q$, we conclude that $a^\star\le (\frac{p_1}{1-p_1})^{1/q}\mu_0$. Consequently, if $p_1\le \varepsilon$, then 
\begin{align*}
f(a^\star) \ge p_1(1-\varepsilon_1)^q \mu_0^q, \quad \text{with } \varepsilon_1 := \pth{\frac{\varepsilon}{1-\varepsilon}}^{1/q}. 
\end{align*}
Therefore, letting $p_1(X^n):= \bP(\theta_1=\mu_0|X^n)$ be the posterior probability of $\theta_1 = \mu_0$, the Bayes $\ell^q$ risk for estimating $\theta_1$ under $\bP$ is
\begin{align*}
\bE_\bP\qth{ \min_{a} \pth{p_1(X^n) |\mu-a|^q + (1-p_1(X^n))|a|^q }} \ge (1-\varepsilon_1)^q\mu_0^q\cdot \bE_\bP\qth{ p_1(X^n) \mathbbm{1}(p_1(X^n)\le \varepsilon) }. 
\end{align*}
To proceed, note that $\bE_\bP[p_1(X^n)] = \bP(\theta_1=\mu_0)=\frac{\lfloor s\rfloor}{n}$, and $\bE_\bP[p_1(X^n)^2] = o(\frac{s}{n})$ by \eqref{eq:posterior_target}. Therefore, the Bayes $\ell^q$ risk for estimating $\theta_1$ is further lower bounded by
\begin{align*}
&(1-\varepsilon_1)^q\mu_0^q\cdot\pth{\frac{\lfloor s \rfloor}{n} - \bE_\bP[p_1(X^n)\mathbbm{1}(p_1(X^n)>\varepsilon)]} \\
&\ge (1-\varepsilon_1)^q\mu_0^q\cdot\pth{\frac{\lfloor s \rfloor}{n} - \frac{\bE_\bP[p_1(X^n)^2]}{\varepsilon}} = (1-\varepsilon_1)^q\mu_0^q\cdot\frac{\lfloor s \rfloor}{n} \pth{1- \frac{o(1)}{\varepsilon}}, 
\end{align*}
which is $(1-o(1))(1-c(\varepsilon))\mu^{q-p}R^p$ with $c(\varepsilon)\to 0$ as $\varepsilon\to 0^+$, as claimed. 

In the sequel we prove \eqref{eq:posterior_target}. Since
\begin{align*}
\bE_\bP\qth{ (\theta_1 - \bE_\bP[\theta_1|X^n])^2 } = \bE_\bP[\theta_1^2] - \bE_\bP\qth{ \bE_\bP[\theta_1 | X^n]^2 } = \frac{\lfloor s\rfloor}{n}\mu_0^2 -  \bE_\bP\qth{ \bE_\bP[\theta_1 | X^n]^2 }, 
\end{align*}
it is equivalent to proving that $\bE_\bP\qth{ (\theta_1 - \bE_\bP[\theta_1|X^n])^2 } \ge (1-o(1))\frac{s\mu_0^2}{n}$. Let $\delta = \delta(\varepsilon) > 0$ be any fixed constant such that
\begin{align}\label{eq:delta}
(1+\delta)(1-\varepsilon) < 1 - \frac{\varepsilon}{2}. 
\end{align}
Our proof is based on the following I-MMSE formula \cite{guo2005mutual}: for an SNR parameter $\lambda\ge 0$ and $X^\lambda := \sqrt{\lambda}\theta + N$ with $N\sim \calN(0,I_n)$ independent of $\theta$, for any prior distribution of $\theta$ it holds that
\begin{align*}
I(\theta; X^{\lambda}) = \frac{1}{2}\int_0^\lambda \bE\qth{ \|\theta - \bE\qth{\theta | X^t }\!\|_2^2 } \rmd t =: \frac{1}{2}\int_0^\lambda \mathrm{mmse}(t) \rmd t.
\end{align*}
We denote by $I_{\bP}(\theta;X^\lambda)$ and $I_{\bQ}(\theta;X^{\lambda})$ the mutual information under the permutation prior and the i.i.d. prior, respectively. Similarly we use the notations $(\mathrm{mmse}_\bP(t), \mathrm{mmse}_\bQ(t))$ for the minimum mean squared errors (MMSEs) under these priors. Note that our target is now equivalent to showing that $\mathrm{mmse}_\bP(1) \ge (1-o(1))s\mu_0^2$. 

By the I-MMSE formula, 
\begin{align*}
\frac{1}{2}\int_1^{1+\delta} \pth{\mathrm{mmse}_{\bP}(t) - \mathrm{mmse}_{\bQ}(t)} \rmd t = I_{\bP}(\theta;X^{1+\delta}) - I_{\bQ}(\theta;X^{1+\delta}) - (I_{\bP}(\theta;X^1) - I_{\bQ}(\theta;X^1)). 
\end{align*}
For the difference in the mutual information, we establish the following inequality: 
\begin{align}\label{eq:mutual_info_difference}
0\le I_{\bQ}(\theta;X^\lambda) - I_{\bP}(\theta;X^\lambda) \le 10\sum_{\ell=2}^n I_{\chi^2}(\theta_1; X_1^\lambda)^{\ell}, 
\end{align}
where $I_{\chi^2}$ is the $\chi^2$ mutual information in \Cref{defn:capacity_diam}. To show \eqref{eq:mutual_info_difference}, dropping the superscript $\lambda$ for simplicity, we have
\begin{align*}
I_{\bQ}(\theta;X) - I_{\bP}(\theta;X) &= nI_{\bQ}(\theta_1; X_1) - \bE_{\bP_{\theta,X}}\qth{\log \frac{\prod_{i=1}^n \rmd \bP_{X_i|\theta_i}}{\rmd \bP_X}} \\
&= nI_{\bQ}(\theta_1; X_1) - \bE_{\bP_{\theta,X}}\qth{\log \frac{\prod_{i=1}^n \rmd \bP_{X_i|\theta_i}}{\rmd \bQ_X}} + \KL(\bP_X \| \bQ_X) \\
&= nI_{\bQ}(\theta_1; X_1) - n \bE_{\bP_{\theta_1,X_1}}\qth{\log \frac{\rmd \bP_{X_1|\theta_1}}{\rmd \bQ_{X_1}}} + \KL(\bP_X \| \bQ_X) \\
&= \KL(\bP_X \| \bQ_X), 
\end{align*}
where the last step uses that $\bP_{\theta_1,X_1} = \bQ_{\theta_1,X_1}$. Since $\bP_X$ and $\bQ_X$ are the permutation mixture and the i.i.d. approximation, respectively, an intermediate step in the proof of the upper bound \eqref{eq:upper_bound_1} (cf. \Cref{subsec:general-case}) shows that
\begin{align*}
0\le \KL(\bP_X \| \bQ_X) \le \chi^2(\bP_X \| \bQ_X) \le 10\sum_{\ell=2}^n I_{\chi^2}(\theta_1; X_1)^{\ell}, 
\end{align*}
establishing the target inequality \eqref{eq:mutual_info_difference}. 

Based on \eqref{eq:mutual_info_difference}, using the decreasing property of $t\mapsto \mathrm{mmse}(t)$ under both priors, we have
\begin{align*}
\frac{\delta}{2}\pth{ \mathrm{mmse}_{\bP}(1) - \mathrm{mmse}_{\bQ}(1+\delta)} \ge \frac{1}{2}\int_1^{1+\delta} \pth{\mathrm{mmse}_{\bP}(t) - \mathrm{mmse}_{\bQ}(t)} \rmd t \ge -10\sum_{\ell=2}^n I_{\chi^2}(\theta_1; X_1^{1+\delta})^{\ell},
\end{align*}
which rearranges to
\begin{align}\label{eq:mmse_inequality}
\mathrm{mmse}_{\bP}(1) \ge \mathrm{mmse}_{\bQ}(1+\delta) - \frac{20}{\delta}\sum_{\ell=2}^n I_{\chi^2}(\theta_1; X_1^{1+\delta})^{\ell}. 
\end{align}
The quantity $\mathrm{mmse}_{\bQ}(1+\delta)$, this quantity corresponds to the Bayes risk under a rescaled i.i.d. prior $\pth{(1-\alpha)\delta_0 + \alpha\delta_{\mu_1}}^{\otimes n}$ for the Gaussian mean vector, with
\begin{align*}
    \mu_1 := \sqrt{(1+\delta)(1-\varepsilon)}\mu = \sqrt{2(1+\delta)(1-\varepsilon)\log(1/R^p)}, \qquad \alpha := \frac{\lfloor s\rfloor}{n}. 
\end{align*}
As $\mu_1 < \mu$ by \eqref{eq:delta} and $s=\omega(1)$, this Bayes risk has been characterized in \cite[Section 6]{donoho1994minimax} that 
\begin{align*}
    \mathrm{mmse}_{\bQ}(1+\delta) \ge \frac{(1-o(1))s\mu_1^2}{1+\delta} = (1-o(1))s\mu_0^2. 
\end{align*} 
Here the $(1+\delta)^{-1}$ factor accounts for discrepancy in estimating $\theta$ and $\sqrt{1+\delta}\cdot \theta$. Finally, since $\delta = \delta(\varepsilon) > 0$ is a fixed constant and $s\mu_0^2 = \omega(1)$, \eqref{eq:mmse_inequality} proves the target claim $\mathrm{mmse}_\bP(1) \ge (1-o(1))s\mu_0^2$ as long as we show that
\begin{align}\label{eq:chi2-GSM}
I_{\chi^2}(\theta_1; X_1^{1+\delta}) = o(1). 
\end{align}

To prove \eqref{eq:chi2-GSM}, note that as $R = o(1)$, for $n$ large enough we have $\mu\ge 1$, so that
\begin{align*}
\alpha \le \frac{s}{n} = \frac{R^p}{\mu^p} \le R^p = \exp\pth{-\frac{\mu_1^2}{2(1+\delta)(1-\varepsilon)}} \overset{\prettyref{eq:delta}}{<} \exp\pth{-\frac{(1+\frac{\varepsilon}{2})\mu_1^2}{2}}. 
\end{align*}
Consequently,
\begin{align*}
I_{\chi^2}(\theta_1;X_1^{1+\delta}) &\stepa{=} \alpha(1-\alpha)\int \frac{(\rmd \calN(\mu_1,1) - \rmd \calN(0,1))^2}{\rmd [\alpha \calN(\mu_1,1) + (1-\alpha)\calN(0,1)]} \\
&\stepb{=} \alpha(1-\alpha)\bE_{Z\sim \calN(0,1)}\qth{\frac{\pth{\exp(\mu_1 Z - \mu_1^2/2)-1}^2}{\alpha \exp(\mu_1 Z - \mu_1^2/2) + 1-\alpha}}, 
\end{align*}
where (a) uses the definition of the $\chi^2$ mutual information, and (b) follows from a simple change of measure. To proceed, note that 
\begin{align*}
&\bE_{Z\sim \calN(0,1)}\qth{\frac{\pth{\exp(\mu_1 Z - \mu_1^2/2)-1}^2}{\alpha \exp(\mu_1 Z - \mu_1^2/2) + 1-\alpha} \indc{Z\le (1+\frac{\varepsilon}{4})\mu_1}} \\
&\le \frac{1}{1-\alpha}\pth{1 + \bE_{Z\sim \calN(0,1)}\qth{\exp(2\mu_1Z - \mu_1^2) \indc{Z\le (1+\frac{\varepsilon}{4})\mu_1} }} \\
&\stepc{=} \frac{1}{1-\alpha}\pth{1 + e^{\mu_1^2}\cdot\bP\pth{\calN(2\mu_1,1) \le \pth{1+\frac{\varepsilon}{4}}\mu_1 }} \\
&\stepd{\le} \frac{1}{1-\alpha}\pth{1 + \exp\pth{\mu_1^2 - \frac{1}{2}\pth{1-\frac{\varepsilon}{4}}^2\mu_1^2}} = \frac{1}{1-\alpha}\pth{1 + \exp\pth{\pth{\frac{1}{2}+\frac{\varepsilon}{4}-\frac{\varepsilon^2}{32}}\mu_1^2}}, 
\end{align*}
where (c) is due to a change of measure, and (d) uses $\bP(\calN(0,1)\ge t)\le \exp(-t^2/2)$ for every $t\ge 0$. For the remaining quantity, we similarly have
\begin{align*}
&\bE_{Z\sim \calN(0,1)}\qth{\frac{\pth{\exp(\mu_1 Z - \mu_1^2/2)-1}^2}{\alpha \exp(\mu_1 Z - \mu_1^2/2) + 1-\alpha} \indc{Z> (1+\frac{\varepsilon}{4})\mu_1}} \\
&\le \frac{1}{\alpha} \bE_{Z\sim \calN(0,1)}\qth{\exp(\mu_1 Z - \mu_1^2/2) \indc{Z> (1+\frac{\varepsilon}{4})\mu_1}} \\
&= \frac{1}{\alpha}\bP\pth{\calN(\mu_1,1) > \pth{1+\frac{\varepsilon}{4}}\mu_1 } \le \frac{1}{\alpha}\exp\pth{-\frac{\varepsilon^2 \mu_1^2}{32}}. 
\end{align*}
Since $\alpha \le \exp\pth{-\frac{(1+\varepsilon/2)\mu_1^2}{2}}$, we have
\begin{align*}
I_{\chi^2}(\theta_1; X_1^{1+\delta}) \le \alpha + 2\exp\pth{-\frac{\varepsilon^2 \mu_1^2}{32}} = o(1), 
\end{align*}
as claimed. The proof of \eqref{eq:chi2-GSM} is then complete.

\subsection{Proof of \Cref{lemma:privacy}}
It is clear that $\bP_n$ and $\bP_n'$ are precisely the permutation mixtures in \Cref{thm:two_mixtures}, where the Laplace and Gaussian mechanisms correspond to the distribution class
\begin{align*}
\calP_1 = \sth{ x+\mathrm{Lap}\pth{\frac{1}{\varepsilon}}: x\in [0,1] }, \quad \calP_2 = \sth{ x+\calN\pth{0, \frac{1}{\varepsilon^2}}: x\in [0,1] }, 
\end{align*}
respectively. Straightforward computations yield
\begin{align*}
\mathsf{D}_{\chi^2}(\calP_1) = \frac{2e^{\varepsilon}+e^{-2\varepsilon}}{3} - 1 \le \varepsilon^2, \quad \text{for } \varepsilon\in [0,1]. 
\end{align*}
Since $\Capa \vee (\Delta_{H^2}(\calP)-1) \le \Diam$, \Cref{thm:two_mixtures} yields
\begin{align*}
\TV(\bP_n, \bP_n')^2 \le \frac{3\varepsilon^2[e(1+\varepsilon^2)]^{3\varepsilon^2}}{n} = O\pth{\frac{\varepsilon^2}{n}}, 
\end{align*}
establishing the results for the Laplace mechanism. For the Gaussian mechanism, we similarly have
\begin{align*}
\mathsf{D}_{\chi^2}(\calP_2) = e^{\varepsilon^2}-1 \le 2\varepsilon^2, \quad \text{for } \varepsilon\in [0,1], 
\end{align*}
and the rest follows from \Cref{thm:two_mixtures}. 

\subsection{Proof of \Cref{cor:EB}}
We first recall the alternative representations of $r^{\mathrm{S}}(\theta)$ and $r^{\mathrm{PI}}(\theta)$ in \cite{greenshtein2009asymptotic} (see also \cite{weinstein2021permutation}). Consider a ``postulated Bayes model'' where $\vartheta$ is a uniformly random permutation of a fixed vector $\theta$, and the observation $X\sim \otimes_{i=1}^n P_{\vartheta_i}$ conditioned on $\vartheta$. Due to the permutation-invariant structure of both $\calD^{\mathrm{S}}$ and $\calD^{\mathrm{PI}}$, it holds that
\begin{align*}
r^{\mathrm{S}}(\theta) = \inf_{\widehat{\vartheta}\in \calD^{\mathrm{S}}} \bE\qth{L(\vartheta, \widehat{\vartheta})}, \quad r^{\mathrm{PI}}(\theta) = \inf_{\widehat{\vartheta}\in \calD^{\mathrm{PI}}} \bE\qth{L(\vartheta, \widehat{\vartheta})}. 
\end{align*}

We next show that
\begin{align}\label{eq:risk-difference}
r^{\mathrm{S}}(\theta) - r^{\mathrm{PI}}(\theta) \le \frac{M}{n} \sum_{i=1}^n \bE\qth{\TV(P_{\vartheta_i | X_i }, P_{\vartheta_i | X})}. 
\end{align}
In fact, 
\begin{align*}
r^{\mathrm{S}}(\theta) &= \inf_{\widehat{\vartheta}\in \calD^{\mathrm{S}}} \bE\qth{L(\vartheta, \widehat{\vartheta})} = \inf_{\Delta(\cdot)} \frac{1}{n}\sum_{i=1}^n \bE\qth{\ell(\vartheta_i, \Delta(X_i))} = \bE\qth{\frac{1}{n}\sum_{i=1}^n \inf_{u} \bE_{\vartheta_i | X_i}\qth{\ell(\vartheta_i, u) }}, \\
r^{\mathrm{PI}}(\theta) &= \inf_{\widehat{\vartheta}\in \calD^{\mathrm{PI}}} \bE\qth{L(\vartheta, \widehat{\vartheta})} =\inf_{\widehat{\vartheta}\in \calD^{\mathrm{PI}}}  \frac{1}{n}\sum_{i=1}^n \bE\qth{\ell(\vartheta_i, \widehat{\vartheta}_i(X))} = \bE\qth{\frac{1}{n}\sum_{i=1}^n \inf_{u} \bE_{\vartheta_i | X}\qth{\ell(\vartheta_i, u) }}, 
\end{align*}
and the optimal decision rules are the Bayes optimal decision rules under the posterior distributions $P_{\vartheta_i|X_i}$ and $P_{\vartheta_i|X}$, respectively, for the loss $\ell$. Consequently, as $0\le \ell(\cdot,\cdot)\le M$, 
\begin{align*}
r^{\mathrm{S}}(\theta) - r^{\mathrm{PI}}(\theta) &\le \frac{1}{n}\sum_{i=1}^n \bE\qth{ \sup_u \left| \bE_{\vartheta_i | X_i}\qth{\ell(\vartheta_i, u) } - \bE_{\vartheta_i | X}\qth{\ell(\vartheta_i, u) } \right| } \\
&\le \frac{M}{n} \sum_{i=1}^n \bE\qth{\TV(P_{\vartheta_i | X_i }, P_{\vartheta_i | X})}, 
\end{align*}
establishing \eqref{eq:risk-difference}. To proceed, note that by the Bayes rule, we have
\begin{align*}
\frac{\rmd P_{\vartheta_i | X_i}}{\rmd P_{\vartheta_i | X}} = \frac{\rmd P_{X_{\backslash i} | X_i} }{\rmd P_{X_{\backslash i} | \vartheta_i}}, 
\end{align*}
and therefore
\begin{align*}
    \bE\qth{\TV(P_{\vartheta_i | X_i }, P_{\vartheta_i | X})} &= \bE_{\vartheta_i, X}\qth{\left| \frac{\rmd P_{X_{\backslash i} | X_i} }{\rmd P_{X_{\backslash i} | \vartheta_i}} - 1\right|} = \bE_{\vartheta_i, X_i}\qth{\TV(P_{X_{\backslash i} | X_i}, P_{X_{\backslash i} | \vartheta_i})} \\
    &\stepa{\le} \bE_{X_i}\bE_{\vartheta_i, \vartheta_i'|X_i}\qth{ \TV(P_{X_{\backslash i} | \vartheta_i'}, P_{X_{\backslash i} | \vartheta_i}) }\le \max_{\vartheta_i, \vartheta_i'} \TV(P_{X_{\backslash i} | \vartheta_i}, P_{X_{\backslash i} | \vartheta_i'}) \\
    &\stepb{\le} \sqrt{\frac{3\Diam(e\AffH)^{3\Capa}}{n-1}}. 
\end{align*}
Here (a) follows from convexity and $P_{X_{\backslash i} | X_i} = \bE_{\vartheta_i'|X_i}\qth{P_{X_{\backslash i} | \vartheta_i'}}$, where $\vartheta_i'$ is an independent copy of $\vartheta_i$ given $X_i$. For (b), note that $P_{X_{\backslash i} | \vartheta_i}$ is the permutation mixture based on $\{P_{\theta_1},\dots,P_{\theta_n}\} \backslash \{P_{\vartheta_i}\}$, so that \Cref{thm:two_mixtures} gives the claimed upper bound. Plugging this upper bound into \eqref{eq:risk-difference} and using $n\ge 2$ leads to the claimed corollary. 

\subsection{Proof of \Cref{thm:permutation_prior}}
For the first inequality \eqref{eq:mutual_info_1}, the triangle inequality applied to the separable loss yields
\begin{align*}
|\bE_{\bP} [L(\vartheta,\widehat{\vartheta})] - \bE_{\bQ} [L(\vartheta,\widehat{\vartheta})]| &\le \frac{1}{n}\sum_{i=1}^n |\bE_{\bP_{\vartheta_i, X}}[\ell(\vartheta_i, \widehat{\vartheta}_i)] - \bE_{\bQ_{\vartheta_i, X}}[\ell(\vartheta_i, \widehat{\vartheta}_i)]| \\
&\le \frac{1}{n}\sum_{i=1}^n \sqrt{\chi^2(\bP_{\vartheta_i, X} \| \bQ_{\vartheta_i, X}) \var_{\bQ}[\ell(\vartheta_i, \widehat{\vartheta}_i)]}. 
\end{align*}
Therefore, it suffices to show that $\chi^2(\bP_{\vartheta_i, X} \| \bQ_{\vartheta_i, X}) \le e(\chi^2(\calP)+1)$ for all $i\in [n]$, where $\chi^2(\calP)$ is any upper bound in \Cref{thm:main} for $\calP=\{P_{\theta_1},\dots,P_{\theta_n}\}$. To this end, note that $\bP_{\theta_i, X_i} = \bQ_{\theta_i, X_i}$ and the Markov structure $X_i - \vartheta_i - X_{\backslash i}$ under both $\bP$ and $\bQ$, we have 
\begin{align*}
\chi^2(\bP_{\vartheta_i, X} \| \bQ_{\vartheta_i, X}) = \bE_{\vartheta_i}\qth{ \chi^2\pth{\bP_{X_{\backslash i} | \vartheta_i} \| \bQ_{X_{\backslash i} | \vartheta_i}}}. 
\end{align*}
Conditioned on $\vartheta_i = \theta_j$, it is clear that $\bP_{X_{\backslash i} | \vartheta_i}$ is the $(n-1)$-dimensional permutation mixture generated by $\{P_{\theta_k}: k\neq j\}$, and $\bQ_{X_{\backslash i} | \vartheta_i}$ is the i.i.d. distribution with marginal $\frac{1}{n}\sum_{i=1}^n P_{\theta_i}$. Defining an auxiliary distribution $\bQ_{n-1} = (\frac{1}{n-1}\sum_{k\neq j} P_{\theta_k})^{\otimes (n-1)}$, then 
\begin{align*}
\frac{\rmd \bQ_{X_{\backslash i}|\vartheta_i = \theta_j}}{\rmd \bQ_{n-1}} \ge \pth{\frac{n-1}{n}}^{n-1} \ge \frac{1}{e}
\end{align*}
almost surely. Consequently, 
\begin{align*}
\chi^2\pth{\bP_{X_{\backslash i} | \vartheta_i=\theta_j} \| \bQ_{X_{\backslash i} | \vartheta_i=\theta_j}} + 1 &= \int \frac{(\rmd \bP_{X_{\backslash i} | \vartheta_i=\theta_j})^2}{\rmd \bQ_{X_{\backslash i} | \vartheta_i=\theta_j}} \le e\int \frac{(\rmd \bP_{X_{\backslash i} | \vartheta_i=\theta_j})^2}{\rmd \bQ_{n-1}} \\
&= e\pth{\chi^2(\bP_{X_{\backslash i} | \vartheta_i=\theta_j} \| \bQ_{n-1}) +1} \stepa{\le} e(\chi^2(\calP)+1), 
\end{align*}
where (a) notes that the i.i.d. distribution $\bQ_{n-1}$ has the same one-dimensional marginal as $\bP_{X_{\backslash i} | \vartheta_i=\theta_j}$. This establishes the inequality \eqref{eq:mutual_info_1}. The second inequality \eqref{eq:consistency} follows from the concavity of $x\mapsto \sqrt{x}$, and the variance upper bound
\begin{align*}
\frac{1}{n}\sum_{i=1}^n \var_{\bQ}[\ell(\vartheta_i, \widehat{\vartheta}_i)] \le \frac{1}{n}\sum_{i=1}^n \bE_{\bQ}[\ell(\vartheta_i, \widehat{\vartheta}_i)^2] \le \frac{M}{n}\sum_{i=1}^n \bE_{\bQ}[\ell(\vartheta_i, \widehat{\vartheta}_i)] = M\cdot \bE_{\bQ}[L(\vartheta,\widehat{\vartheta})]. 
\end{align*}

\subsection{Proof of \Cref{lemma:EB-quadratic}}
We begin with a technical result. Let $P_0, \dots, P_n\in \calP$ be $n+1$ probability distributions, and for $i\in \{0,\dots,n\}$, define the permutation mixture $\bP_{-i}$ as
\begin{align*}
\bP_{-i} := \bE_{\pi\sim \Unif(S_{ \{0,\dots,n\} \backslash \{i\} })}\qth{\otimes_{j\neq i} P_{\pi(j)}}, 
\end{align*}
where $S_A$ denotes the permutation group over the finite set $A$. In other words, $\bP_{-i}$ is the permutation mixture of $\{P_0,\dots,P_n\} \backslash \{P_i\}$. The idea of \cite{greenshtein2009asymptotic} leads to the following inequality (though originally stated in a more complicated way): 
\begin{lemma}\label{lemma:Greenshtein-Ritov}
For $\overline{\bP} = \frac{1}{n+1}\sum_{i=0}^n \bP_{-i}$ and every $i\in \{0,\dots,n\}$ it holds that
\begin{align*}
\chi^2(\overline{\bP} \| \bP_{-i} ) \le \frac{\Diam}{n+1}. 
\end{align*}
In particular, for every $i,j\in \{0,\dots,n\}$ we have
\begin{align*}
    H^2(\bP_{-i}, \bP_{-j}) \le \frac{4\Diam}{n+1}. 
\end{align*}
\end{lemma}


The proof of \Cref{lemma:Greenshtein-Ritov} is deferred to the end of this section. Under the quadratic loss, it is shown in \cite{greenshtein2009asymptotic} that the optimal simple and permutation invariant decision rules are
\begin{align*}
\widehat{\vartheta}^{\mathrm{S}} = \pth{ \bE\qth{\vartheta_1 | X_1}, \dots, \bE\qth{\vartheta_n | X_n} }, \quad \text{and} \quad \widehat{\vartheta}^{\mathrm{PI}} = \pth{ \bE\qth{\vartheta_1 | X}, \dots, \bE\qth{\vartheta_n | X} }, 
\end{align*}
respectively, where the joint distribution of $(\vartheta,X)$ is given in the ``postulated Bayes model'' in the proof of \Cref{cor:EB}. In addition, the risk difference admits a simple form
\begin{align*}
    r^{\mathrm{S}}(\theta) - r^{\mathrm{PI}}(\theta) = \bE\qth{ \pth{\bE\qth{\vartheta_1 | X_1}-\bE\qth{\vartheta_1 | X}}^2 }. 
\end{align*}

Since
\begin{align*}
\bE\qth{\vartheta_1 | X_1}-\bE\qth{\vartheta_1 | X} = \sum_{i=1}^n \theta_i \pth{ P_{\vartheta_1 = \theta_i | X_1} - P_{\vartheta_1 = \theta_i | X} } = \sum_{i=1}^n \theta_i P_{\vartheta_1 = \theta_i | X_1} \bpth{1 - \frac{P_{X_2^n | \vartheta_1 = \theta_i}}{P_{X_2^n| X_1}}}
\end{align*}
according to the Bayes rule, we have
\begin{align*}
&\bE\qth{\pth{\bE\qth{\vartheta_1 | X_1}-\bE\qth{\vartheta_1 | X}}^2} \\
&= \bE_{X_1}\sum_{i=1}^n \sum_{j=1}^n \theta_i\theta_jP_{\vartheta_1 = \theta_i | X_1}P_{\vartheta_1 = \theta_j | X_1} \bE_{X_2^n|X_1}\qth{\pth{1 - \frac{P_{X_2^n | \vartheta_1 = \theta_i}}{P_{X_2^n| X_1}}}\pth{1 - \frac{P_{X_2^n | \vartheta_1 = \theta_j}}{P_{X_2^n| X_1}}}}. 
\end{align*}
Next we fix any realization of $X_1$. Let $w_i := P_{\vartheta_1 = \theta_i | X_1}$ and $P_i := P_{X_2^n| \vartheta_1=\theta_i}$, then $\sum_{i=1}^n w_i = 1$, and
\begin{align*}
\overline{P} := P_{X_2^n | X_1} = \sum_{i=1}^n w_i P_i. 
\end{align*}
Therefore, the entire term inside $\bE_{X_1}$ can be written as $\theta^\top M\theta$, where
\begin{align*}
M_{ij} = w_iw_j \bE_{\overline{P}}\qth{\pth{1-\frac{\rmd P_i}{\rmd \overline{P}}} \pth{1-\frac{\rmd P_j}{\rmd \overline{P}}}} = w_iw_j\pth{ \int \frac{\rmd P_i\rmd P_j}{\rmd \overline{P}} -1}. 
\end{align*}
Comparing with the matrix $A$ defined in \eqref{eq:A-matrix} and the centered version $\overline{A}$ before \eqref{eq:T_ell}, it is clear that
\begin{align*}
M = n\mathrm{diag}(w_1,\dots,w_n)\overline{A}\mathrm{diag}(w_1,\dots,w_n), 
\end{align*}
and therefore
\begin{align*}
\theta^\top M\theta \le n\lambda_{\max}(\overline{A}) \sum_{i=1}^n w_i^2\theta_i^2 \le nM^2\lambda_{\max}(\overline{A}) \|w\|_2^2.
\end{align*}
To upper bound the largest eigenvalue of $\overline{A}$, note that \Cref{lemma:Greenshtein-Ritov} implies that $H^2(P_i, P_j)\le \frac{4\Diam}{n}$ for all $i,j\in [n]$. Consequently, following the spectral gap argument in \Cref{lemma:A-property}, it holds that
\begin{align*}
\lambda_{\max}(\overline{A}) = \lambda_2(A) \le 1 - \pth{1-\max_{i,j}\frac{H^2(P_i,P_j)}{2}}^2 \le \max_{i,j}{H^2(P_i,P_j)} \le \frac{4\Diam}{n}. 
\end{align*}
In addition, 
\begin{align*}
\bE_{X_1}\qth{\|w\|_2^2} &= \bE_{X_1}\qth{ \sum_{i=1}^n P_{\vartheta_1 = \theta_i | X_1}^2 } = \sum_{i=1}^n \bE_{X_1}\qth{ \pth{\frac{\rmd P_{\theta_i}}{\sum_{j=1}^n \rmd P_{\theta_j} } }^2}\\
&= \frac{1}{n^2}\sum_{i=1}^n \bpth{1+\chi^2\bigg(P_{\theta_i} \bigg\| \frac{1}{n} \sum_{j=1}^n P_{\theta_j} \bigg) } \le \frac{1+\Capa}{n}. 
\end{align*}
A combination of the above inequalities shows that
\begin{align*}
\bE\qth{\pth{\bE\qth{\vartheta_1 | X_1}-\bE\qth{\vartheta_1 | X}}^2} = \bE_{X_1}[\theta^\top M\theta] &\le 4M^2 \Diam \cdot \bE_{X_1}\qth{\|w\|_2^2} \\
&\le \frac{4M^2 \Diam (1+\Capa)}{n}, 
\end{align*}
as claimed. 

\begin{proof}[Proof of \Cref{lemma:Greenshtein-Ritov}]
Without loss of generality assume that $i=0$. The proof is via a careful coupling in \cite{greenshtein2009asymptotic}: for $\pi\in S_n$ and $j\in [n]$, define joint distributions of $(X_1,\dots,X_n)$ as
\begin{align*}
P_{\pi} &:= P_{\pi(1)}\otimes \cdots \otimes P_{\pi(n)} \\
P_{\pi,j} &:= P_{\pi(1)}\otimes \cdots\otimes P_{\pi(\pi^{-1}(j)-1)} \otimes P_0 \otimes P_{\pi(\pi^{-1}(j)+1)} \otimes \cdots \otimes P_{\pi(n)}. 
\end{align*}
For $j=0$ we also define $P_{\pi,0} = P_\pi$. Since $\bP_{-j} = \bE_{\pi\sim \Unif(S_n)}[P_{\pi,j}]$ for all $j\in \{0,\dots,n\}$, the joint convexity of the $\chi^2$ divergence yields
\begin{align*}
\chi^2(\overline{\bP} \| \bP_{-0}) &= \chi^2\bpth{ \bE_{\pi\sim \Unif(S_n)}\bqth{ \frac{1}{n+1}\sum_{j=0}^{n} P_{\pi,j}  }  \bigg\| \bE_{\pi\sim \Unif(S_n)}\qth{P_{\pi} }} \\
&\le \bE_{\pi\sim \Unif(S_n)}\bqth{ \chi^2\bpth{\frac{1}{n+1}\sum_{j=0}^{n} P_{\pi,j}  \bigg\| P_{\pi} } } \\
&\stepa{\le} \frac{n}{n+1}\bE_{\pi\sim \Unif(S_n)}\bqth{ \chi^2\bpth{\frac{1}{n}\sum_{j=1}^{n} P_{\pi,j}  \bigg\| P_{\pi} } }.
\end{align*}
Here (a) uses $\frac{1}{n+1}\sum_{j=0}^{n} P_{\pi,j} = \frac{P_{\pi}}{n+1}+\frac{n}{n+1}\cdot \frac{1}{n}\sum_{j=1}^{n} P_{\pi,j}$ and the convexity again. To proceed we perform the second moment computation. By simple algebra, the likelihood ratio between $P_{\pi,j}$ and $P_{\pi}$ is
\begin{align*}
\frac{\rmd P_{\pi,j}}{\rmd P_{\pi}}(X_1,\dots,X_n) =
\frac{\rmd P_0}{\rmd P_j}(X_{\pi^{-1}(j)}).
\end{align*}
Consequently, for $j\neq j'$, 
\begin{align*}
\int \frac{\rmd P_{\pi,j} \rmd P_{\pi,j'}}{\rmd P_{\pi}} &= \bE_{(X_1,\dots,X_n)\sim P_{\pi}}\qth{ \frac{\rmd P_0}{\rmd P_j}(X_{\pi^{-1}(j)})\frac{\rmd P_0}{\rmd P_{j'}} (X_{\pi^{-1}(j')}) } \\
&= \bE_{X_{\pi^{-1}(j)} \sim P_j} \qth{ \frac{\rmd P_0}{\rmd P_j}(X_{\pi^{-1}(j)})}\bE_{X_{\pi^{-1}(j')} \sim P_{j'}}\qth{\frac{\rmd P_0}{\rmd P_{j'}} (X_{\pi^{-1}(j')}) } = 1,
\end{align*}
and
\begin{align*}
\int \frac{(\rmd P_{\pi,j})^2}{\rmd P_{\pi}} =\bE_{(X_1,\dots,X_n)\sim P_{\pi}}\qth{ \pth{\frac{\rmd P_0}{\rmd P_j}(X_{\pi^{-1}(j)})}^2} = \chi^2(P_0 \| P_j) + 1 \le \Diam + 1. 
\end{align*}
Therefore, the second moment method gives
\begin{align*}
\chi^2\bpth{\frac{1}{n}\sum_{j=1}^{n} P_{\pi,j}  \bigg\| P_{\pi} } = \frac{1}{n^2}\sum_{j,j'=1}^n \int \frac{\rmd P_{\pi,j} \rmd P_{\pi,j'}}{\rmd P_{\pi}} - 1 \le \frac{\Diam}{n}, 
\end{align*}
and we obtain the first claim. The second claim follows from
\begin{align*}
H^2(\bP_{-i}, \bP_{-j}) \stepb{\le} 2H^2(\overline{\bP}, \bP_{-i}) + 2H^2(\overline{\bP}, \bP_{-j}) \stepc{\le} 2\chi^2(\overline{\bP} \| \bP_{-i}) + 2\chi^2(\overline{\bP} \|\bP_{-j}) \le \frac{4\Diam}{n+1}, 
\end{align*}
where (b) is the triangle inequality of the Hellinger distance, and (c) uses $H^2(P,Q)\le \chi^2(P\|Q)$.    
\end{proof}

\subsection{Proof of \Cref{lemma:permanent}}
By the definition of $\chi^2$-divergence, 
\begin{align*}
\chi^2(\bP_n \| \bQ_n) + 1 = \bE_{\bQ_n}\qth{\pth{\frac{\rmd \bP_n}{\rmd \bQ_n}}^2} = \bE_{X_1,\ldots,X_n\sim \overline{P}}\qth{ \pth{\bE_{\pi\sim \Unif(S_n)} \qth{\prod_{i=1}^n \frac{\rmd P_{\pi(i)}}{\rmd \overline{P}}(X_i)} }^2 }. 
\end{align*}
By introducing an independent copy $\pi'\sim \Unif(S_n)$ of $\pi$, it holds that
\begin{align*}
\chi^2(\bP_n\|\bQ_n)+1 &= \bE_{\pi,\pi'\sim \Unif(S_n)} \sth{ \bE_{X_1,\ldots,X_{n} \sim \overline{P}} \qth{ \prod_{i=1}^{n} \frac{\rmd P_{\pi(i)} }{\rmd \overline{P}}(X_i)\prod_{i=1}^{\ell} \frac{\rmd P_{\pi'(i)}}{\rmd \overline{P}}(X_i)} } \\
&= \bE_{\pi,\pi'\sim \Unif(S_n)} \qth{ \prod_{i=1}^{n}\pth{\int \frac{\rmd P_{\pi(i)} \rmd P_{\pi'(i)}}{\rmd \overline{P}}  }} \\
&\overset{\prettyref{eq:A-matrix}}{=} \bE_{\pi,\pi'\sim \Unif(S_n)} \qth{\prod_{i=1}^{n} (nA_{\pi(i),\pi'(i)})} = \frac{n^n}{n!} \Perm(A), 
\end{align*}
as claimed. 

\subsection{Proof of \Cref{lemma:A-property}}
The definition of $\overline{P} = \frac{1}{n}\sum_{i=1}^n P_i$ trivially implies that $A$ is doubly stochastic. To see why $A$ is PSD, simply note that
\begin{align*}
    A_{ij} = \frac{1}{n}\bE_{X\sim \overline{P}}\qth{ \frac{\rmd P_i}{\rmd \overline{P}}(X) \frac{\rmd P_j}{\rmd \overline{P}}(X) }
\end{align*}
could be written as an inner product of $\frac{\rmd P_i}{\rmd \overline{P}}$ and $\frac{\rmd P_j}{\rmd \overline{P}}$. The trace upper bound is also straightforward: 
\begin{align*}
\trace(A) = \frac{1}{n}\sum_{i=1}^n \int \frac{(\rmd P_i)^2}{\rmd \overline{P}} = \frac{1}{n}\sum_{i=1}^n \chi^2(P_i \| \overline{P}) + 1 \le \Capa + 1. 
\end{align*}
As for the eigenstructure of $A$, the same leading eigenvalue/eigenvector holds for any doubly stochastic matrix. The only non-trivial property is the spectral gap. To this end, note that the Laplacian $L = I-A$ satisfies that
\begin{align*}
x^\top Lx = \sum_{1\le i<j\le n} A_{ij}(x_i - x_j)^2. 
\end{align*}
For each entry $A_{ij}$, we can lower bound it as
\begin{align*}
    A_{ij} = \frac{1}{n}\int \frac{\rmd P_i \rmd P_j}{\rmd \overline{P}} \stepa{\ge}  \frac{1}{n}\pth{\int \sqrt{\rmd P_i \rmd P_j}}^2 = \frac{1}{n}\pth{1-\frac{H^2(P_i,P_j)}{2}}^2 \ge \frac{1}{n\AffH}, 
\end{align*}
where (a) is due to the Cauchy--Schwarz inequality. Consequently, for any unit vector $x\in \bR^n$ with ${\bf 1}^\top x = 0$, we have
\begin{align*}
    x^\top L x \ge \frac{1}{n\AffH}\sum_{1\le i<j\le n}(x_i - x_j)^2 = \frac{1}{2n\AffH}\sum_{i,j=1}^n (x_i - x_j)^2 = \frac{1}{\AffH},
\end{align*}
which proves the spectral gap lower bound.

\subsection{Proof of \Cref{lemma:wick-formula}}\label{append:wick_formula_proof}
Let $P = (P_{ij})_{i\in [m], j\in [n]}$ and $z=(z_1,\ldots,z_n)$. Then
\begin{align*}
\bE\qth{ \prod_{i=1}^m |(Pz)_i|^2 } &= \bE\bqth{ \prod_{i=1}^m \Big|\sum_{j=1}^n P_{ij}z_j \Big|^2} \\
&= \bE\bqth{ \prod_{i=1}^m \bpth{\sum_{j=1}^n P_{ij}z_j} \bpth{\sum_{j=1}^n P_{ij}\bar{z}_j} } \\
&\stepa{=} \sum_{\pi\in S_m} \prod_{i=1}^m \bE\bqth{ \bpth{\sum_{j=1}^n P_{ij}z_j} \bpth{\sum_{j=1}^n P_{\pi(i)j}\bar{z}_j} } \\
&= \sum_{\pi\in S_m} \prod_{i=1}^m \bpth{\sum_{j=1}^n P_{ij}P_{\pi(i)j}} \\
&= \sum_{\pi\in S_m} \prod_{i=1}^m (PP^\top)_{i\pi(i)} = \Perm(PP^\top), 
\end{align*}
where (a) is due to Isserlis' theorem \cite{isserlis1918formula}. 

\subsection{Proof of \Cref{lemma:UB-empirical-bayes}}\label{append:proof_UB-empirical-bayes}
By the Cauchy--Schwarz inequality, for any probability measure $\bR_n$ it holds that
\begin{align*}
\TV(\bP_n, \bP_n')^2  = \pth{\frac{1}{2}\int |\rmd \bP_n - \rmd \bP_n'|}^2 \le \frac{1}{4}\int \frac{(\rmd \bP_n - \rmd \bP_n')^2}{\rmd \bR_n}. 
\end{align*}
Choosing $\bR_n = \overline{P}^{\otimes n}$ completes the proof of the first claim. The likelihood ratio between the signed measure $\bP_n - \bP_n'$ and the product measure $\overline{P}^{\otimes n}$ is computed as
\begin{align*}
\frac{\rmd(\bP_n - \bP_n')}{\rmd \overline{P}^{\otimes n}}(X) &\stepa{=} \bE_{I\sim \Unif([n]), \pi: [n]\backslash \{I\} \leftrightarrow \{2,\ldots,n\}}\bqth{\frac{\rmd (P_1-P_1')}{\rmd \overline{P}}(X_I)\cdot \prod_{i\neq I} \frac{\rmd P_{\pi(i)}}{\rmd \overline{P}}(X_i) }  \\
&\stepb{=} \bE_{I\sim \Unif([n]), \pi: [n]\backslash \{I\} \leftrightarrow \{2,\ldots,n\}} \bqth{\frac{\rmd (P_1-P_1')}{\rmd \overline{P}}(X_I)\cdot  \sum_{T\subseteq [n]\backslash \{I\}}\prod_{i\in T} \frac{\rmd (P_{\pi(i)}-\overline{P})}{\rmd \overline{P}}(X_i) }  \\
&\stepc{=} \sum_{S\subseteq [n]: |S|\ge 1} \bpth{\frac{1}{n}\sum_{i\in S} \frac{\rmd (P_1-P_1')}{\rmd \overline{P}}(X_i) \cdot \bE_{\pi: [n]\backslash \{i\} \leftrightarrow \{2,\ldots,n\}} \prod_{j\in S \backslash \{i\}} \frac{\rmd (P_{\pi(j)}-\overline{P})}{\rmd \overline{P}}(X_j) } \\
&=: \sum_{S\subseteq [n]: |S|\ge 1} f_S(X), 
\end{align*}
where (a) decomposes the random permutation over $S_n$ into a random choice of $I = \pi^{-1}(1)$ and a random bijection between the rest (i.e. $[n]\backslash \{i\}$ and $\{2,\cdots,n\}$), and (b) is the identity $\prod_{i=1}^n (1+x_i) = \sum_{T\subseteq [n]}\prod_{i\in T} x_i$. In (c), we expand the expectation over $I\sim \Unif([n])$ and swap the sum as $\sum_{i=1}^n \sum_{T\subseteq [n]\backslash \{i\}} g_i(T) = \sum_{S\subseteq [n]: |S|\ge 1} \sum_{i\in S} g_i(S\backslash \{i\})$. 

Next, we argue that the functions $\{f_S(X): S\subseteq [n], |S|\ge 1\}$ are orthogonal under $\overline{P}^{\otimes n}$. Since $f_{S}(X)$ could be expressed as $\frac{1}{n}\sum_{i\in S} \bE_{\pi \sim \mu_i}\qth{g_{S,i,\pi}(X)}$ for some probability measure $\mu_i$ and function $g_{S,i,\pi}(X)$, it suffices to show the orthogonality between $g_{S,i,\pi}(X)$ and $g_{T,i',\pi'}(X)$ for all $S\neq T$ and $(i,i',\pi,\pi')$. To see so, simply note that
\begin{align*}
g_{S,i,\pi}(X) = \frac{\rmd (P_1-P_1')}{\rmd \overline{P}}(X_i) \prod_{j\in S \backslash \{i\}} \frac{\rmd (P_{\pi(j)}-\overline{P})}{\rmd \overline{P}}(X_j)
\end{align*}
is a product of zero-mean functions of $\{X_j: j\in S\}$ under $\overline{P}^{\otimes n}$, so the orthogonality follows. This orthogonality yields
\begin{align*}
&\int \frac{(\rmd \bP_n - \rmd \bP_n')^2}{\rmd \overline{P}^{\otimes n}} = \bE_{\overline{P}^{\otimes n}}\bqth{\bpth{\sum_{S\subseteq [n]: |S|\ge 1} f_S(X)}^2} = \sum_{S\subseteq [n]: |S|\ge 1} \bE_{\overline{P}^{\otimes n}} 
\qth{f_S(X)^2}  \\
&\stepd{=} \sum_{\ell=1}^n \binom{n}{\ell} \bE_{\overline{P}^{\otimes n}} 
\qth{f_{[\ell]}(X)^2} \\
&\stepe{\le} \sum_{\ell=1}^n \binom{n}{\ell}  \frac{\ell}{n^2}\sum_{i=1}^{\ell} \bE_{X_1,\ldots,X_\ell\sim \overline{P}}\bqth{ \bpth{\frac{\rmd (P_1-P_1')}{\rmd \overline{P}}(X_i) \cdot \bE_{\pi: [n]\backslash \{i\} \leftrightarrow \{2,\ldots,n\}} \prod_{j\in [\ell] \backslash \{i\}} \frac{\rmd (P_{\pi(j)}-\overline{P})}{\rmd \overline{P}}(X_j)}^2 } \\
&\stepf{=} \sum_{\ell=1}^n \binom{n}{\ell}  \frac{\ell^2}{n^2} \int \frac{(\rmd P_1 - \rmd P_1')^2}{\rmd \overline{P}}\cdot \bE_{X_2,\ldots,X_{\ell}\sim \overline{P}}\bqth{ \bpth{\bE_{\pi: \{2,\ldots,n\} \leftrightarrow \{2,\ldots,n\}} \prod_{j=2}^{\ell} \frac{\rmd (P_{\pi(j)}-\overline{P})}{\rmd \overline{P}}(X_j)}^2}, 
\end{align*}
where (d) notes that the second moment of $f_S$ only depends on the size of $S$, (e) uses the Cauchy--Schwarz inequality $(\sum_{i=1}^\ell x_i)^2 \le \ell\sum_{i=1}^\ell x_i^2$, and (f) notes that the inner expectation does not depend on $i$. To proceed, note that
\begin{align*}
\int \frac{(\rmd P_1 - \rmd P_1')^2}{\rmd \overline{P}} &\le 2\int\frac{(\rmd P_1 - \rmd \overline{P})^2}{\rmd \overline{P}} + 2\int\frac{(\rmd P_1' - \rmd \overline{P})^2}{\rmd \overline{P}} \\
&= 2(\chi^2(P_1\|\overline{P}) + \chi^2(P_1'\|\overline{P})) \le 4\Diam, 
\end{align*}
and 
\begin{align*}
\bE_{X_2,\ldots,X_{\ell}\sim \overline{P}}\bqth{ \bpth{\bE_{\pi: \{2,\ldots,n\} \leftrightarrow \{2,\ldots,n\}} \prod_{j=2}^{\ell} \frac{\rmd (P_{\pi(j)}-\overline{P})}{\rmd \overline{P}}(X_j)}^2} = R_{\ell-1} = \frac{S_{\ell-1}}{\binom{n-1}{\ell-1}}
\end{align*}
by the definition of $R_{\ell}$ in \eqref{eq:R_ell} and the identity between $R_{\ell}$ and $S_{\ell}$ in \Cref{lemma:RST}, where the matrix $A\in \bR^{(n-1)\times (n-1)}$ is now constructed from $(P_2,\ldots,P_n)$. A combination of all the above then gives
\begin{align*}
\TV(\bP_n, \bP_n')^2 \le \frac{1}{4}\sum_{\ell=1}^n \binom{n}{\ell}\frac{\ell^2}{n^2}\cdot 4\Diam\frac{S_{\ell-1}}{\binom{n-1}{\ell-1}} = \frac{\Diam}{n}\sum_{\ell=1}^n \ell S_{\ell-1}, 
\end{align*}
as claimed. 

\subsection{Proof of \Cref{lemma:wick-formula-2}}\label{append:wick-formula-2-proof}
Let $v = (v_1,\cdots,v_n) := \widetilde{P}z  \sim \calCN(0, \widetilde{P}\widetilde{P}^\top)$, where
\begin{align*}
\widetilde{P}\widetilde{P}^\top = \widetilde{U}\widetilde{D}\widetilde{U}^\top = A - \lambda_1 u_1 u_1^\top = A - \frac{1}{n}{\bf 1}{\bf 1}^\top = \overline{A}, 
\end{align*}
where $\lambda_1 = 1$ and $u_1 = {\bf 1}/\sqrt{n}$ are the leading eigenvalue and eigenvector of $A$, respectively (cf. \Cref{lemma:A-property}), and $\overline{A}$ is the centered version of $A$ defined above \eqref{eq:T_ell}. Consequently, 
\begin{align*}
\bE\qth{|e_{\ell}(v)|^2} &= \bE\bqth{ \Big| \sum_{S\subseteq [n]: |S|=\ell} \prod_{i\in S} v_i \Big|^2 }  = \sum_{\substack{S,S'\subseteq [n]\\ |S| = |S'|=\ell}}  \bE\bqth{ \prod_{i\in S} v_i \prod_{j\in S'} \bar{v}_j} \\
&\stepa{=} \sum_{\substack{S,S'\subseteq [n]\\ |S| = |S'|=\ell}} \sum_{\pi: S \leftrightarrow S'} \prod_{i\in S} \bE\qth{v_i \bar{v}_{\pi(i)}} \\
&\stepb{=} \sum_{\substack{S,S'\subseteq [n]\\ |S| = |S'|=\ell}} \sum_{\pi: S \leftrightarrow S'} \prod_{i\in S} \overline{A}_{i\pi(i)} = \sum_{\substack{S,S'\subseteq [n]\\ |S| = |S'|=\ell}} \Perm(\overline{A}_{S,S'}) \overset{\prettyref{eq:T_ell}}{=} T_{\ell}, 
\end{align*}
where (a) is due to Isserlis' theorem \cite{isserlis1918formula} and the observation $\bE[v_iv_j] = \bE[\bar{v}_i\bar{v}_j]=0$ for all $i,j\in [n]$ (here $\pi: S\leftrightarrow S'$ denotes that $\pi$ is a bijection between $S$ and $S'$), and (b) follows from $v\sim \calCN(0,\overline{A})$. Then the claimed result is a direct consequence of the identity between $S_\ell$ and $T_\ell$ in \Cref{lemma:RST}.

\subsection{Proof of \Cref{lemma:tightness}}
Fix any $n$ and the choices $P_1, \dots, P_n\in \calP$. For a large integer $m$, consider the probability measures $\bP_{mn}$ and $\bQ_{mn}$ based on $mn$ distributions $\sth{P_1,\dots,P_1,P_2,\dots,P_2,\dots,P_n}$, where each $P_i$ appears $m$ times. Let $\overline{P} = \frac{1}{n}\sum_{i=1}^n P_i$, and $f\in L^2(\overline{P})$ be any function with $\bE_{\overline{P}}[f]=0$. We use $f$ as the test function, and study the distributions of the test statistic $\frac{1}{\sqrt{mn}}\sum_{i=1}^{mn} f(X_i)$ under both $\bP_{mn}$ and $\bQ_{mn}$, denoted by $\bP_{f,m}$ and $\bQ_{f,m}$, respectively.

Under $\bP_{mn}$, it is clear that
\begin{align*}
\frac{1}{\sqrt{mn}}\sum_{i=1}^{mn} f(X_i) \overset{\mathrm{d}}{=} \frac{1}{\sqrt{mn}}\sum_{i=1}^n \sum_{j=1}^m f(Z_{i,j}), 
\end{align*}
where $Z_{i,1},\dots,Z_{i,m}\overset{\mathrm{i.i.d.}}{\sim} P_i$ for all $i\in [n]$, and are mutually independent for different $i\in [n]$. By CLT, it is then clear that 
\begin{align*}
\bP_{f,m} \rightsquigarrow \calN\pth{0, \frac{1}{n}\sum_{i=1}^n \var_{P_i}(f) } \qquad \text{as } m \to \infty, 
\end{align*}
where $\rightsquigarrow$ denotes the weak convergence of probability measures. Under $\bQ_{mn}$, it is clear that $X_1,\dots,X_{mn}$ follow an i.i.d. distribution $\overline{P}$, so CLT gives
\begin{align*}
\bQ_{f,m} \rightsquigarrow \calN\pth{0, \var_{\overline{P}}(f) } \qquad \text{as } m \to \infty. 
\end{align*}
Based on the limiting distributions, the $\chi^2$-divergence can be lower bounded as
\begin{align*}
\liminf_{m\to\infty} \chi^2(\bP_{mn} \| \bQ_{mn}) &\stepa{\ge} \liminf_{m\to\infty} \chi^2(\bP_{f,m}\|\bQ_{f,m}) \\
&\stepb{\ge} \chi^2\pth{\calN\pth{0, \frac{1}{n}\sum_{i=1}^n \var_{P_i}(f) }\| \calN\pth{0, \var_{\overline{P}}(f) }} \\
&\stepc{=} \qth{1-\pth{1- \frac{ \frac{1}{n}\sum_{i=1}^n \var_{P_i}(f)}{\var_{\overline{P}}(f)} }^2}^{-\frac{1}{2}}-1, 
\end{align*}
where (a) is due to the data-processing inequality, (b) is the lower-semicontinuity of the $\chi^2$ divergence (using arguments similar to \cite[Theorem 4.9]{polyanskiy2024information}), and (c) uses
\begin{align*}
\chi^2(\calN(0,\sigma_1^2) \| \calN(0,\sigma_2^2)) = \sqrt{\frac{\sigma_2^4}{\sigma_1^2(2\sigma_2^2-\sigma_1^2)}} - 1 = \pth{1 - \pth{1-\frac{\sigma_1^2}{\sigma_2^2}}^2}^{-\frac{1}{2}} - 1
\end{align*}
as long as $\sigma_2\ge \sigma_1$, where we also note that $\var_{\overline{P}}(f)\ge \frac{1}{n}\sum_{i=1}^n \var_{P_i}(f)$ by the concavity of variance. Consequently, it remains to evaluate
\begin{align*}
S := \sup_{f\in L^2(\overline{P}): \bE_{\overline{P}}[f] = 0} \pth{1- \frac{ \frac{1}{n}\sum_{i=1}^n \var_{P_i}(f)}{\var_{\overline{P}}(f)}} = \sup_{f\in L^2(\overline{P}): \bE_{\overline{P}}[f] = 0} \frac{\frac{1}{n}\sum_{i=1}^n\pth{\bE_{P_i}[f]}^2 }{\bE_{\overline{P}}[f^2]}. 
\end{align*}

We show that $S \ge \lambda_2(A)$, the second largest eigenvalue of the matrix $A$ constructed from \eqref{eq:A-matrix}. To this end, we choose
\begin{align*}
    f = \sum_{j=1}^n u_j\frac{\rmd P_j}{\rmd \overline{P}}, 
\end{align*}
with $u=(u_1,\dots,u_n)$ being the unit-length eigenvector of $A$ associated with $\lambda_2(A)$. Clearly $\bE_{\overline{P}}[f]=\sum_{j=1}^n u_j = 0$ by the orthogonality of $u$ and $\bf 1$ in \Cref{lemma:A-property}, and
\begin{align*}
\frac{1}{n}\sum_{i=1}^n\pth{\bE_{P_i}[f]}^2 &= \frac{1}{n}\sum_{i=1}^n\pth{ n \sum_{j=1}^n A_{ij}u_j }^2 = \frac{1}{n} \sum_{i=1}^n (n\lambda_2(A)u_i)^2 = n\lambda_2(A)^2 \\
\bE_{\overline{P}}[f^2] &= n\sum_{i=1}^n\sum_{j=1}^n A_{ij}u_iu_j = n\sum_{i=1}^n \lambda_2(A) u_i^2 = n\lambda_2(A). 
\end{align*}
This shows that $S\ge \lambda_2(A)$.\footnote{In fact $S = \lambda_2(A)$ holds, but we do not need this upper bound.}

A combination of the above steps then gives
\begin{align*}
\liminf_{m\to\infty} \chi^2(\bP_{mn} \| \bQ_{mn}) \ge \frac{1}{\sqrt{1-\lambda_2(A)^2}}-1,
\end{align*}
and taking the supremum over the choice of $n$ and $P_1,\dots,P_n\in \calP$ leads to the first result. For the second result, pick $n=2$, so that
\begin{align*}
    \lambda_2(A) = \trace(A) - 1 = \frac{1}{2}\int \frac{(\rmd P_1 - \rmd P_2)^2}{\rmd P_1 + \rmd P_2} = \LC(P_1, P_2) \ge \frac{H^2(P_1,P_2)}{2}. 
\end{align*}
Here $\LC(P_1,P_2)$ denotes the Le Cam distance, and the last inequality follows from \cite[Eqn. (7.35)]{polyanskiy2024information}. Consequently, the second statement follows from
\begin{align*}
\frac{1}{\sqrt{1-\lambda_2(A)^2}}-1 \ge \frac{\lambda_2(A)^2}{2}, \quad \text{and} \quad \frac{1}{\sqrt{1-\lambda_2(A)^2}}-1 \ge \frac{1}{\sqrt{2(1-\lambda_2(A))}} - 1, 
\end{align*}
and taking the supremum over $P_1,P_2\in \calP$. 

\subsection{Proof of \Cref{lemma:tightness-perm}}
The constants $\sfC_0 \ge 3$ and $\Delta_0 \le \frac{1}{4}$ will be specified later in the proof. Given $\sfC \ge \sfC_0$ and $\Delta \le \Delta_0$, construct a matrix $A\in \bR^{mn\times mn}$ with
\begin{align*}
    A_{ij} = \begin{cases}
        \frac{\Delta}{mn} + \frac{1-\Delta}{n} &\text{if } i,j\in [(k-1)n+1, kn] \text{ for some }k\in [m], \\
        \frac{\Delta}{mn} &\text{otherwise,}
    \end{cases}
\end{align*}
with $m = \lceil \sfC \rceil$, and $n\in \naturals$ to be specified later. In other words, $A$ is the Kronecker product
\begin{align*}
A = \pth{\frac{\Delta}{m}J_m + (1-\Delta)I_m} \otimes \frac{J_n}{n}, 
\end{align*}
where $I$ and $J$ are the identity and all-ones matrices, respectively. By simple algebra, $A$ is doubly stochastic, and has eigenvalues $\lambda_1(A) = 1, \lambda_2(A) = \cdots = \lambda_m(A) = 1-\Delta$, and $\lambda_{m+1}(A) = \cdots = \lambda_{mn}(A) = 0$. In particular, $\trace(A) \le m < 1 + \sfC$, and the spectral gap of $A$ is $\Delta$. 

It remains to lower bound the permanent $\Perm(A)$. Restricting only to the diagonal blocks, one has
\begin{align*}
\Perm(A) \ge (n!)^m \pth{\frac{\Delta}{mn} + \frac{1-\Delta}{n}}^{mn} \ge (n!)^m \pth{\frac{1-\Delta}{n}}^{mn}. 
\end{align*}
By Stirling's approximation $\sqrt{2\pi n}\pth{\frac{n}{e}}^n \le n!\le \sqrt{2\pi n}\pth{\frac{n}{e}}^n \exp\pth{\frac{1}{12n}}$, we have
\begin{align*}
\Perm(A) \ge (2\pi n)^{\frac{m}{2}} \pth{\frac{1-\Delta}{e}}^{mn}\ge \frac{(mn)!}{(mn)^{mn}}\cdot \frac{(2\pi n)^{\frac{m}{2}}}{\sqrt{2\pi mn}}(1-\Delta)^{mn}\exp\pth{-\frac{1}{12mn}}. 
\end{align*}
Now we choose $n = \lceil \frac{1}{2\log\pth{\frac{1}{1-\Delta}}} \rceil \ge 2$, so that
\begin{align*}
\Perm(A) \ge \frac{(mn)!}{(mn)^{mn}}\cdot (2\pi n)^{\frac{m-1}{2}} \frac{(1-\Delta)^{m} e^{-\frac{m}{2}}}{e^{1/12} \sqrt{m} } \stepa{\ge} \frac{(mn)!}{(mn)^{mn}}\cdot \pth{\frac{3}{\Delta}}^{\frac{m-1}{2}} \frac{(1-\Delta)^{m} e^{-\frac{m}{2}}}{e^{1/12} \sqrt{m}}, 
\end{align*}
where (a) uses that $n\sim \frac{1}{2\Delta}$ as $\Delta\to 0$, and thus $2\pi n \ge \frac{3}{\Delta}$ as long as $\Delta\le \Delta_0$ is small enough. Moreover, since the leading exponential term of
\begin{align*}
3^{\frac{m-1}{2}} \frac{(1-\Delta)^{m} e^{-\frac{m}{2}}}{e^{1/12} \sqrt{m}}
\end{align*}
is $\pth{\frac{3}{e}(1-\Delta)^2}^{m/2}$, the above quantity is no smaller than $1$ for $\sfC_0$ large enough and $\Delta_0$ small enough. Consequently, under the above choice of $(\sfC_0,\Delta_0)$, it holds that
\begin{align*}
\Perm(A)\ge \frac{(mn)!}{(mn)^{mn}} \pth{\frac{1}{\Delta}}^{\frac{\sfC}{3}}, 
\end{align*}
i.e. the first statement of the lemma holds with $r = \frac{1}{3}$. 

For the second statement, let $\calP = \{P_1,\dots,P_m\}$ with $m = \lceil \sfC \rceil$, and
\begin{align*}
P_i = \sqrt{\Delta}\delta_0 + (1-\sqrt{\Delta})\delta_i
\end{align*}
Here $\delta_x$ is the Dirac delta measure on the singleton $\{x\}$. It is clear that $\Capa \le |\calP| - 1 < \sfC$ (cf. \Cref{lemma:chi2-capacity}), and $\AffH = \frac{1}{\Delta}$. Now consider $\chi^2(\bP_{mn} \| \bQ_{mn})$, with each $P_i$ appearing $n$ times. By \Cref{lemma:permanent}, 
\begin{align*}
\chi^2(\bP_{mn} \| \bQ_{mn}) = \frac{(mn)^{mn}}{(mn)!} \Perm(A) - 1, 
\end{align*}
where the matrix $A\in \bR^{mn\times mn}$ is given by
\begin{align*}
A = \pth{ \frac{\sqrt{\Delta}}{m} + (1-\sqrt{\Delta})I_m } \otimes \frac{J_n}{n}. 
\end{align*}
By the proof of the first part, as long as $\sfC \ge \sfC_0$ is large enough and $\sqrt{\Delta} \le \sqrt{\Delta_0}$ is small enough, there is a suitable choice of $n\in \mathbb{N}$ such that
\begin{align*}
    \chi^2(\bP_{mn} \| \bQ_{mn}) \ge \pth{\frac{1}{\sqrt{\Delta}}}^{\frac{\sfC}{3}} - 1 = \pth{\frac{1}{\Delta}}^{\frac{\sfC}{6}}-1, 
\end{align*}
establishing the second part with $r' = \frac{1}{6}$.

\subsection{Proof of \Cref{lemma:permutation_channel}}
For the capacity upper bound, let $Q_1,\dots,Q_m$ be an $\varepsilon$-cover of $\calP_n$ under the KL divergence, with a fixed $\varepsilon>0$ and $m = N_{\mathrm{KL}}(\calP_n, \varepsilon)$. By the standard entropic upper bound of the mutual information (cf. \cite{yang1999information} or \cite[Theorem 32.4]{polyanskiy2024information}), one has
\begin{align*}
I(X^n;Y^n)&\le \bE_{X^n}\qth{ \min_{i\in [m]} \KL( P_{Y^n|X^n} \| Q_i^{\otimes n} ) } + \log m\\
&\stepb{=} \bE_{X^n}\bqth{ \KL( P_{Y^n|X^n} \| \overline{P}^{\otimes n} ) + \min_{i\in [m]} \bE_{P_{Y^n|X^n}}\qth{\log \frac{\rmd \overline{P}^{\otimes n}}{\rmd Q_i^{\otimes n}}} } + \log m \\
&\stepc{=} \bE_\theta\qth{ \KL( P_{Y^n|X^n} \| \overline{P}^{\otimes n} ) + \min_{i\in [m]} n\cdot \KL(\overline{P} \| Q_i) } + \log m \\
&\stepd{\le} \Capa(1+\log\AffH) + n\varepsilon^2 + \log m, 
\end{align*}
where $\overline{P}$ denotes the average distribution $\frac{1}{n}\sum_{i=1}^n \sfK_{X_i}$ in (b), (c) uses $\Law(Y_i|X^n) = \overline{P}$ for all $i\in [n]$, and (d) follows from \Cref{thm:main} and the definition of the KL covering applied to $\overline{P}\in \calP_n$. Taking the infimum over $\varepsilon>0$ then completes the proof of the upper bound.

For the capacity lower bound, let $Q_1,\dots,Q_m$ be an $\varepsilon$-packing of $\calP_n$ under the Hellinger distance, with a fixed $\varepsilon>0$ and $m = M_{\mathrm{H}}(\calP_n, \varepsilon)$. By the definition of $\calP_n$, each $Q_i$ can be expressed as
\begin{align*}
Q_i = \frac{1}{n}\sum_{j=1}^n \sfK_{x_{ij}}, \quad \text{for some }x_i^n :=(x_{i1},\dots,x_{in})\in \calX^n. 
\end{align*}
Now let $X^n \sim \Unif(\sth{x_1^n, \dots, x_m^n})$. By \cite[Lemma 3]{haussler1997mutual}, it holds that
\begin{align}\label{eq:HO}
I(X^n; Y^n) \ge -\frac{1}{m}\sum_{i=1}^m \log \bpth{\frac{1}{m}\sum_{j=1}^m \exp\pth{-\frac{1}{2}D_{1/2}( P_{Y^n|X^n=x_i^n}, P_{Y^n|X^n=x_j^n} ) } }, 
\end{align}
where $D_{1/2}(P,Q) = -2\log(1-H^2(P,Q)/2) = -2\log\int \sqrt{\rmd P \rmd Q}$ is the $\frac12$-R\'enyi divergence. To proceed, we prove the following perturbation bound: for any $P', Q'$, it holds that
\begin{align}\label{eq:perturbation}
D_{1/2}(P,Q)\ge \frac{1}{2}\pth{D_{1/2}(P',Q') - \log(\chi^2(P\|P')+1) - \log(\chi^2(Q\|Q')+1)}. 
\end{align}
Indeed, by H\"older's inequality, we have
\begin{align*}
\pth{\int \sqrt{\rmd P\rmd Q}}^4 \le \pth{\int \sqrt{\rmd P' \rmd Q'}}^2 \pth{\int \frac{(\rmd P)^2}{\rmd P'}} \pth{\int \frac{(\rmd Q)^{2}}{\rmd Q'}}, 
\end{align*}
so that taking the logarithm gives \eqref{eq:perturbation}. Next, for $i\neq j$, choosing $P = P_{Y^n|X^n=x_i^n}, Q=P_{Y^n|X^n=x_j^n}, P' = Q_i^{\otimes n}$, and $Q' = Q_j^{\otimes n}$ in \eqref{eq:perturbation} gives
\begin{align*}
D_{1/2}( P_{Y^n|X^n=x_i^n}, P_{Y^n|X^n=x_j^n} ) &\ge \frac{1}{2}\pth{ nD_{1/2}(Q_i, Q_j) - 2\Capa(1+\AffH) } \\
&\ge \frac{n\varepsilon^2}{2} - \Capa(1+\AffH), 
\end{align*}
for $D_{1/2}(Q_i, Q_j)\ge H^2(Q_i, Q_j)\ge \varepsilon^2$. Therefore, \eqref{eq:HO} gives
\begin{align*}
I(X^n; Y^n) &\ge - \log\pth{\exp\pth{-\frac{n\varepsilon^2}{4} + \frac{1}{2}\Capa(1+\AffH)} + \frac{1}{m} } \\
&\ge \min\sth{ \log m, \frac{n\varepsilon^2}{4} - \frac{1}{2}\Capa(1+\AffH) } - \log 2, 
\end{align*}
using $\frac{1}{a}+\frac{1}{b}\le \frac{2}{\min\{a,b\}}$ for $a,b>0$. Taking the supremum over $\varepsilon>0$ proves the lower bound.

\bibliographystyle{alpha}
\bibliography{ref.bib}

\end{document}